\documentclass[12pt, letterpaper]{amsart}
\usepackage[margin=2cm]{geometry}

\usepackage[pdfstartpage=1,colorlinks=true,bookmarks=false,pdfstartview={FitH}]{hyperref}

\usepackage{cleveref}
\usepackage{amsfonts}
\usepackage{amssymb}
\usepackage{amsmath}
\usepackage{amsthm}
\usepackage{mathtools}
\usepackage{latexsym}
\usepackage{color}
\usepackage{comment}
\usepackage{esint}
\usepackage{dsfont}

\newtheorem{thm}{Theorem}[section]
\newtheorem{lemma}[thm]{Lemma}
\newtheorem{lem}[thm]{Lemma}
\newtheorem{prop}[thm]{Proposition}
\newtheorem{cor}[thm]{Corollary}

\newtheorem{observation}[thm]{Observation}

\newtheorem*{thm*}{Theorem}

\theoremstyle{remark}
\newtheorem{remark}[thm]{Remark}

\theoremstyle{definition}
\newtheorem{defi}[thm]{Definition}

\numberwithin{equation}{section}

\newcommand{\R}{\mathbb{R}}

\DeclareMathOperator{\re}{Re}
\DeclareMathOperator{\im}{Im}
\DeclareMathOperator{\Log}{Log}
\DeclareMathOperator{\sign}{sign}

\DeclareMathOperator{\argg}{Arg}
\newcommand{\M}{\mathcal{M}}
\newcommand{\Ds}{(-\Delta)^{s}}

\newcommand{\angles}[1]{\left\langle{#1}\right\rangle}
\newcommand{\floors}[1]{\left\lfloor{#1}\right\rfloor}
\newcommand{\ceils}[1]{\left\lceil{#1}\right\rceil}

\newcommand{\abs}[1]{\left\lvert{#1}\right\rvert}
\newcommand{\norm}[2][]{\left\|{#2}\right\|_{#1}}
\newcommand{\set}[1]{\left\{#1\right\}}

\newcommand{\eps}{\varepsilon}

\newcommand{\p}{\partial}

\newcommand{\bC}{\mathbb{C}}

\newcommand{\bN}{\mathbb{N}}
\newcommand{\bS}{\mathbb{S}}
\newcommand{\bZ}{\mathbb{Z}}

\newcommand{\cC}{\mathcal{C}}

\newcommand{\cF}{\mathcal{F}}
\newcommand{\cG}{\mathcal{G}}
\newcommand{\cH}{\mathcal{H}}
\newcommand{\cL}{\mathcal{L}}
\newcommand{\cM}{\mathcal{M}}
\newcommand{\cU}{\mathcal{U}}

\newcommand{\be}{\begin{equation}}
\newcommand{\ee}{\end{equation}}
\newcommand{\bee}{\begin{equation*}}
\newcommand{\eee}{\end{equation*}}
\newcommand{\bea}{\begin{eqnarray}}
\newcommand{\eea}{\end{eqnarray}}
\newcommand{\bs}{\begin{split}}
\newcommand{\es}{\end{split}}

\DeclareMathOperator{\Hyperg}{\mbox{ }_2 F_{1}}
\DeclareMathOperator{\divergence}{div}

\DeclareMathOperator{\Real}{Re}

\DeclareMathOperator{\Res}{Res}

\begin{document}

\title{Spectral properties of L\'evy Fokker--Planck equations}

\author[H. Chan]{Hardy Chan}

\address{Hardy Chan
\hfill\break\indent
Insituto de Ciencias Matem\'aticas,
\hfill\break\indent 28049 Madrid, Spain}
\email{hardy.chan@icmat.es}

\author[M. Fontelos]{Marco A. Fontelos}

\address{Marco Antonio Fontelos
\hfill\break\indent
Insituto de Ciencias Matem\'aticas,
\hfill\break\indent
C. Nicol\'as Cabrera 13-15,
Campus de Cantoblanco, UAM, 28049 Madrid, SPAIN}
\email{marco.fontelos@icmat.es}

\author[M.d.M. Gonz\'alez]{Mar\'ia del Mar Gonz\'alez}

\address{Mar\'ia del Mar Gonz\'alez
\hfill\break\indent
Universidad Aut\'onoma de Madrid
\hfill\break\indent
Departamento de Matem\'aticas, and ICMAT.  28049 Madrid, Spain}
\email{mariamar.gonzalezn@uam.es}


\maketitle
\begin{abstract}
Hermite polynomials, which are associated to a Gaussian weight and solve the Laplace equation with a drift term of linear growth, are classical in analysis and well-understood via ODE techniques.  Our main contribution is to give explicit Euclidean formulae of the fractional analogue of Hermite polynomials, which appear as eigenfunctions of a L\'evy Fokker--Planck equation.  We will restrict, without loss of generality, to radially symmetric functions.

A crucial tool in our analysis is the Mellin transform, which is essentially the Fourier transform in  logarithmic variable and which turns weighted derivatives into multipliers. This allows to write the weighted space in the fractional case that replaces the usual $L_r^2(\mathbb R^n, e^{|x|^2/4})$.

After proving compactness, we obtain a exhaustive description of the spectrum of the L\'evy Fokker--Planck equation and its dual, the fractional Ornstein--Uhlenbeck problem, which forms a basis thanks to the spectral theorem for self-adjoint operators. As a corollary,
we obtain a full asymptotic expansion for solutions of the fractional heat equation.
\end{abstract}

\section{Introduction and statement of results}

Given $s\in(0,1)$, $n\geq 2s$, we would like to describe the asymptotic behavior of the fractional heat equation
\begin{equation}\label{heat-equation}
\begin{cases}
\varphi_t+(-\Delta)^s \varphi=0,& x\in\mathbb R^n, t>0,\\
\varphi|_{t=0}=\varphi_0,
&x\in \mathbb R^n,
\end{cases}
\end{equation}
in terms of convergence to its fundamental solution.  In this paper we will always restrict to radially symmetric functions, that is, they depend only on the variable $r=|x|$, $x\in\mathbb R^n$.

In the classical case, that is, for $s=1$, it is well known that solutions of the heat equation converge to a Gaussian and there are several proofs of this fact (see the survey \cite{Vazquez1}, for example). The functional analytical approach is to study the heat equation in self-similar variables, which yields a Fokker--Plank equation. A precise description of the spectrum for this Fokker--Planck equation in $L^2$ spaces with Gaussian weights, and its relation to an Ornstein--Uhlenbeck process,  allows to obtain more precise asymptotics in terms of higher order moments.

In the fractional case, while one still obtains convergence to the fundamental solution in relative error (see, for instance, \cite{Vazquez2} and the references therein), the open question of finding the right analytic framework for refining this convergence has persisted  for some time.


In this paper we give an answer to this problem, calculating in Theorem \ref{thm1} the full spectrum of the L\'evy Fokker--Planck operator
\begin{equation}\label{FP-introduction}
L_s u:=(-\Delta)^s u-\frac{1}{2s}x\cdot \nabla u-\frac{n}{2s}u,\quad x\in\mathbb R^n,
\end{equation}
and of its dual, the fractional Ornstein--Uhlenbeck equation, in Theorem \ref{thm2}.  For this, we provide a suitable function space 
where this operator is self-adjoint. As a consequence, we immediately obtain asymptotics of all orders for \eqref{heat-equation}.

Let us state our precise results. 
The fundamental solution for the fractional heat equation is given by
\begin{equation}\label{G}
G_s(t,x)
=\frac{1}{(2\pi)^{\frac{n}{2}}}
\int_{\mathbb R^n} e^{-ix\cdot\xi}e^{-t|\xi|^{2s}}\,d\xi.
\end{equation}
It is well-known (e.g. \cite{BG}) that $G_s$ is  
not of Gaussian type, but a radially decreasing function with polynomial tails:
\[
G_s\left(t,x\right)=\frac{1}{t^{\frac{n}{2s}}}\mathcal G_s\left(\frac{|x|}{t^\frac{1}{2s}}\right)
\asymp
\dfrac{t}{(t^{\frac{1}{s}}+|x|^2)^{\frac{n+2s}{2}}}.
\]
Long-time behavior and asymptotic convergence to (a multiple) of the fundamental solution can be seen by writing  equation \eqref{heat-equation} in self-similar variables

\begin{equation}\label{selfsimilar-variables}
\tilde x=\frac{x}{(t+1)^{\frac{1}{2s}}}, \quad \tilde t=\log (t+1),\quad \varphi(t,x)=\frac{1}{(1+t)^{\frac{n}{2}}}u(\tilde t,\tilde x),
\end{equation}
obtaining
\begin{equation*}\label{self-similar}
u_{\tilde t}+(-\Delta_{\tilde x})^s u-\frac{1}{2s}\tilde x\cdot \nabla_{\tilde x} u-\frac{n}{2s}u=0,\quad \tilde x\in\mathbb R^n,\,\tilde{t}>0,
\end{equation*}
which is the fractional version of the Fokker--Planck equation, usually termed  L\'evy Fokker--Planck. For simplicity, in most of our upcoming discussion, we will drop the tilde in the self-similar variables and just denote them by $(t,x)$. Thus, it is natural to consider the fractional Fokker--Planck operator $L_s$ introduced in \eqref{FP-introduction} which, for  a radially symmetric function $u=u(r)$, is given by
\begin{equation}\label{LFK}
L_s u=(-\Delta)^s u-\frac{1}{2s}r\partial_r u-\frac{n}{2s}u,\quad r>0.
\end{equation}
This operator  contains a fractional diffusion and a first order drift in competition (the zeroth order term does not play a significant role in our discussion, although it simplifies under Fourier transform).

\medskip

Our first Theorem deals with the eigenvalue problem
\begin{equation}\label{eigenvalue-problem}
	L_su	=\nu u,\quad u=u(r),
\end{equation}
which is a fractional version of the Hermite differential equation.

In the local case $s=1$, solutions to \eqref{eigenvalue-problem} in the space $L_r^2(\mathbb R^n, e^{|x|^2/4})$ of radial functions with an inverse Gaussian weight are given in terms of Hermite ($n=1$) or Laguerre ($n\geq 2$) polynomials
-- we refer to Section \ref{subsection:local-Hermite}
 for more details. For general $s\in(0,1)$ we introduce a suitable functional space $\mathds L^2_\dag(s)$ in which the operator $L_s$ still enjoys good compactness properties and one can use the spectral theorem to produce a  (countable) basis of eigenfunctions for \eqref{eigenvalue-problem}. We characterize the full spectrum in Theorem \ref{thm1} below.

The main idea behind   $\mathds L^2_\dag(s)$ is that, changing the radial Fourier variable $\zeta$ to its power
\begin{equation}\label{change}
\vartheta=\zeta^{s}
\end{equation}
brings the non-local operator $L_s$ to the local $L_1$. More precisely, let $u^\dag(\rho)$ to be the inverse Fourier transform of $\widehat u(\zeta)$ in the new variable $\vartheta$, that is,
\begin{equation*}
u^\dag(\rho)
=\cF^{-1}_{\vartheta\to\rho}\left\{\widehat{u}(\vartheta^{1/s})\right\}(\rho).
\end{equation*}
Then, the crucial observation (see \Cref{lem:Ls-L1}) is that the operation $u\mapsto u^\dag$ relates $L_s$ to $L_1$ by 
\begin{equation}\label{eq:intro-Ls-L1}
L_s u=f
\qquad \Longleftrightarrow \qquad
L_1 u^\dag=f^\dag.
\end{equation}
That is, we can use the classical spectral theorem for $u^\dag$ in $L_r^2(\mathbb R^n, e^{|x|^2/4})$ and then translate it back to $u$. As a consequence, it is natural to define the space $\mathds L^2_\dag(s)$ by the scalar product
\begin{equation*}\begin{split}
		\angles{u_1,u_2}_\dag
		&=\int_0^\infty
		u_1^\dag(\rho)
		\overline{u_2^\dag(\rho)}
		e^{\frac{\rho^{2}}{4}}
		\rho^{n-1}
		\,d\rho.
\end{split}\end{equation*}

Nevertheless, we have found that working with Mellin transform, rather than in Fourier variables, yields more information on the fractional case. For this, we first realize that this scalar product may be written as
\begin{equation}\label{space-introduction}
\angles{u_1,u_2}_\dag
=\int_0^\infty
	\Phi_s^{-1}u_1(r)\,
	\overline{\Phi_s^{-1}u_2(r)}
	\,e^{\frac{r^{2s}}{4}}
	r^{n(2-s)-1}
\,dr,
\end{equation}
where $\Phi_s$ is a ``quasi"-isometry  defined through a Mellin multiplier, and described precisely in Section \ref{subsection:multiplier}.    The relevant observation is that the transformation \eqref{change} is encoded in $\Phi_s$. Indeed, we
 will show in Proposition \ref{lemma:equivalence-new} that
\begin{equation*}\label{passage}
\Phi_s^{-1}u(r)=\left (\tfrac{r^{2s}}{4}\right )^{\frac{n}{2}-\frac{n}{2s}}u^\dag(r^s).
\end{equation*}
Even though the Mellin formulation may seem a bit obscure at this point, we will show in Section \ref{subsection:Gevrey} that $\mathds L^2_\dag(s)$ can be interpreted as a fractional version of a Gevrey--Sobolev space. More significantly, this approach allows to prove Theorem \ref{thm1} without relying on previous knowledge of the spectrum in the local case. Instead, we give an independent calculation of \emph{all} solutions of the eigenvalue problem \eqref{eigenvalue-problem}.

Finally, note that when $s\to 1$, we recover the usual $L_r^2(\mathbb R^n, e^{|x|^2/4})$ space with a Gaussian weight. From now on, unless there is a risk of confusion, we will simply write  $\mathds L^2_\dag:=\mathds L^2_\dag(s)$, dropping the dependence on $s$.

\begin{thm}\label{thm1}
For any $s\in(0,1)$, $n\geq 2$, the Hilbert space $\mathds L^2_\dag$ has a complete orthonormal basis
$\{e_k^{(s)}(r)\}_{k\in\mathbb N}$ of eigenfunctions of $L_s$, that is,
\begin{equation*}
L_s e_k^{(s)}=\nu_k e_k^{(s)},\quad \nu_k=k, \quad k\in\mathbb N.
\end{equation*}
The eigenfunctions are given uniquely (up to a scalar multiple) by
	\begin{equation}\label{eigenfunction-introduction}
	e_k^{(s)}(r)=
		\Phi_s 	\left\{\big (
			\tfrac{r^{2s}}{4}
		\big)^{-\frac{n(1-s)}{2s}}
		e^{-\frac{r^{2s}}{4}}
		L_k^{(\frac{n-2}{2})}
		\big(
			\tfrac{r^{2s}}{4}
		\big)\right\},
	\end{equation}
or, alternatively, in terms of special functions by \eqref{eq:expl-eig-func}.

Here $L_k^{(\alpha)}$ is the generalized Laguerre polynomial of degree $k$ with parameter $\alpha$, as defined in \Cref{section:Hermite}.
In the Fourier domain, setting $\zeta=|\xi|$ to be the radial Fourier variable, one has the equivalent formula
\begin{equation}\label{eigen-int}
\widehat{e_k^{(s)}}(\zeta)
=e^{-\zeta^{2s}}
	\zeta^{2sk},
\quad k\in\mathbb N.
\end{equation}
\end{thm}

We remark that even if we have an explicit formula for the eigenfunctions \eqref{eq:expl-eig-func}, it is useful consider different points of view. Working on the Fourier side, we first realize that
the change of variable \eqref{change}
transforms the eigenfunction \eqref{eigen-int} into
$$\widehat{e_k^{(s)}}(\vartheta)=e^{-\vartheta^{2}}\vartheta^{2k},$$
which is the eigenfunction in the local case $s=1$. In particular, for $k=0$, this change brings the non-local heat kernel \eqref{G} to the Gaussian  $G_1(t,r)=\frac{1}{(4\pi t)^{n/2}}e^{-\frac{r^2}{4t}}$.

However, working with the transformation $\Phi_s$ in Mellin variables brings forward the analogy with the local case. Indeed, for any $\nu\in\mathbb R$, any (reasonable) solution of the eigenvalue problem \eqref{eigenvalue-problem} is of the form
$$e^{(s)}_\nu=\Phi_s U_\nu$$
for some $U_\nu$ (given precisely in Theorem \ref{thm:u-nu}) with the asymptotic behavior as $r\to\infty$
\begin{equation}\label{asymptotics-introduction}
\left\{\begin{split}
&U_k(r)
\sim
	\,c_{s,k} e^{-\frac{r^{2s}}{4}}
	\left (
		\tfrac{r^{2s}}{4}
	\right )^{k+\frac{n}{2}-\frac{n}{2s}}
\quad \text{ for } \quad
\nu=k\in\bN,\\
&U_\nu(r)
\asymp
	\left (
		\tfrac{r^{2s}}{4}
	\right )^{-\nu-\frac{n}{2s}}
\quad \text{ for } \quad
\nu\notin \bN.
\end{split}\right.
\end{equation}
The distinctiveness of the case $\nu\in\bN$ is seen in the reduction of Kummer functions to Laguerre polynomials; see \Cref{rmk:kummer}. This motivates the choice of scalar product in \eqref{space-introduction}, since only eigenfunctions for $\nu=k\in\mathbb N$ have exponential decay.
Note also that this the most technical step in the proof of Theorem \ref{thm1} and it uses complex analytic methods.

Our results include a statement regarding uniqueness. Note that, as in the classical case $s=1$,  additional energy bounds are required in lower dimensions. This can be seen, in the local case, from the fact that solutions of an ODE can be characterized in terms of the asymptotic behavior both at $r\to 0$ and $r\to\infty$ (this is a classical result; see, for instance, \cite[Proposition 2.2]{Mizoguchi}). However, in the non-local setting this simple approach is not valid; we use again Mellin transform. Finally, remark that the uniqueness result \Cref{prop:uniq} could become pathological when $n=1$, where we need to impose $\nu\in\bZ$.

\medskip

As a direct application of our arguments we can apply separation of variables to the fractional heat equation \eqref{heat-equation} in the space $\mathds L^2_\dag$ to produce a solution of the equation in terms of the eigenvalue expansion from Theorem \ref{thm1}:

\begin{cor} Let $\varphi_0 \in \mathds L^2_\dagger$ be a radially symmetric function. Then the Cauchy problem \eqref{heat-equation} has a unique (radially symmetric) solution in $\mathds L^2_\dagger$ given by
\begin{equation*}
\varphi(t,r)=\sum_{k=0}^\infty a_k e^{-\nu_k t} e_k^{(s)}(r),
	\qquad
a_k:=\langle \varphi_0, e_k^{(s)}\rangle_\dag,\ \nu_k:=k.
\end{equation*}
\end{cor}

The first consequence of this expansion is that, as $t\to\infty$, the solution $\varphi$ behaves as the self-similar solution $e_0^{(s)}=G_s(1,\cdot)$ since the first eigenvalue is $\nu_0=0$, that is,
\begin{equation*}
\varphi(t,r)\sim \frac{1}{t^{\frac{n}{2s}}}e_0^{(s)}\left(\frac{r}{t^{\frac{1}{2s}}}\right),
\end{equation*}
and similar formulas hold for all higher moments.

\bigskip

The adjoint of $L_s$ in $L^2(\mathbb R^n)$ is the fractional Ornstein--Uhlenbeck operator, given by
\begin{equation*}\label{problem-adjoint-0}
L_s^*w:=(-\Delta)^s w+\frac{1}{2s}x\cdot \nabla w,
\end{equation*}
which in the radially symmetric case reduces to
\begin{equation*}\label{problem-adjoint-operator}
L_s^*w:=(-\Delta)^s w+\frac{1}{2s}r \partial_r w.
\end{equation*}
Note that the $L^*_s$ operator appears naturally when we consider solutions to the heat equation which blows up in finite time. Indeed, if we change variables
\begin{equation*}
\tilde x=\frac{x}{(T-t)^{\frac{1}{2s}}}, \quad \tilde t=-\log (T-t),\quad \varphi(t,x)=w(\tilde t,\tilde x),
\end{equation*}
in equation \eqref{heat-equation} we obtain
\begin{equation*}
w_{\tilde t}
+(-\Delta_{\tilde x})^s w+\frac{1}{2s}\tilde x\cdot \nabla_{\tilde x} w
=0.
\end{equation*}

The $L^2$-dual space of $\mathds L^2_\dag$ will be denoted by $\mathds L^2_\ddag$ and it will be characterized in Section \ref{subsection:dual-space}. It is defined by the scalar product
\begin{equation*}
\begin{split}
	\angles{w_1,w_2}_\ddag
	&=\int_0^\infty
		(\Phi_s^*)^{-1}w_1(r)
\overline{		(\Phi_s^*)^{-1}w_2(r)}
		\,e^{-\frac{r^{2s}}{4}}
		r^{sn-1}
	\,dr
\end{split}
\end{equation*}
for a ``quasi"-isometry $\Phi_s^*$. The results for $\mathds L^2_\ddag$ are parallel to those for  $\mathds L^2_\ddag$ and we refer the reader to Section \ref{section:adjoint} for the precise statements. Nevertheless, in the latter case one has a more user-friendly formula for the scalar product in $\mathds L^2_\ddag$. Indeed, formally
\begin{equation*}
\angles{u_1,u_2}_\ddag=\int_0^\infty G^\ddag_s(r_1,r_2)w_1(r_1)w_2(r_2)\,dr_1 dr_2,
\end{equation*}
for a kernel given in \eqref{formula-G}.

Our second Theorem characterizes the dual basis  to that of Theorem \ref{thm1}:

\begin{thm}\label{thm2}
The Hilbert space $\mathds L^2_\ddag$ has a complete (orthonormal) basis
$\{\omega_k^{(s)}(r)\}_{k\in\mathbb N}$ of eigenfunctions of $L^*_s$, that is,
\begin{equation*}
L^*_s \,\omega_k^{(s)}=\nu_k \omega_k^{(s)},\quad \nu_k=k\in\mathbb N.
\end{equation*}
Here
\begin{equation}\label{eigenfunction-introduction-adjoint}
\omega_k^{(s)}(r)
=\Phi^*_s \left\{\frac{2s}{\Gamma(k+\frac{n}{2})}
L_k^{(\frac{n-2}{2})}(\tfrac{r^{2s}}{4})\right\},\quad k\in\mathbb N,
\end{equation}
up to multiplicative constant.
\end{thm}

Similarly to \eqref{eq:intro-Ls-L1}, we observe (in \Cref{lem:Ls-L1-adj}) that $w\mapsto w^\ddag$ relates $L_s^*$ to $L_1^*$:
\begin{equation}\label{eq:intro-Ls*-L1*}
L_s^* w=g
\qquad \Longleftrightarrow \qquad
L_1 w^\ddag=g^\ddag.
\end{equation}

\medskip

In the classical case, it is known that, if $w$ is a solution to the Ornstein--Uhlenbeck problem
$$L_1^*w=-\cG_1^{-1}\divergence(\cG_1\nabla w)=0,$$ then $u=w\mathcal G_1$ is a solution to the Fokker--Planck equation $$L_1u=-\divergence(\cG_1\nabla(\cG_1^{-1}u))=0$$ (recall that $\cG_1$ is a multiple of $e^{-|x|^2/4}$). For other values of $s$ we have an analogous relation, which can be obtained by combining \eqref{eq:intro-Ls-L1}--\eqref{eq:intro-Ls*-L1*}. Indeed, if $u^\dag=w^\ddag\cG_1$, then
\[
L_s^*w=0
	\quad \Longleftrightarrow
	\quad
L_1^*w^\ddag=0
	\quad \Longleftrightarrow \quad
L_1(w^\ddag \cG_1)=L_1u^\dag=0
	\quad \Longleftrightarrow \quad
L_su=0.
\]

\medskip

Let us comment on some relevant bibliography. Related papers on different aspects of the fractional heat equation are \cite{Bonforte-Sire-Vazquez,BG,Metzler-Klafter,Valdinoci,Barrios-Peral-Soria-Valdinoci}, although by no means is this list exhaustive. From the probabilistic side, L\'evy Fokker--Planck equations describe the evolution of the probability distribution of a nonlinear stochastic differential equation driven by
non-Gaussian L\'evy stable noises (see \cite{SLDYL} and the references therein). Convergence to equilibrium was considered in the papers \cite{Biler-Karch,Gentil-Imbert} using the method of entropies. Other convergence results are: \cite{Tristani}, for convergence in spaces with power weights;  \cite{Mischler-Tristani}, where they use the semigroup approach; and \cite{Toscani} on the fractional Fisher information.

For other drift terms, we cite \cite{Bogdan-Jakubowski}, where they give heat kernel estimates for first order terms in the Kato class, and  \cite{Epstein-Pop} where they show regularity using pseudodifferential analysis when the drift coefficient is bounded.

The approach of finding eigenfunctions in Fourier variables by the change \eqref{change} was known in the Physics community (see  \cite{TSP1,TSP2} and \cite{Janakiraman-Sebastian,Janakiraman-Sebastian:erratum}). However, their results are formal and do not consider functional analytic aspects.

Note that our eigenfunction \eqref{eigenfunction-introduction} is written in terms of Laguerre polynomials. Such orthogonal polynomials have been used in many contexts, for instance, in the paper  \cite{Mazzitelli-Stinga-Torrea} where they study fractional powers of first order differential operators such as $\left(\pm\frac{d}{dx}+x\right)$, obtaining a generalized  Rodrigues formula.

The higher order local problem was considered in \cite{EGKP}, where the diffusion term is given by $\Delta^{m}$, $m\in \mathbb N$.  Fractional powers of the  Ornstein--Uhlenbeck operator, that is, $(-\Delta+2x\cdot\nabla+n)^s$ were studied in \cite{Ganguly-Manna-Thangavelu}.

Let us also mention that in this paper we are restricting to the simplest linear case. However, a Mellin transform approach may be also applicable in the presence of Hardy type terms, that is, a potential term of the form $\frac{u}{r^{2s}}$, since this term has the same homogeneity as $(-\Delta)^s$, although explicit formulas will not be available anymore. In addition, our approach could be extended to non-radially symmetric solutions by projecting over spherical harmonics (the necessary modifications are only needed in Proposition \ref{prop:M-Ds}, and these can be found in \cite{DelaTorre-Gonzalez}.) Note also that there are many possible non-linear generalizations. We hope to return to these problems elsewhere.

\medskip

This paper is structured as follows:  Section \ref{section:motivation} is included as a motivation of our results but not strictly needed for the proofs in the rest of the paper; however, it provides a good intuition of the problem. Section  \ref{section:transforms} contains some standard background on the Mellin transform in order to make this paper self-contained. In Section \ref{section:functional-new} we characterize the function space $\mathds L^2_\dag(s)$ and show compactness.In Section \ref{section:eigenvalue} we prove Theorem \ref{thm1}, finding precise formulas for the eigenfunctions. The proof of Theorem \ref{thm2} on the adjoint problem is contained in Section \ref{section:adjoint}. Finally, the Appendix summarizes some well known results on orthogonal polynomials.

\medskip

Some comments on notation:
\begin{itemize}
\item Here $f \sim g$ means that $\lim\limits_{r\to+\infty} \frac{f}{g}=1$, while $f\asymp g$ means that $C^{-1}g \leq f \leq Cg$ for a positive constant $C$ that may depend on $s$ and $n$.

\item The notation $\mathbb N$ for natural numbers also includes the zero.

\item The integer part (floor) of a real number will be denoted by $\floors{\cdot}$.
\end{itemize}

\section{Motivation}\label{section:motivation}

This Section is not necessary in the proofs and the expert reader may skip it. However, we decided to include a preliminary discussion in order to illustrate our approach in a simple setting. Our objective is to find a precise asymptotic expansion as $t\to +\infty$ for the solution to the heat equation \eqref{heat-equation}, and to relate it to an eigenfunction  decomposition. Recall that we are only considering radially symmetric functions. The radial coordinate is denoted by $r=|x|$.

\subsection{A review of the local case $s=1$}\label{subsection:local-Hermite}
Calculations are very transparent in the case $n=1$, $s=1$, so we describe first the classical approach to the  heat equation (see \cite{Vazquez1} for a nice presentation on the subject).  The asymptotic behavior of a solution of the heat equation
\begin{equation}\label{heat-equation1}
\begin{cases}
\varphi_t+(-\Delta) \varphi=0,\quad r>0, t>0,\\
\varphi|_{t=0}=\varphi_0,
\end{cases}
\end{equation}
is given by the self-similar solution
\begin{equation}\label{Green1}
G_1(t,r)=\frac{1}{t^{\frac{1}{2}}}\frac{1}{(4\pi)^{\frac{1}{2}}}\mathcal G_1\left(\frac{r}{t^{\frac{1}{2}}}\right),\quad \text{for} \quad \mathcal G_1(r)=e^{-\frac{r^2}{4}}.
\end{equation}
Indeed, as $t\to +\infty$,
\begin{equation*}
\varphi (r,t)\sim M G_1(r,t),
\end{equation*}
where $M$ is the mass of the initial data $M=2\int_{0}^\infty \varphi_0(r)\,dr$. More precise asymptotics can be obtained by expanding via higher order moments, and we quote here the following formula  \cite[Theorem 4]{Duoandikoetxea-Zuazua} (based on the results in \cite{Escobedo-Zuazua}), which gives a precise  asymptotic expansion for $\varphi$,
\begin{equation*}
\varphi(t,r)=\sum_{k=0}^\infty a_k \frac{\partial^k}{\partial r^k} G_1(r,t),
\end{equation*}
where the constants $a_k$, $k=0,1,\ldots$ depend on the initial data. Recalling \eqref{Green1}, we can rewrite this expansion as
\begin{equation}\label{asymptotic-varphi}
\varphi(t,r)=\frac{1}{(4\pi)^{\frac{1}{2}}}\sum_{k=0}^\infty \frac{1}{t^{\frac{1}{2}+\frac{k}{2}}}a_k \frac{d^k}{dr^k} \mathcal G_1\Big(\frac{r}{t^{\frac{1}{2}}}\Big).
\end{equation}
Another way to understand this formula is to use Rodrigues' formula for the Hermite polynomials $\mathcal H_k(r)$, which  is
\begin{equation}\label{Rodrigues0}
e^{-\frac{r^2}{4}}\mathcal H_k(r)=(-1)^k2^k \frac{d^k}{dr^k}\mathcal G_1(r).
\end{equation}
A quick review of Hermite and Laguerre polynomials is given in Appendix \ref{section:Hermite}.

We define the Hilbert space $\mathds L^2_\dag(1):=L^2_r(\mathbb R,e^{\frac{r^2}{4}})$ as the set of even functions on $\mathbb R $ with the scalar product
\begin{equation*}
\langle u_1,u_2\rangle_\dag=\int_{0}^\infty u_1(r)\overline{u_2(r)}\,e^{\frac{r^2}{4}}\,dr
\end{equation*}
and the weighted $L^2$-norm
\begin{equation*}
\|u\|^2_\dag=\int_{0}^\infty |u(r)|^2\,e^{\frac{r^2}{4}}\,dr.
\end{equation*}
Note that if we are looking for even functions, we need to restrict to even values of $k$ and thus we define
\begin{equation*}
e_k(r):=e^{-\frac{r^2}{4}}\mathcal H_{2k}(r).
\end{equation*}
An important property of the set $\{e_k(r)\}_{k=1}^\infty$ is that they form a complete set of eigenfunctions for Fokker--Planck operator
\begin{equation}\label{eigenvalue-L1}
L_1u:=-u''-\frac{1}{2}r u'-\frac{1}{2}u=\nu u,\quad u=u(r),
\end{equation}
with eigenvalues $\nu_k=k$, $k=0,1,2,\ldots$ (see Appendix \ref{section:Hermite}, taking into account the shift of notation $\nu=\frac{\lambda}{2}$, and the fact that we are restricting to even functions).

Now, the key observation is that  \eqref{asymptotic-varphi} is just a separation of variables formula for the solution of \eqref{heat-equation1} in terms of the eigenfunctions of the Fokker--Planck operator $L_1$, that is,
\begin{equation}
u(t,r)=\sum_{k=0}^\infty \alpha_k e^{-kt} e_k(r),
\end{equation}
where the $\alpha_k$ depend on the initial data. We recover \eqref{asymptotic-varphi} by going back to the original variables \eqref{selfsimilar-variables}, up to the usual shifts
 $t\leftrightarrow t+1$ and $k\leftrightarrow 2k$.

As a side comment, recall that Hermite polynomials are special cases of Laguerre polynomials and Kummer's functions (see \eqref{Hermite-Laguerre} and \eqref{Laguerre}), more precisely,
\begin{equation*}
\begin{split}
\mathcal H_{2k}(r)&=H_{2k}\big( \tfrac{r}{2}\big)
=(-1)^{k}2^{2k}k!L_{k}^{(-1/2)}\big(\tfrac{r^{2}}{4}\big)
=(-1)^{k}2^{2k}\frac{\Gamma (k+1/2)}{\Gamma
(1/2)}M\big({-k},\tfrac{1}{2};\tfrac{r^{2}}{4}\big).
\end{split}
\end{equation*}
This relation will be useful to understand the  higher dimensional generalization.

Next, one may also look at the Ornstein--Uhlenbeck version, which is a crucial ingredient in the study of  convergence in relative error $\varphi/G_1$.
The $L^2$-adjoint of the operator $L_1$ is given by
\begin{equation*}
L_1^*w :=-w''+\frac{1}{2}rw'=\mu w, \quad w=w(r),
\end{equation*}
defined on the dual space of $\mathds L^2_\dag(1)$, which is denoted by $\mathds L^2_\ddag(1):=L^2_r(\mathbb R,e^{-\frac{r^2}{4}})$, with the scalar product
\begin{equation*}
\langle w_1,w_2\rangle_\ddag=\int_{\mathbb R}w_1(r)\overline{w_2(r)}\,e^{-\frac{r^2}{4}}\,dr
\end{equation*}
and the norm
\begin{equation*}
\|w\|^2_\ddag=\int_{0}^\infty |w(r)|^2\,e^{-\frac{r^2}{4}}\,dr.
\end{equation*}
It is well known that the set $\{\mathcal H_{2k}(r)\}_{k=0}^\infty$ forms a complete set of eigenfunctions for $L^*_1$ in the space $\mathds L^2_\ddag(1)$ with eigenvalues $\mu_k=0,1,\ldots$ and, in addition, they form an orthogonal basis. This last property  can be checked directly using the properties of the Hermite polynomials. \\

From another point of view, let us look at the above calculations in  Fourier variables.
Take Fourier transform of  \eqref{heat-equation1} in $\mathbb R$ and $\zeta=|\xi|$, the radial Fourier variable. Then, for the problem
\begin{equation*}
\begin{cases}
\widehat \varphi_t+|\xi|^2\widehat \varphi =0,\quad \xi\in\mathbb R,t>0, \\
\widehat\varphi|_{t=0}=\widehat\varphi_0,
\end{cases}
\end{equation*}
it is well known that the solution is given by
\begin{equation*}
\widehat \varphi(t,\xi)=e^{-|\xi|^{2}t}\widehat\varphi_0(\xi)
\end{equation*}
so that
\begin{equation}\label{varphi}
\begin{split}
\varphi(t,x)&=\frac{1}{(2\pi)^{\frac{1}{2}}}\int_{\mathbb R} e^{i\xi x}e^{-|\xi|^{2}t}\widehat \varphi_0(\xi)\,d\xi=\frac{1}{(2\pi)^{\frac{1}{2}}}\int_0^\infty 2\cos(x\zeta)e^{-\zeta^{2}t}\widehat \varphi_0(\zeta)\,d\zeta.
\end{split}
\end{equation}
 On the Fourier side, since $\varphi_0$ is an even function,
\begin{equation*}
\widehat \varphi_0(\zeta)
=\sum_{k=0}^\infty   \alpha_k
\zeta^{2k} e^{-\zeta^2}.
\end{equation*}
Formally substituting into \eqref{varphi} we obtain
\begin{equation*}
\varphi(t,r)=\frac{1}{(2\pi)^{\frac{1}{2}}}\sum_{k=0}^\infty \alpha_k
\int_0^\infty \cos(\zeta r)e^{-\zeta^{2}(t+1)}\zeta^{2k}\,d\zeta.
\end{equation*}
The change of variable $\zeta \mapsto \zeta (t+1)^{\frac{1}{2}}$ gives
\begin{equation}\label{formula-dimension1}
\begin{split}
\varphi(t,r)
&=\frac{1}{(2\pi)^{\frac{1}{2}}}\sum_{k=0}^\infty  \alpha_k \frac{
1}{(t+1)^{\frac{1}{2}+\frac{2k}{2}}}\int_0^\infty \cos\Big(\zeta \frac{x}{(t+1)^{1/2}}\Big)e^{-\zeta^{2}}\zeta^{2k}\,d\zeta,\\
\end{split}
\end{equation}
which is the Fourier version of the asymptotic expansion \eqref{asymptotic-varphi} for an even function $\varphi(t,\cdot)$.
To see this, note that an important property of Hermite polynomials is that the
Fourier transform of $e_{k}(r)$  is
another polynomial of degree $2k$ up to a complex multiplicative constant, more precisely,
\begin{equation*}\label{Fourier-ek}
\widehat e_{k}(\zeta):=\cF(e_{k})(\zeta)=
\zeta^{2k}e^{-\zeta^{2}},
\end{equation*}
 which is just the Fourier representation of Rodrigues' formula \eqref{Rodrigues0}.
This implies that, for the Fourier counterpart of \eqref{eigenvalue-L1}, given by
\[
\zeta^{2}\widehat u(\zeta)+\frac{\zeta}{2}\partial_\zeta\widehat u(\zeta)=\nu\widehat u(\zeta),\quad \widehat u=\widehat u(\zeta),
\]%
the set $\{ \widehat e_k(\zeta)\}_{k=0}^\infty$ is an orthogonal, complete set of eigenfunctions with eigenvalues $\nu _{k}=0,1,2,....$.

\subsection{The local case $s=1$ in higher dimensions $n>1$}\label{subsection:local-n}

Let us go now to general $n$ and consider the eigenvalue equation
\begin{equation}\label{eq:70}
L_1u:=-\Delta u
-\frac{x}{2}\cdot \nabla u
-\frac{n}{2}u
=\nu u,\quad x\in\mathbb R^n.
\end{equation}%
Our scheme is to take Fourier transform of the eigenvalue equation. Assume that $u$ is radially symmetric, so that $\widehat u$ only depends on the variable $|\xi|$. For simplicity, we will denote $\zeta=|\xi|$
which, for radially symmetric functions leads to the Fourier transform
version of \eqref{eq:70},
\begin{equation*}\label{fourier-ek-n-dim}
\zeta^{2}\widehat u(\zeta)
+\frac{\zeta}{2}\partial_\zeta\widehat u(\zeta)
=\nu \widehat u(\zeta),\quad \zeta>0,
\end{equation*}
which is independent of the dimension $n$. Hence, as above, eigenvalues are also $\nu_k=0,1,2,\dots$, with eigenfunctions
\begin{equation}\label{eq:80}
\widehat e_k(\zeta)=\zeta^{2k}e^{-\zeta^{2}}
\end{equation}
again. Note, also, that we are only interested in the even powers.

Let us find an explicit expression for $e_{k}$, taking the inverse Fourier transform in $\mathbb R^n$ (which is the Hankel transform given in  \eqref{Hankel}), and use the formula \eqref{eq:J-exp-2}:
\begin{equation}\label{eigenfunctions-adjoint}
e_{k}(r)=r^{\frac{2-n}{2}%
}\int_{0}^{\infty }\zeta^{2k+\frac{n}{2}}e^{-\zeta^{2}%
}J_{\frac{n-2}{2}}(\zeta r)\,d\zeta=2^{-n/2}k!L_{k}^{(\frac{n-2}{2})}\big(\tfrac{r^{2}}{4}\big)e^{-\frac{r^{2}}{4}}.
\end{equation}
One can also use the relation to Kummer's function from \eqref{Laguerre} to write
\begin{equation*}
e_{k}(r)
=2^{-\frac{n}{2}}e^{-\frac{r^2}{4}}
\frac{\Gamma (k+\frac{n}{2})}{\Gamma (\frac{n}{2})}M\big(-k,\tfrac{n}{2};\tfrac{r^2}{4}\big).
\end{equation*}
The crucial observation here is that the Laguerre polynomials $\{L_{k}^{(\alpha )}(r)\}_{k=0}^\infty$ form an orthogonal basis of $L^2(\mathbb R_+)$ with the appropriate weight. Indeed, recalling the product formula from \eqref{Laguerre-orthogonality}, we have
\[
2^{-(n-1)}\int_{0}^{\infty
}L_{k}^{(\frac{n-2}{2})}\big(\tfrac{r^{2}}{4}\big)L_{m}^{(\frac{n-2}{2})}\big(\tfrac{r^{2}}{4}\big)e^{-\frac{r^{2}}{4}}r^{n-1}\,dr=\frac{%
\Gamma (k+\frac{n}{2})}{\Gamma (k+1)\Gamma (\frac{n}{2})}\delta _{m,k}.
\]%
The natural space  for the problem is thus $\mathds L^2_\dag(1):= L^2_r(\mathbb R^n,e^{\frac{r^{2}}{4}})$ with the scalar product
\[
\langle u_1,u_2\rangle_\dag=\int_0^\infty u_1(r)u_2(r)e^{\frac{r^{2}}{4}}r^{n-1}\,dr.
\]%

Now  \eqref{varphi} is replaced by the Hankel transform formula \eqref{Hankel} in order to obtain the precise asymptotic expansion for the solution of the heat equation $\varphi$ in terms of our eigenbasis $\{e_k(r)\}_k$,
\begin{equation*}
\begin{split}
\varphi(t,x)=&\frac{1}{(2\pi)^{\frac{n}{2}}}\int_{\mathbb R^n} e^{i\xi \cdot x}e^{-|\xi|^{2}t}\widehat \varphi_0(\xi)\,d\xi=r^{\frac{2-n}{2}}\int_0^\infty J_{\frac{n-2}{2}}(\zeta r)\zeta^{\frac{n-2}{2}}e^{-\zeta^{2}t}\widehat \varphi_0(\zeta)\zeta\,d\zeta\\
=&r^{\frac{2-n}{2}}\sum_{k=0}^\infty \alpha_k\int_0^\infty J_{\frac{n-2}{2}}(\zeta r)\zeta^{\frac{n-2}{2}}e^{-\zeta^{2}(t+1)}\zeta^{2k}\zeta\,d\zeta.
\end{split}
\end{equation*}

\subsection{The fractional setting}\label{subsection:fractional-introduction}

The approach described above is suited to be generalized to the fractional case, for any $s\in(0,1)$. Taking Fourier transform of the heat equation \eqref{heat-equation} we obtain
\begin{equation*}
\begin{cases}
\widehat \varphi_t+|\xi|^{2s}\widehat \varphi =0,\quad \xi\in\mathbb R,\,t>0, \\
\widehat\varphi|_{t=0}=\widehat\varphi_0.
\end{cases}
\end{equation*}
Similarly, for even functions, setting $\zeta=|\xi|$,
\begin{equation}\label{varphi-s}
\begin{split}
\varphi(t,x)&=\frac{1}{(2\pi)^{\frac{n}{2}}}\int_{\mathbb R} e^{i\xi \cdot x}e^{-|\xi|^{2s}t}\widehat \varphi_0(\xi)\,d\xi
=
r^{\frac{2-n}{2}}\int_0^\infty J_{\frac{n-2}{2}}(\zeta r)\zeta^{\frac{n-2}{2}}e^{-\zeta^{2s}t}\widehat \varphi_0(\zeta)\zeta\,d\zeta\\
&=\frac{1}{s}\,r^{\frac{2-n}{2}}\int_0^\infty J_{\frac{n-2}{2}}(\vartheta^{1/s} r)\vartheta^{\frac{n}{2s}}e^{-\vartheta^{2}t}\widehat \varphi_0(\vartheta^{1/s})
\vartheta^{1/s-1}
\,d\vartheta
\end{split}
\end{equation}
where, in the last step, we have made the change of variable $\vartheta=\zeta^{s}$. Let $\widehat \varphi_0(\vartheta)$ be the Fourier transform of $\varphi_0$ and denote by $\varphi_0^\dag$  the inverse Fourier transform of the function $\widehat \varphi_0(\vartheta^{1/s})$. If $\varphi_0^\dag$ belongs to the space $\mathds L^2_\dag(1)$, then we can expand it in eigenbasis as
\begin{equation*}
\varphi_0^\dag(r)=\sum_{k=0}^\infty  \alpha_k e_k(r),
\end{equation*}
as explained in Subsection \ref{subsection:local-n}. In Fourier variables, taking into account \eqref{eq:80},
\begin{equation*}
\widehat\varphi_0(\zeta)=\widehat \varphi_0(\vartheta^{1/s})=\sum_{k=0}^\infty \alpha_k \vartheta^{2k} e^{-\vartheta^{2}}=\sum_{k=0}^\infty \alpha_k \zeta^{2ks} e^{-\zeta^{2s}}.
\end{equation*}
Substituting in \eqref{varphi-s}, we have that
\begin{equation*}\label{varphi-s-2}
\begin{split}
\varphi(t,r)&=r^{\frac{2-n}{2}}\sum_{k=0}^\infty \alpha_k \int_0^\infty J_{\frac{n-2}{2}}(\zeta r)\zeta^{\frac{n}{2}}e^{-\zeta^{2s}(t+1)}\zeta^{2ks} \,d\zeta\\
&=r^{\frac{2-n}{2}}\sum_k \alpha_k \frac{1}{(t+1)^{k+\frac{n+2}{4s}}}\int_0^\infty J_{\frac{n-2}{2}}\Big(\zeta\frac{r}{(t+1)^{\frac{1}{2s}}} \Big)\zeta^{\frac{n}{2}}e^{-\zeta^{2s}}\zeta^{2ks} \,d\zeta,
\end{split}
\end{equation*}
which is the non-local generalization of the asymptotic expansion \eqref{asymptotic-varphi} in Fourier variables. Remark that this formula reduces to \eqref{formula-dimension1} when $n=1$ and $s=1$  by taking into account that
\begin{equation*}
\lim_{\alpha\to -1/2} J_\alpha(z)=\sqrt{\frac{2}{\pi z}}\cos z.
\end{equation*}

The underlying idea in many the proofs in this paper is that the change of variable $\vartheta=\zeta^{s}$ reduces a non-local problem to a local eigenvalue expansion. Thus,  defining $e_k^{(s)}(\zeta)$ by
\begin{equation*}
\widehat e_k^{(s)}(\zeta)=\zeta^{2ks} e^{-\zeta^{2s}}, \quad k=0,1,...
\end{equation*}
then the $\{e_k^{(s)}(r)\}_{k=0}^\infty$  are  eigenfunctions for the L\'evy Fokker-Planck operator \eqref{LFK}. Indeed,  for any $\nu\in\mathbb R$ the functions
\begin{equation}\label{hat-e-nu}
\widehat e_{\nu}^{(s)}(\zeta)=\zeta^{2s\nu }e^{-\zeta^{2s}}
\end{equation}
solve the eigenvalue equation \eqref{LFK} in  Fourier variables, that is,
\begin{equation*}
\zeta^{2s}\widehat u(\zeta)
+\frac{1}{2s}\zeta\partial_\zeta\widehat u(\zeta)
=\nu\widehat u(\zeta).
\end{equation*}
 Unfortunately, this Fourier approach gives no information on which are admissible eigenfunctions or what is the right function space in the fractional case. The aim of our work is to show that the $\{e_k^{(s)}(r)\}_{k=0}^\infty$ form an orthonormal basis for a Hilbert space that we  denote by $\mathds L^2_\dagger(s)$. We will see that Mellin transform, rather than Fourier transform,  gives us a precise answer to this question, which is the main ingredient in the proof of Theorem \ref{thm1}.

\section{Fourier, Hankel and Mellin transforms}\label{section:transforms}

\subsection{Fourier and Hankel}

The Fourier transform on $\mathbb R$ will be
denoted by $\mathcal F_1$, with normalization constant $\sqrt{2\pi}$:
\begin{equation}\label{eq:F1}
\cF_1 u(\xi)
=\cF_1\set{u(x)}(\xi)
:=\frac{1}{\sqrt{2\pi}}
\int_{\R}
	e^{-x \xi i}
	u(x)
\,dx,\quad \xi\in\mathbb R.
\end{equation}
Similarly, the Fourier transform on $\mathbb R^n$ is defined to be
\begin{equation*}
\widehat u(\xi)=\cF u(\xi)
=\frac{1}{(2\pi)^{\frac{n}{2}}}
\int_{\R^n}
	e^{-x\cdot\xi i}
	u(x)
\,dx,\quad\xi\in\mathbb R^n.
\end{equation*}
If $u$ is a radially symmetric function in $\mathbb R^n$, then its Fourier transform is radially symmetric, and it
can be written in terms of the Hankel transform as
\begin{equation}\label{Hankel}
	\begin{split}
		&\mathcal Fu(\zeta)
		=(2\pi )^{-\frac{n}{2}}
		\int_{\mathbb R^n} e^{i\xi\cdot x}u(r)\,dx
		=\zeta^{\frac{2-n}{2}}\int_0^\infty J_{\frac{n-2}{2}}(\zeta r)r^{\frac{n-2}{2}}u(r)r\,dr,\\
		&\mathcal F^{-1}w(r)
		=(2\pi )^{-\frac{n}{2}}\int_{\mathbb R^n} e^{-i\xi\cdot x}w(\zeta)\,dx
		=r^{\frac{2-n}{2}}\int_0^\infty J_{\frac{n-2}{2}}(\zeta r)\zeta^{\frac{n-2}{2}}w(\zeta)\zeta\,d\zeta,
	\end{split}
\end{equation}
where we have denoted $\zeta=|\xi|$ for $\xi\in\mathbb R^n$, $r=|x|$ for $x\in\mathbb R^n$.

\subsection{Mellin transform}

Here we review the properties of the classical Mellin transform (see, for instance, the book \cite{book:Mellin}). For a function $u=u(r)$ defined on the half-line $\R_+=(0,\infty)$, its Mellin transform, whenever defined, is given by the integral\footnote{The classical notation for the  Mellin variable is $s=\sigma+it \in \bC$, as widely used in analytic number theory (note in particular that the Riemann zeta function is expressible in terms of the Mellin transform of $(e^r-1)^{-1}$). In this work we reserve the letter $s$ as the power of the fractional Laplacian, and $t$ as time variable.}
\begin{equation*}
	\widetilde{u}(z)
	=\cM u(z)
	=\cM\set{u(r)}(z)
	:=
	\int_0^\infty
	r^{z} u(r)
	\,\frac{dr}{r}.
\end{equation*}
This integral is absolutely convergent on each line $\re z=\sigma$ for which $u\in L^1(\R_+,r^{\sigma-1})$.
In \Cref{sec:M-inv} we briefly discuss the domain and range of the Mellin transform. Under appropriate mild conditions (see \Cref{thm:M-inv}), the original function is recoverable by its (analytic) Mellin transform along a complex line parallel to the imaginary axis.

Note that if $u$ is a smooth function with compact support outside the origin, then its Mellin transform is analytic on $\mathbb C$. If $u(r)$ decays exponentially at infinity but grows like $r^{-\sigma_0}$ for some $\sigma_0\in \mathbb R$ as $r\to 0$, then $\mathcal M u(z)$ is a holomorphic function on the right half-plane $\Real z > \sigma_0$. Similarly, if $u$ decays fast as $r\to 0$ but grows as $r^{-\sigma_1}$ at infinity for some $\sigma_1\in\mathbb R$, then $\mathcal M u(z)$  is holomorphic on the left half-plane
$\Real z < \sigma_1$. If $u$ has such polynomial growth at both ends, then it is only holomorphic on the strip $\sigma_0<\Real z<\sigma_1$. In addition, there are possible meromorphic extensions to a larger set \cite{Zagier} and generalized Mellin transforms. However, we will not need those.


\subsubsection{Mellin vs. Fourier}

\label{sec:M-vs-F}

The Mellin transform on $\re z=0$ is exactly the Fourier transform in the Emden--Fowler variable $\varrho=-\log r$, for $z=\lambda i$ ($\lambda\in\R$). More precisely,
\begin{equation}\label{eq:M-F-1}
	\cM u(\lambda i)
	=\int_{0}^{\infty}
	r^{\lambda i}
	u(r)
	\,\frac{dr}{r}
	=\int_{\R}
	e^{-\varrho\lambda i }
	u(e^{-\varrho})
	\,d\varrho
	=\sqrt{2\pi}
	\cF_1\set{
		u(e^{-\varrho})
	}(\lambda),
\end{equation}
where $\mathcal{F}_1$ denotes the one-dimensional Fourier transform given in \eqref{eq:F1}.

More generally, the Mellin transform of $u(r)$ on $\re z=\sigma$ is the Fourier transform of the  conjugation $r^{\sigma}u(r)$ in the logarithmic variable: for $z=\sigma+\lambda i$ ($\sigma,\lambda\in\R$) and $r=e^{-\varrho}$,
\begin{equation}\label{eq:M-F-2}
	\cM u(\sigma+\lambda i)
	=\int_{0}^{\infty}
	r^{\sigma+\lambda i}
	u(r)
	\,\frac{dr}{r}
	=\int_{\R}
	e^{-\varrho\lambda i}
	e^{-\sigma \varrho}
	u(e^{-\varrho})
	\,d\varrho
	=\sqrt{2\pi}
	\cF_1\set{
		e^{-\sigma \varrho}
		u(e^{-\varrho})
	}(\lambda).
\end{equation}
We are primarily interested in the case $\sigma=\frac{n}{2}$, since the factor of isometry (see \eqref{eq:M-isom-1} below) equals $r^{n-1}$, the Jacobian in polar coordinates.

By \eqref{eq:M-F-1}--\eqref{eq:M-F-2}, the Mellin transform along each vertical line can be simply viewed as the Fourier transform up to a change of variable. Consequently, they enjoy similar elementary properties, convolution identities, inverse transforms and Plancherel theorems. Specifically, we emphasize that the differential operator $-r\p_r=\p_\varrho$ Mellin-transforms to multiplication by $z$, demonstrating the suitability convenience of the Mellin transform in studying the Hermite differential equation.

\subsubsection{Inverse transform and uniqueness}
\label{sec:M-inv}

\begin{thm}[Mellin inversion]
\label{thm:M-inv}
	Suppose that
	\begin{itemize}
		\item $\cM u(z)$ is analytic in some infinite vertical strip $a< \re z <b$, and
		\item $|\cM u(z)| \leq C|z|^{-2}$ in the strip $a< {\rm Re\,} z<b$.
	\end{itemize}
	Then for any $\sigma\in(a,b)$,
	\begin{equation}\label{eq:M-inv}
		u(r)
		=
		\dfrac{1}{2\pi i}
		\int_{\sigma-i\infty}^{\sigma+i\infty}
		r^{-z}
		\cM u(z)
		\,dz
	=\dfrac{1}{2\pi}
	\int_{-\infty}^{\infty}
	r^{-\sigma-\lambda i}
	\cM u(\sigma+\lambda i)
	\,d\lambda.
	\end{equation}
In other words, the inverse Mellin transform of $\widetilde{u}(z)$ along $\re z=\sigma \in (a,b)$ is given by
\begin{equation*}\label{eq:M-inv-2}
\cM_\sigma^{-1} \widetilde{u}(r)
=		\dfrac{1}{2\pi i}
\int_{\sigma-i\infty}^{\sigma+i\infty}
r^{-z}
\widetilde{u}(z)
\,dz.
\end{equation*}
\end{thm}

%

\begin{cor}[Uniqueness]
	\label{cor:uniq}
Suppose $\cM u_1$, $\cM u_2$ are analytic in the \emph{same} strip $a<\re z<b$ on which $|\cM u_1(z)|+|\cM u_2(z)| \leq C|z|^{-2}$. If
\[
\cM u_1(z)=\cM u_2(z)
	\quad \text{ for } \re z=\sigma,
\]
for some $\sigma\in(a,b)$, then
\[
u_1(r)=u_2(r)
	\quad \text{ on } \R.
\]
\end{cor}

\begin{remark}[On the domain of $\widetilde{u}$]
\label{remark:poles-inverse}
We note that the location of the strip is crucial for the uniqueness of the inverse Mellin transform. In particular, if $\widetilde u$ has a pole at a point $z=\sigma_0+i\lambda$, then shifting the vertical line of integration may create an extra power $r^{-\sigma_0}$ in $\mathcal M^{-1} \widetilde u$. For instance,
\[
\cM\set{
	e^{-r}
}(z)
=\Gamma(z)
	\quad \text{ for } \re z>0,
\]
while
\[
\cM\set{
	e^{-r}-1
}(z)
=\Gamma(z)
	\quad \text{ for } -1<\re z<0.
\]
\end{remark}

\begin{remark}[Generalized Mellin pairs]
For $u\in L^1(\R_+,r^{\sigma-1})$, $(u,\cM u)$ is an admissible Euclidean--Mellin pair. Thanks to \Cref{thm:M-inv}, one is able to enlarge this class and consider pairs $(\cM_{\sigma}^{-1}\widetilde{u},\widetilde{u})$ for  $\widetilde{u}$ satisfying the assumptions of the theorem.
In this way, polynomials arising from a shift of $\sigma$ are still paired with meromorphic functions having the corresponding poles. (This generalization is different from \cite{Zagier}, though.) In a similar spirit, one has defined large classes of special functions such as the Meijer $G$-function and, more generally, the Fox $H$-function \cite{book:Mellin}.
\end{remark}

%
%

\subsubsection{Elementary properties of $\cM$}
\label{sec:M-prop}

\begin{prop}\label{prop:M-elem}
$\cM$ satisfies the following:
\begin{itemize}
\item Linearity.
\item Differentiation
\begin{equation}\label{eq:M-elem-1}
\cM\set{ru'(r)}(z)
=-z\cM\set{u(r)}(z).
\end{equation}
\item Monomial multiplication in $r$ $\Longleftrightarrow$ (real) translation in $z$,
\begin{equation}\label{eq:M-trans}
\cM\set{u(r)}(z)
=\cM\set{
	r^{a} u(r)
}
(z-a),
	\quad a\in \R.
\end{equation}
\item Rescaling in $r$ $\Longleftrightarrow$ multiplication by exponential in $z$,\footnote{Recall the identity $e^{a+b}=e^a\,e^b$ holds for all $a,b\in\bC$.}
\begin{equation}\label{eq:M-mult-exp}
\cM\set{u(r)}(z)
=a^z \cM\set{u(a r)}(z),
	\quad a>0.
\end{equation}
\item Composition with monomial in $r$ $\Longleftrightarrow$ rescaling in $z$,\footnote{Recall that $e^{ab}=(e^a)^b$ does not always hold, but the principal branch has to be chosen due to \Cref{rmk:exp}.}
\begin{equation}\label{eq:M-rescale}
\cM\set{u(r)}(z)
=a \cM\set{u(r^a)}(a z),
	\quad a>0.
\end{equation}
\end{itemize}
Similar properties for inverse:
\begin{itemize}
	\item (Real) translation in $z$ and $\sigma$ $\Longleftrightarrow$ monomial multiplication in $r$,
	\begin{equation}\label{eq:M-inv-trans}
		\cM_{\sigma}^{-1}
		\set{
			\widetilde{u}(z)
		}(r)
		=r^a
		\cM_{\sigma+a}^{-1}
		\set{\widetilde{u}(z-a)}(r),
			\quad a\in\R.
	\end{equation}
	\item Multiplication by exponential in $z$ $\Longleftrightarrow$ rescaling in $r$,
	\begin{equation}\label{eq:M-inv-exp}
	\cM_{\sigma}^{-1}
	\set{
		\widetilde{u}(z)
	}(r)
	=\cM_{\sigma}^{-1}
	\set{
		a^{z}
		\widetilde{u}(z)
	}\left (
		a r
	\right ),
		\quad a>0.
	\end{equation}
	\item Rescaling in $z$ and $\sigma$ $\Longleftrightarrow$ composition with monomial in $r$,
	\begin{equation}\label{eq:M-inv-rescale}
	\cM_{\sigma}^{-1}
		\set{
			\widetilde{u}(z)
		}(r)
	=a
	\cM_{\frac{\sigma}{a}}^{-1}
	\set{
			\widetilde{u}(a z)
	}\left (
		r^{a}
	\right ),
		\quad a>0.
	\end{equation}
\end{itemize}

\end{prop}

\begin{remark}\label{rmk:exp}
We emphasize that great care must be used in proving \eqref{eq:M-rescale} and \eqref{eq:M-inv-rescale}, as the exponential function with complex arguments does not follow satisfy the same relation as its real counterpart. The two sides of \eqref{eq:M-rescale} can be written respectively as
\begin{align*}
\cM\set{u(r)}(z)
&=\int_0^\infty	r^{\sigma+\lambda i} u(r) \frac{dr}{r},\\
a\cM\set{u(r^a)}(az)
&=a\int_0^\infty r^{a\sigma+a\lambda i} u(r^a) \frac{dr}{r}
=\int_0^\infty (r^{\frac1a})^{a\sigma+a\lambda i} \,dr.
\end{align*}
As a matter of fact, for $a'=\frac{1}{a}$, $b=a\sigma+a\lambda i\in\bC$, the multi-valued sets of complex numbers
\[
	(r^{a'})^b=e^{b\log(r^{a'})}=e^{a'b\Log r+b\cdot 2\pi i\,\bZ}
\quad \text{ and } \quad
	r^{a'b}=e^{\log (r^{a'b})}=e^{a'b\Log r+a'b\cdot 2\pi i\,\bZ}
\]
do not coincide.

Fortunately, in taking Mellin transform one requires a (uniform) decay as $\lambda \to \pm\infty$. This implies that in both of these exponentials one must only take the principal branch in such a way that the term $(\lambda i)(2\pi i\,\bZ)=-2\pi \lambda \,\bZ$ vanishes. The consideration for \eqref{eq:M-inv-rescale} is similar. In all the rest of the paper, the above implicit choice will remain in force.
\end{remark}

\subsubsection{Multiplicative convolution identity}
\label{sec:M-conv}

Similarly to Fourier transform, Mellin transform also satisfies some kind of convolution property.

Recall from \eqref{eq:M-F-1} that Mellin and Fourier transforms are related by
\[
\cM u(\lambda i)=\sqrt{2\pi}\cF_1\{u(e^{-\varrho})\}(\lambda).
\]
Because of the logarithmic change of variable, the convolution identity for Mellin transform is valid for the multiplicative convolution instead. We denote it by the the operation $\star$ that satisfies
\[
(u_1 \star u_2)(e^{-\varrho})
= [u_1(e^{-\cdot})\ast u_2(e^{-\cdot})](\varrho),
\]
where $\ast$ denotes the standard convolution in $\R$. In other words, $\star$ is defined by
\[
(u_1\star u_2)(r)
=\int_{0}^{\infty}
u_1\left(\frac{r}{r'}\right)
u_2(r')
\dfrac{dr'}{r'}
=(u_2\star u_1)(r).
\]
Then, by the convolution identity of the Fourier transform,
\begin{align*}
	\cM\set{u_1\star u_2}(\lambda i)
	&=\sqrt{2\pi}\cF_1
	\left(u_1(e^{-\cdot}) \ast u_2(e^{-\cdot})\right)(\lambda)\\
	&=\sqrt{2\pi}\cF_1\left(u_1(e^{-\cdot})\right)(\lambda)
	\sqrt{2\pi}\cF_1\left(u_2(e^{-\cdot})\right)(\lambda)
	=\cM u_1(\lambda i)
	\cM u_2(\lambda i).
\end{align*}
More generally, for any $\sigma\in\R$, since $e^{-\sigma \varrho}(u_1\ast u_2)(t)=[(e^{-\sigma \cdot}u_1)\ast (e^{-\sigma \cdot}u_2)](\varrho)$, there holds
\begin{equation*}\label{eq:M-conv}
	\cM\set{u_1 \star u_2}(\sigma+\lambda i)
	=\cM u_1(\sigma+\lambda i)
	\cM u_2(\sigma+\lambda i).
\end{equation*}
By duality, given Mellin transforms $\tilde{u}_1(z)$, $\tilde{u}_2(z)$ of admissible functions,
\begin{equation*}\label{M-conv-dual}
\cM_{\sigma}^{-1}\set{
	\widetilde{u}_1
	\widetilde{u}_2
}(r)
=[
	\cM_{\sigma}^{-1} \widetilde{u}_1
	\ast \cM_{\sigma}^{-1} \widetilde{u}_2
](r).
\end{equation*}

\normalcolor

\subsubsection{Isometry and Plancherel identity}

\label{sec:M-isom}

The Mellin transform yields an identification between
$$L^2_r(\R^n) \longrightarrow L^2(\tfrac{n}{2}+\R i) \cong L^2(\R),$$
given by the identity
\begin{equation}\label{eq:M-isom-1}
\int_{\R}
	|\cM u(\tfrac{n}{2}+\lambda i)|^2
\,d\lambda
=\int_{0}^{\infty}
	|u(r)|^2
	r^{n-1}
\,dr
=\frac{1}{\cH^{n-1}(\bS^{n-1})}
\|u\|^2_{L^2(\mathbb R^n)}.
\end{equation}
More generally, for functions $u_1,u_2$ and parameters $\sigma_1,\sigma_2\in\R$\footnote{We keep in mind that if $u_j$ ($j=1,2$) are smooth radial functions with $u_j(0)>0$, then one is not allowed to take $\sigma_1=\sigma_2=0$ in \eqref{eq:M-isom-gen}. Indeed, not only the right hand side but also each Mellin transform on the left hand side would be all undefined.
}
such that the respective Mellin transforms are defined, it holds
\begin{equation}\label{eq:M-isom-gen}
\int_{\R}
	\cM u_1(\sigma_1+\lambda i)
	\cM u_2(\sigma_2-\lambda i)
\,d\lambda
=\int_{0}^{\infty}
	u_1(r)
	u_2(r)
	r^{\sigma_1+\sigma_2-1}
\,dr.
\end{equation}
Note that $\cM u_2(\sigma_2-\lambda i)$ is simply the complex conjugate $\overline{\cM u_2(\sigma_2+\lambda i)}$.

The Plancherel identity \eqref{eq:M-isom-1} is particular useful in understanding the function space $\cH_{\mathcal M}^{\beta}$ described below.

\subsection{Mellin--Sobolev function spaces}


We work on the space of radially symmetric functions $u=u(r)$ on $\mathbb R^n$. For such functions we define the space $\cH_{\mathcal M}^\beta$, for  $\beta\geq 0$, by the norm 
\begin{equation*}\label{norm}
	\|u\|^2_{\cH_{\mathcal M}^\beta}
	=\int_{\R}
		(1+|\lambda|^2)^{\beta}
		\abs{\cM u(\tfrac{n}{2}+\lambda i)}^2
	\,d\lambda.
\end{equation*}
In the particular case that $\beta=0$ we recover the standard $L_r^2(\mathbb R^n)$ thanks to  \eqref{eq:M-isom-1}.

For $\beta=1$ we obtain a variation of the usual Sobolev space $H^1_r(\mathbb R^n)$ with different scaling properties. Indeed,
\begin{align*}
-(\tfrac{n}{2}+\lambda i)\cM u(\tfrac{n}{2}+\lambda i)
&=-(\tfrac{n}{2}+\lambda i)
\int_{0}^{\infty}
	r^{\frac{n}{2}+\lambda i}
	u(r)
\,\frac{dr}{r}
=-\int_{0}^{\infty}
	(r^{\frac{n}{2}+\lambda i})'
	u(r)
\,dr\\
&=\int_{0}^{\infty}
	r^{\frac{n}{2}+\lambda i}
	u'(r)
\,dr
=\int_{0}^{\infty}
	r^{\frac{n}{2}+\lambda i}
	(ru'(r))
\,\frac{dr}{r}
=\cM\set{ru'(r)}(\tfrac{n}{2}+\lambda i),
\end{align*}
so that
\begin{align*}
\int_{\R}
	\abs{\tfrac{n}{2}+\lambda i}^2
	\abs{\cM u(\tfrac{n}{2}+\lambda i)}^2
\,d\lambda
&=\int_{\R}
	\abs{\cM\set{ru'(r)}(\tfrac{n}{2}+\lambda i)}^2
\,d\lambda\\
&=\int_{\R}
	(r u'(r))^2
	r^{n-1}
\,dr
=\frac{1}{\cH^{n-1}(\bS^{n-1})}\int_{\R^n}
	|x|^2|\nabla u|^2
\,dx.
\end{align*}
We conclude that the norm in $\cH_{\mathcal M}^{1}$ is equivalent to the norm
\begin{equation*}
\left (
	\int_{\mathbb R^n}|u|^2\,dx
	+ \int_{\mathbb R^n} r^2|\nabla u|^2\,dx
\right )^{\frac12},
\end{equation*}
which is closely related, but not exactly equal, to the usual Sobolev norm of $H^1(\mathbb R^n)$.

Finally, we remark here that the $\mathcal H^\beta_\M$ are the $Z$ spaces from \cite{FV}. As in that paper, one can show that smooth, compactly supported functions on $(0,\infty)$ are dense in $\mathcal H^\beta_{\mathcal M}$.


\subsection{The Mellin symbol of the fractional Laplacian}

A key point in our discussion is to understand the fractional Laplacian in Mellin variables. We will show that Mellin transform converts it into a multiplication operator with a nicely behaved symbol:

\begin{prop}\label{prop:M-Ds}
For $s\in(0,\frac{n}{2})$, we have
\begin{equation*}\label{eq:M-Ds}
	\cM\set{
		\Ds u(r)
	}(z)
=\Theta_s(z)\cM u(z-2s),
	\qquad z\in\bC,
\end{equation*}
where the Mellin symbol is given by
\begin{equation}\label{eq:M-Ds-sym}
\Theta_{s}(z)
:=
2^{2s}
\dfrac{
	\Gamma\left (
	\frac{
		z
	}{2}
	\right )
}{
	\Gamma\left (
	\frac{
		n-z
	}{2}
	\right )
}
\dfrac{
	\Gamma\left (
	\frac{
		n+2s-z
	}{2}
	\right )
}{
	\Gamma\left (
	\frac{
		z-2s
	}{2}
	\right )
},
	\qquad z\in\bC.
\end{equation}
On any fixed vertical strip,
\begin{equation}\label{symbol-asymptotics}
\Theta_s(z)\asymp |z|^{2s}\quad \text{as}\quad |z|\to\infty.
\end{equation}
In addition, the fractional seminorm is written as
\begin{equation}\label{formula:seminorm}
\begin{split}
\int_{0}^{\infty} |(-\Delta)^{\frac{s}{2}} u|^2r^{n-1}\,dr
= \int_{\mathbb R} \Theta_s(\tfrac{n}{2}+s+i\lambda)|\mathcal Mu(\tfrac{n}{2}-s+i\lambda)|^2\,d\lambda.
\end{split}
\end{equation}
\end{prop}

Before we give the proof of the Proposition, we need to introduce some special formulas. First we write the composition of Mellin transform with the Hankel transform in terms of a Mellin transform. For this purpose, we recall the following formula for the integral of Bessel functions \cite[formula 14 in section 6.561]{Gradshteyn-Ryzhik}
\begin{equation}\label{eq:Bessel}
	\int_{0}^{\infty}
	r^{\mu} J_{\nu}(\zeta r)
	\,dr
	=2^{\mu}\zeta^{-\mu-1}
	\dfrac{
		\Gamma\left(
		\frac{\mu+\nu+1}{2}
		\right)
	}{
		\Gamma\left(
		\frac{-\mu+\nu+1}{2}
		\right)
	},
\end{equation}
for $-{\rm Re\,} \nu -1 < {\rm Re\,} \mu <\frac12$, $\zeta>0$.
\begin{lemma}\label{lem:MF-M}
For $0<{\rm Re\,} z <\frac{n+1}{2}$, there holds
\begin{equation}\label{eq:MF-M}
\begin{split}
\cM \set{
	\cF u(\zeta)
}(z)
=
\cM \set{
	\cF^{-1} u(r)
}(z)
&=
2^{z-\frac{n}{2}}
\frac{
	\Gamma\left (
	\frac{
		z
	}{2}
	\right )
}{
	\Gamma\left (
	\frac{
		n-z
	}{2}
	\right )
}
\cM u(n-z).
		\end{split}
\end{equation}
\end{lemma}

\begin{proof}
In a recent revision of the paper, we have found that this is a result known from 1979 \cite{Ormerod}. We have decided to keep our proof here for convenience of the reader.

First, note that Hankel transform is self-inverse. Using \eqref{eq:Bessel},
\begin{equation*}
\begin{split}
\cM\set{
	\cF u(\zeta)
}(z)
&=\int_0^\infty
	\zeta^{z}
	\widehat u(\zeta)
\,\frac{d\zeta}{\zeta}
=\int_0^\infty
	\left[
		\int_{0}^{\infty}
			\zeta^{z-1}
			\zeta^{-\frac{n-2}{2}}
			J_{\frac{n-2}{2}}(\zeta r)
		\,d\zeta
	\right]
	r^{\frac{n}{2}}
	u(r)
	\,dr\\
&=\int_0^\infty
\left [
	2^{z-\frac{n}{2}}
	r^{-z+\frac{n-2}{2}}
	\frac{
		\Gamma\left (
			\frac{
				z
			}{2}
		\right )
	}{
		\Gamma\left (
			\frac{
				n-z
			}{2}
		\right )
	}
\right ]
r^{\frac{n}{2}}
u(r)
\,dr
=
2^{z-\frac{n}{2}}
\frac{
	\Gamma\left (
	\frac{
		z
	}{2}
	\right )
}{
	\Gamma\left (
	\frac{
		n-z
	}{2}
	\right )
}
\cM u(n-z),
\end{split}
\end{equation*}
as desired.
\end{proof}

\begin{proof}[Proof of Proposition \ref{prop:M-Ds}]
This symbol \eqref{eq:M-Ds-sym} was first calculated in the paper \cite{Frank} for $n=3$ and $s=\frac{1}{2}$ in the context of the intermediate long-wave equation. A general formula for any $s\in(0,\frac{n}{2})$ was given in  \cite{DelaTorre-Gonzalez} using an extension argument to calculate the conformal fractional Laplacian on the cylinder (this is, in terms of Emden-Fowler coordinate $\varrho=-\log r$, recall the discussion in Section \ref{sec:M-vs-F}). A more recent reference, using Mellin transform, is \cite{Pagnini-Runfola}. Here we give a  proof for
 for $s\in(0,\frac{1}{2})$ and $\frac{n-1}{2}+2s<{\rm Re\,} z<\frac{n+1}{2}$.\\

Using \eqref{eq:MF-M} twice,
\begin{equation*}
\begin{split}
\cM\set{
	\cF^{-1}\set{
		\zeta^{2s}
		\cF u(\zeta)
	}(r)
}(z)
&=2^{z-\frac{n}{2}}
\frac{
	\Gamma\left (
	\frac{
		z
	}{2}
	\right )
}{
	\Gamma\left (
	\frac{
		n-z
	}{2}
	\right )
}
\cM \set{
	\zeta^{2s} \cF u(\zeta)
}(n-z)\\
&=2^{z-\frac{n}{2}}
\frac{
	\Gamma\left (
	\frac{
		z
	}{2}
	\right )
}{
	\Gamma\left (
	\frac{
		n-z
	}{2}
	\right )
}
\cM\set{\cF u(\zeta)}(n+2s-z)\\
&=2^{z-\frac{n}{2}}
\frac{
	\Gamma\left (
	\frac{
		z
	}{2}
	\right )
}{
	\Gamma\left (
	\frac{
		n-z
	}{2}
	\right )
}
2^{(n+2s-z)-\frac{n}{2}}
\frac{
	\Gamma\left (
	\frac{
		n+2s-z
	}{2}
	\right )
}{
	\Gamma\left (
	\frac{
		n-(n+2s-z)
	}{2}
	\right )
}
\cM u\bigl(n-(n+2s-z)\bigr)\\
&=2^{2s}
\frac{
	\Gamma\left (
	\frac{
		z
	}{2}
	\right )
}{
	\Gamma\left (
	\frac{
		n-z
	}{2}
	\right )
}
\frac{
	\Gamma\left (
	\frac{
		n+2s-z
	}{2}
	\right )
}{
	\Gamma\left (
	\frac{
		z-2s
	}{2}
	\right )
}
\cM u(z-2s),
\end{split}
\end{equation*}
and this yields the formula for the symbol $\Theta_s(z)$. To show \eqref{symbol-asymptotics}, just use \eqref{asymptotics-gamma0} to see that for $z$ in a bounded vertical strip,
\begin{align*}
2^{2s}
\frac{
	\Gamma\left (
	\frac{
		z
	}{2}
	\right )
}{
	\Gamma\left (
	\frac{
		z-2s
	}{2}
	\right )
}
\frac{
	\Gamma\left (
	\frac{
		n+2s-z
	}{2}
	\right )
}{
	\Gamma\left (
	\frac{
		n-z
	}{2}
	\right )
}
&\sim
2^{2s}
\left(
	\tfrac{z-2s}{2}
\right)^{s}
\left(
	\tfrac{n-z}{2}
\right)^{s}
\sim
(\im z)^s(-\im z)^s
=\abs{\im z}^{2s}
\sim |z|^{2s}.
\end{align*}

Now, for the proof of \eqref{formula:seminorm} simply observe that the isometry property \eqref{eq:M-isom-gen} implies that
\begin{equation}\label{fractional-seminorm}
\begin{split}
\int_{\mathbb R^n} u(-\Delta)^su\,r^{n-1}\,dr
&=\int_{0}^{\infty} \mathcal Mu(\tfrac{n}{2}-s+i\lambda) \cdot\mathcal M\{(-\Delta)^s u\}(\tfrac{n}{2}+s-i\lambda)\,d\lambda\\
&=\int_{\mathbb R} \mathcal Mu(\tfrac{n}{2}-s+i\lambda) \Theta_s(\tfrac{n}{2}+s-i\lambda)\mathcal Mu(\tfrac{n}{2}-s-i\lambda)\,d\lambda\\
&= \int_{\mathbb R} \Theta_s(\tfrac{n}{2}+s+i\lambda)|\mathcal Mu(\tfrac{n}{2}-s+i\lambda)|^2\,d\lambda.
\end{split}
\end{equation}
An alternative proof of \eqref{formula:seminorm} can be given as follows. We start with
\begin{equation}\label{fractional-seminorm1}\begin{split}
\int_{0}^{\infty} |(-\Delta)^{\frac{s}{2}} u|^2r^{n-1}\,dr
&=
\int_{0}^{\infty}
	\Theta_{\frac{s}{2}}(\tfrac{n}{2}+\lambda i)
	\cM u(\tfrac{n}{2}-s+\lambda i)
	\Theta_{\frac{s}{2}}(\tfrac{n}{2}-\lambda i)
	\cM u(\tfrac{n}{2}-s-\lambda i)
\,d\lambda.
\end{split}
\end{equation}
Taking into account that
\begin{equation*}\label{doubling}
\Theta_{\frac{s}{2}}(\tfrac{n}{2}+\lambda i)
\Theta_{\frac{s}{2}}(\tfrac{n}{2}-\lambda i)
=2^s
\frac{
	\Gamma(\frac{n+2\lambda i}{4})
	\Gamma(\frac{n+2s-2\lambda i}{4})
}{
	\Gamma(\frac{n-2\lambda i}{4})
	\Gamma(\frac{n-2s+2\lambda i}{4})
}
\cdot
2^s
\frac{
	\Gamma(\frac{n-2\lambda i}{4})
	\Gamma(\frac{n+2s+2\lambda i}{4})
}{
	\Gamma(\frac{n+2\lambda i}{4})
	\Gamma(\frac{n-2s-2\lambda i}{4})
}
=\Theta_s(\tfrac{n}{2}+s+\lambda i),
\end{equation*}
we recover the well known fact that  \eqref{fractional-seminorm} and \eqref{fractional-seminorm1} yield the same fractional seminorm.

\end{proof}

\section{Functional analysis}\label{section:functional-new}

Here we give the functional analysis framework for the eigenvalue problem \eqref{eigenvalue-problem}, and the precise definition of the function space $\mathds L^2_\dag$.

\subsection{The Fourier approach}

Let us make more precise the change of variable  \eqref{change}.  Let $\widehat u(\zeta)$ be the Fourier/Hankel transform of $u$.  Recalling that $\vartheta=\zeta^s$, we  set
\begin{equation}\label{f-sharp}
u^\sharp(\vartheta )=\widehat{u}(\vartheta^{1/s})=\vartheta^{\frac{2-n}{2s}}\int_0^\infty J_{\frac{n-2}{2}}(\vartheta^{1/s}r)r^{\frac{n}{2}}u(r)\,dr.
\end{equation}
Fix $u^\dag=u^\dag(\rho)$ to be the inverse Fourier/Hankel transform of $u^\sharp$, that is,
\begin{equation*}\label{equation1000}
u^\dag (\rho)
=\rho^{\frac{2-n}{2}}
\int_{0}^{\infty }
	J_{\frac{n-2}{2}}(\vartheta \rho)
	\vartheta^{\frac{n}{2}}
	u^\sharp(\vartheta)
\,d\vartheta.
\end{equation*}
Using again equation \eqref{Hankel} for the Fourier (Hankel) transform of $u$ we can give an integral formula for $u^\dag$. Indeed,
\begin{equation}\label{f-dag}\begin{split}
u^\dag(\rho)
&=\rho^{\frac{2-n}{2}}
\int_0^\infty
	J_{\frac{n-2}{2}}(\vartheta \rho)
	\vartheta^{\frac{n}{2}}
	\left(
		\vartheta^{\frac{2-n}{2s}}
		\int_0^\infty
			J_{\frac{n-2}{2}}(\vartheta^{1/s}\rho_1)
			u(\rho_1)
		\,d\rho_1
	\right)
\,d\vartheta\\
&=
\int_0^\infty
	A_s(\rho,\rho_1)
	u(\rho_1)
\,d\rho_1,
\end{split}
\end{equation}
where we have defined the kernel
\begin{align*}
A_s(\rho,\rho_1)
:=\rho^{\frac{2-n}{2}}
\rho_1^{\frac{n}{2}}
\int_0^\infty
	J_{\frac{n-2}{2}}(\vartheta \rho)
	J_{\frac{n-2}{2}}(\vartheta^{1/s}\rho_1)
	\vartheta^{\frac{n}{2}-\frac{n-2}{2s}}
\,d\vartheta.
\end{align*}
Note that the integral above is convergent as $\vartheta\to\infty$. Indeed, looking at the asymptotics of the Bessel function, the asymptotic behavior of the integrand (for fixed $\rho,\rho_1$) is given by
\begin{equation*}
J_{\frac{n-2}{2}}(\vartheta\rho)J_{\frac{n-2}{2}}(\vartheta^{1/s}\rho_1)
	\vartheta^{\frac{n}{2}-\frac{n-2}{2s}}
=  
\vartheta^{-\frac{n-1}{2s}(1-s)}
\left(
	\cos\left(\vartheta \rho-\tfrac{n-1}{4}\pi\right)
	\cos\left(\vartheta^{1/s}\rho_1-\tfrac{n-1}{4}\pi\right)
	+O\left(\vartheta^{-1}\right)
\right),
\end{equation*}
as $\vartheta\to\infty$, and the oscillating part is
\begin{multline*}
\int_{1}^{\infty}
	\tau^{-\frac{n-1}{2}(1-s)}
	\cos\left(\tau^s\rho-\tfrac{n-1}{4}\pi\right)
	\cos\left(\tau\rho_1-\tfrac{n-1}{4}\pi\right)
	\cdot s\tau^{s-1}\,d\tau
\\
=\frac{s}{2}
\int_{1}^{\infty}
	\tau^{-\frac{n+1}{2}(1-s)}
\,d\left(
	\frac{\cos(\tau^s\rho-\tau\rho_1)}{s\tau^{s-1}\rho-\rho_1}
	+\frac{\cos(\tau^s \rho+\tau\rho_1-\frac{n-1}{2}\pi)}{s\tau^{s-1}\rho+\rho_1}
\right),
\end{multline*}
which is integrable (after an integration by parts).
%

\medskip

The crucial observation is encoded in the following Lemma (the function spaces are indicated below):

\begin{lem}\label{lem:Ls-L1}
If $u$ is a solution to
\begin{equation}\label{eq:100}
L_su=f,\quad u=u(r),\, r>0,
\end{equation}
then $u^\dag$ is a solution to
\begin{equation}\label{eq:101}
L_1 u^\dag =f^\dag, \quad u^\dag=u^\dag(\rho),\, \rho>0.
\end{equation}
\end{lem}

\begin{proof}
Taking Fourier transform of \eqref{eq:100} yields
\begin{equation*}
\zeta^{2s}\widehat u+\tfrac{1}{2s}\zeta\partial_\zeta \widehat u=\widehat f,\quad \widehat u=\widehat u(\zeta),\, \zeta>0.
\end{equation*}
The change of variable $\vartheta=\zeta^s$  reduces this equation to a local one
\begin{equation*}
\vartheta^{2}u^\sharp+\tfrac{1}{2}\vartheta\partial_\vartheta u^\sharp= f^\sharp,\quad  u^\sharp= u^\sharp(\vartheta),\, \vartheta>0.
\end{equation*}
Now we Fourier back to arrive to
\begin{equation*}
-\Delta u^\dag-\tfrac{1}{2}\rho \partial_\rho u^\dag-\tfrac{n}{2}u^\dag=f^\dag, \quad u^\dag=u^\dag(\rho),\,\rho>0,
\end{equation*}
which is precisely \eqref{eq:101}.
\end{proof}

The relevant function spaces for equation \eqref{eq:101} are $L^2({\rm exp}):=L_r^2(\mathbb R^n,e^{\frac{\rho^2}{4}})$ and $H^1({\rm exp}):=H^1_r(\mathbb R^n,e^{\frac{\rho^2}{4}})$. It is well known that we have compactness so that the spectral theorem can be applied
\cite{Kavian,Escobedo-Kavian}.

Now we translate those spaces to our setting. The space $L^2({\rm exp})$ is equipped with the norm
\begin{equation*}
\|u\|^2_{L^2({\rm exp})}=\int_0^\infty |u(\rho)|^2 e^{\frac{\rho^2}{4}} \rho^{n-1}\,d\rho.
\end{equation*}
In the light of \eqref{simple-relation} it is natural to define, for each $s\in(0,1)$,  the Hilbert space $\mathds L^2_\dag:=\mathds L^2_\dag(s)$ with the norm induced by the scalar product
\begin{equation}\label{defi:scalar-product-dag}\begin{split}
\angles{u_1,u_2}_\dag
=\int_{0}^{\infty}
	u_1^{\dagger}(\rho)
	\overline{u_2^{\dagger}(\rho)}
	e^{\frac{\rho^{2}}{4}}
	\rho^{n-1}
\,d\rho,
	\quad \text{i.e.} \quad
\norm[\mathds L^2_\dag]{u}^2
:=\int_{0}^{\infty}
	\abs{u^\dag(\rho)}^2
	e^{\frac{\rho^2}{4}}
	\rho^{n-1}
\,d\rho,
\end{split}\end{equation}
and similarly the (energy) seminorm for $\ell\in \bN$ by
\begin{equation*}\label{defi:seminorm}
\|u\|^2_{\dot{\mathds H}^\ell_\dag}
:=\int_0^\infty
	|\partial_\rho^\ell u^\dag(\rho)|^2 e^{\frac{\rho^2}{4}}
	\rho^{n-1}
\,d\rho,
\end{equation*}
and the energy space $\mathds H^\ell_\dag:p=\mathds H^\ell_\dag(s)$ by
\begin{equation*}
\|u\|^2_{{\mathds H^\ell_\dag}}=\|u\|^2_{{\mathds L^2_\dag}}+\|u\|^2_{\dot{\mathds H}^\ell_\dag}.
\end{equation*}
Then, automatically, if $f\in \mathds L^2_\dag$, then $u\in \mathds H^1_\dag$. Moreover, compactness for the operator $L_s$ is also inherited from the local case.

\begin{remark}\label{remark:interpret}
We may observe that
\begin{equation*}
\int_0^\infty |u^\sharp(\vartheta)|^2\vartheta^{n-1}\,d\vartheta=s\int_0^\infty |\widehat u(\zeta)|^2 \zeta^{sn-1}\,d\zeta,
\end{equation*}
so that the transformation $u^\dag$ can be interpreted as a differentiation of order $\frac{sn-n}{2}$ in $\mathbb R^n$ (i.e., fractional integration of order $\frac{n-sn}{2}$).
\end{remark}

Nevertheless, the Fourier approach gives little information on how the space $\mathds L^2_\dag$ looks like. Instead, we will rewrite them using Mellin transform.

\subsection{The ``isometry" $\Phi_s$}\label{subsection:multiplier}

Consider the  Mellin multiplier $\Lambda_s$ defined by
\begin{equation}\label{multiplier-Lambda}
\Lambda_s(z)
:=2^{z}
\frac{
	\Gamma\left (
		\frac{
			z
		}{2}
	\right )
}{
	\Gamma\left (
		\frac{
			n-z
		}{2}
	\right )
}
2^{-\frac{z}{s}}
\frac{\Gamma(\tfrac{n-z}{2s})}{	\Gamma(\tfrac{n}{2}-\tfrac{n}{2s}+\tfrac{z}{2s})}.
\end{equation}
Observe that it  simplifies to $\Lambda_1(z)\equiv 1$ in the limit $s\to 1$.

We first show that it admits an inverse Mellin transform on certain vertical line.

\begin{lem}[Asymptotics of multiplier]
\label{lem:asymp-mult}
As $\lambda \to \pm\infty$,
\begin{equation}\label{eq:multiplier-asymp}
|\Lambda_s(\sigma+\lambda i)|
=2^{-(\frac1s-1)n}
	s^{\frac{n}{2}-\frac{n-\sigma}{s}}
	|\lambda|^{(\frac1s-1)(n-\sigma)}
	(1+O(|\lambda|^{-1})).
\end{equation}
In particular, when $\sigma > n+\frac{2s}{1-s}$ and $\sigma \notin n+2s\bN$, $\cM_{\sigma}^{-1}\Lambda_s(r)$ is well defined.
\end{lem}

\begin{proof}
Using the asymptotics of the Gamma function from \eqref{estimate-Gamma-2}--\eqref{estimate-Gamma-2s}, we see that as for $\sigma,\lambda\in\R$ with $|\lambda|\to\infty$,
\begin{align*}
|\Lambda_s(\sigma+\lambda i)|
&=2^{\sigma-\frac{\sigma}{s}}
\dfrac{
	\sqrt{2\pi}e^{-\frac{\pi}{4}|\lambda|}
	\left(\frac{|\lambda|}{2}\right)^{\frac{\sigma}{2}-\frac12}
	\cdot
	\sqrt{2\pi}e^{-\frac{\pi}{4s}|\lambda|}
	\left(\frac{|\lambda|}{2s}\right)^{\frac{n-\sigma}{2s}-\frac12}	
}{
	\sqrt{2\pi}e^{-\frac{\pi}{4}|\lambda|}
	\left(\frac{|\lambda|}{2}\right)^{\frac{n-\sigma}{2}-\frac12}
	\cdot
	\sqrt{2\pi}e^{-\frac{\pi}{4s}|\lambda|}
	\left(\frac{|\lambda|}{2s}\right)^{\frac{n}{2}-\frac{n-\sigma}{2s}-\frac12}
}
(1+O(|\lambda|^{-1})).
\end{align*}
This simplifies to \eqref{eq:multiplier-asymp}.
\end{proof}


We are actually more interested in the multiplier $\Lambda_s^{-1}(z)$, which has poles at the points
\begin{equation}\label{poles-Lambda}
z_j=n+2j, \,j\in\mathbb N, \quad\text{and}\quad z'_\ell=n(-s+1)-2s\ell,\,  \ell\in\mathbb N.
\end{equation}
 Thus, if $u$ is, for instance, a smooth function with compact support outside the origin (which implies that $\mathcal M u(z)$ is holomorphic on the whole complex plane), then one may define $\Phi_s^{-1}u$ by
	\begin{equation*}
		\cM\set{
			\Phi_s^{-1} u(r)
		}(z)
		=\Lambda_s^{-1}(z)
		\cM u(z),
	\end{equation*}
and then take inverse Mellin on any vertical line in the region $\{n(1-s)<\Real z<n\}$. For non-smooth functions, $\Phi^{-1}_s u$ is defined by approximation (this argument will be made precise later). Note that $\Phi_s u$ may be introduced in a similar manner.

In the following Lemma we will show that  $\Phi_s$ can be seen as  a (well behaved) differentiation operator in Mellin--Sobolev spaces:

\begin{lemma}\label{lemma:isometry}
Let $\beta\geq 0$. Then $\Phi_s:\cH_{\mathcal M}^{\beta+\frac{n}{2s}-\frac{n}{2}}\to \cH_{\mathcal M}^{\beta}$ is a quasi-isometry, namely,
\begin{equation*}
\|\Phi_s  u\|^2_{ \cH_{\mathcal M}^\beta}\asymp\|  u\|^2_{ \cH_{\mathcal M}^{\beta+\frac{n}{2s}-\frac{n}{2}}},
	\qquad \text{for all } u\in  \cH_{\mathcal M}^{\beta+\frac{n}{2s}-\frac{n}{2}},
\end{equation*}
where the implicit constants depend only on $n$, $s$ and $\beta$.
\end{lemma}

\begin{remark}
There are other possible normalizations, which involve different choices of $\Lambda_s$.
We have chosen this particular one due to the simplicity of relation \eqref{simple-relation}, even if other choices would also yield a quasi-isometry $\cH_{\mathcal M}^{\beta'}\to \cH_{\mathcal M}^\beta$.
\end{remark}

\begin{proof}[Proof of Lemma \ref{lemma:isometry}]
Using \eqref{eq:multiplier-asymp},
we see that for 
$\lambda\in\mathbb R$,
\begin{equation*}
|{\Lambda}_s(\tfrac{n}{2}+\lambda i)|\asymp  |\lambda|^{\frac{n}{2s}-\frac{n}{2}} \quad\text{as}\quad |\lambda|\to\infty.
\end{equation*}
Then
\begin{equation*}
		\begin{split}
			\|\Phi_s  u\|^2_{ \cH_{\mathcal M}^\beta}
			&=\int_{\R}
				(1+|\lambda|^2)^{\beta}
				\abs{
					\cM\set{\Phi_s u}(\tfrac{n}{2}+\lambda i)
				}^2
			\,d\lambda\\
          &=
 			\int_{\R}
				(1+|\lambda|^2)^{\beta}|\Lambda_s(\tfrac{n}{2}+\lambda i)|^2
				\abs{
					\cM u(\tfrac{n}{2}+\lambda i)
				}^2
			\,d\lambda\\
			&\asymp
			\int_{\R}
				(1+|\lambda|^2)^{\beta+\frac{n}{2s}-\frac{n}{2}}
				\abs{
					\cM u(\tfrac{n}{2}+\lambda i)
				}^2
			\,d\lambda
			=
			\| u\|^2_{ \cH_{\mathcal M}^{\beta+\frac{n}{2s}-\frac{n}{2}}},
		\end{split}
	\end{equation*}
	as desired.
\end{proof}

\medskip

Now we can give a precise interpretation of the isometry $\Phi_s$. Indeed, it encodes the change of variable \eqref{change} from the Mellin point of view:

\begin{lemma} \label{lemma:equivalence-new}
It holds
\begin{equation}\label{simple-relation}
\Phi_s^{-1}u(r)=\left (\tfrac{r^{2s}}{4}\right )^{\frac{n}{2}-\frac{n}{2s}}u^\dag(r^s).
\end{equation}
\end{lemma}


\begin{proof}
We will need \eqref{eq:MF-M} in the following forms:
\begin{equation}\label{eq:MF-M-1}
\cM \set{
	\cF^{-1} u(r)
}(n-z)
=
2^{\frac{n}{2}-z}
\frac{
	\Gamma\left (
		\frac{n}{2}
		-\frac{z}{2}
	\right )
}{
	\Gamma\left (
	\frac{
		z
	}{2}
	\right )
}
\cM u(z),
	\quad \text{ for } \quad
0<\re z<\tfrac{n+1}{2},
\end{equation}
\begin{equation}\label{eq:MF-M-2}
\cM \set{
	\cF^{-1} u(r)
}\left (n-\tfrac{z}{s}\right )
=
2^{\frac{n}{2}-\frac{z}{s}}
\frac{
	\Gamma\left (
		\frac{n}{2}
		-\frac{z}{2s}
	\right )
}{
	\Gamma\left (
	\frac{
		z
	}{2s}
	\right )
}
\cM u(\tfrac{z}{s}),
	\quad \text{ for } \quad
0<\re \tfrac{z}{s}<\tfrac{n+1}{2}.
\end{equation}
Using \eqref{eq:MF-M-1}--\eqref{eq:MF-M-2}, for $0<\re z<\frac{n+1}{2}s$, we have
\begin{align*}
\cM\set{\Phi_s^{-1}u}(n-z)
&=\Lambda_s(n-z)^{-1} \cM u(n-z)\\
&=2^{\frac{n}{s}-n}
	2^z
	\frac{
		\Gamma(\frac{z}{2})
	}{
		\Gamma(\frac{n}{2}
		-\frac{z}{2})
	}
	2^{-\frac{z}{s}}
	\frac{
		\Gamma(\frac{n}{2}
		-\frac{z}{2s})
	}{
		\Gamma(\frac{z}{2s})
	}
	\cM u(n-z)\\
&=2^{\frac{n}{s}-n}
	2^{-\frac{z}{s}}
	\frac{
		\Gamma(\frac{n}{2}
		-\frac{z}{2s})
	}{
		\Gamma(\frac{z}{2s})
	}
	\cdot
\left (
	2^z
	\frac{
		\Gamma(\frac{z}{2})
	}{
		\Gamma(\frac{n}{2}
		-\frac{z}{2})
	}
	\cM\set{\cF^{-1}\widehat{u}}(n-z)
\right )\\
&=2^{\frac{n}{s}-n}
	2^{-\frac{z}{s}}
	\frac{
		\Gamma(\frac{n}{2}
		-\frac{z}{2s})
	}{
		\Gamma(\frac{z}{2s})
	}
	\cdot
	2^{\frac{n}{2}}
	\cM\set{\widehat{u}(\zeta)}(z)\\
&=2^{\frac{n}{s}-n}
	2^{-\frac{z}{s}}
	\frac{
		\Gamma(\frac{n}{2}
		-\frac{z}{2s})
	}{
		\Gamma(\frac{z}{2s})
	}
	\cdot
	2^{\frac{n}{2}}
	\frac{1}{s}
	\cM\set{\widehat{u}(\zeta^{\frac{1}{s}})}\left (\tfrac{z}{s}\right ),
\end{align*}
and recalling the definition of $u^\sharp$ from \eqref{f-sharp},
\begin{align*}
\cM\set{\Phi_s^{-1}u}(n-z)
&=\frac{2^{\frac{n}{s}-n}}{s}
	\cdot
	2^{\frac{n}{2}-\frac{z}{s}}
	\frac{
		\Gamma(\frac{n}{2}
		-\frac{z}{2s})
	}{
		\Gamma(\frac{z}{2s})
	}
	\cM\set{u^\sharp(\zeta)}\left (\tfrac{z}{s}\right )\\
&=\frac{2^{\frac{n}{s}-n}}{s}
	\cM\set{(\cF^{-1}u^\sharp)(r)}
	\left (n-\tfrac{z}{s}\right )\\
&=\frac{2^{\frac{n}{s}-n}}{s}
	\cM\set{u^\dag(r)}
	\left (n-\tfrac{z}{s}\right )\\
&=2^{\frac{n}{s}-n}
	\cM\set{u^\dag(r^s)}(ns-z)\\
&=2^{\frac{n}{s}-n}
	\cM\set{r^{ns-n}u^\dag(r^s)}(n-z).
\end{align*}
Hence by Mellin inversion, taking into account that $\Real (n-z)$ can be chosen in the interval $(n(1-s),n)$, we obtain the conclusion of the Lemma.
\end{proof}

In the light of \eqref{simple-relation}, from the expression \eqref{defi:scalar-product-dag} it is natural to set  the Hilbert space $\mathds L^2_\dag:=\mathds L^2_\dag(s)$ with the norm induced by the scalar product (up to multiplicative constant)
\begin{equation*}\label{defi:scalar-product}\begin{split}
\angles{u_1,u_2}_\dag
&:=\int_0^\infty
	\Phi_s^{-1}u_1(r)\,
	\overline{\Phi_s^{-1}u_2(r)}
	\,e^{\frac{r^{2s}}{4}}
	r^{n(2-s)-1}
\,dr
=2^{\frac{n}{s}-n}
\int_{0}^{\infty}
	u_1^{\dagger}(r^s)
	\overline{u_2^{\dagger}(r^s)}
	e^{\frac{r^{2s}}{4}}
	r^{sn-1}
\,dr\\
&=2^{\frac{n}{s}-n}s^{-1}\int_{0}^{\infty}
	u_1^{\dagger}(\rho)
	\overline{u_2^{\dagger}(\rho)}
	e^{\frac{\rho^{2}}{4}}
	\rho^{n-1}
\,d\rho.
\end{split}\end{equation*}
That is,
\begin{equation}\label{defi:norm}\begin{split}
\|u\|^2_\dag
:=\int_0^\infty
	|\Phi_s^{-1}u(r)|^2\,
	e^{\frac{r^{2s}}{4}}
	r^{n(2-s)-1}
\,dr=2^{\frac{n}{s}-n}s^{-1}\int_{0}^{\infty}
	|u^{\dagger}(\rho)|^2
	e^{\frac{\rho^{2}}{4}}
	\rho^{n-1}
\,d\rho.
\end{split}\end{equation}

Note that $\Phi^{-1}u$ is well defined if, for instance, $u$ is smooth with compact support outside the origin. $\mathds L^2_\dag$ is more precisely defined as the closure of such functions with respect to the $\|\cdot\|_\dag$ norm.

\begin{lemma}\label{lemma:u-holomorphic}
If $u$ belongs to $\mathds L^2_\dag$, then $u$ is holomorphic in a vertical strip of width at least $2s$.
\end{lemma}

\begin{proof}
First assume that $u$ is smooth of compact support outside the origin (so that $\mathcal M u(z)$ is holomorphic on $\mathbb C$),  and set $v:=\Phi^{-1} u$, that is $\mathcal Mv(z)=\Lambda_s^{-1}(z)\mathcal Mu(z)$. We have that $\mathcal Mv(z)$ is a meromorphic function with (possibly) simple poles given by those of $\Lambda_s^{-1}(z)$, given in \eqref{poles-Lambda}.

For $z=\sigma+\lambda i$,
\begin{equation*}
|\mathcal Mv(z)|\leq \int_0^\infty r^{\sigma-1} |v(r)|\,dr\leq \left(\int_0^\infty r^{2\sigma-1+ns-2n} e^{-\frac{r^{2s}}{4}}
	\,dr\right)^{\frac{1}{2}}
\left(\int_0^\infty |v(r)|^2 e^{\frac{r^{2s}}{4}}
	r^{n(2-s)-1}\,dr\right)^{\frac{1}{2}}
\end{equation*}
By our definition of the space $\mathds L^2_\dag$, we know that
\begin{equation*}
\int_0^\infty
	|v(r)|^2\,
	e^{\frac{r^{2s}}{4}}
	r^{n(2-s)-1}
\,dr<\infty,
\end{equation*}
so $\mathcal M v$ is bounded in the half-plane $\{\Real z>n\frac{2-s}{2}\}$. In particular, it is analytic.

Now, writing $\mathcal Mu(z)=\Lambda_s(z)\mathcal Mv(z)$, and taking into account that the poles of $\Lambda_s(z)$ are located at
$z_j=-2j$, $j\in\mathbb N$, and $z'_\ell=n+2s\ell$, $\ell\in\mathbb N$, we deduce that $\mathcal Mu(z)$ is holomorphic on the strip
$\{n\frac{2-s}{2}<\Real z<n\}$, which has width $\frac{n}{2}s$,
as desired.

Finally, for functions $u$ that are limits of $u_n$ with respect to  the $\|\cdot\|_\dag$ norm, we realize that the above argument also works for each $u_n$ and for differences $u_n-u$, thus showing that $\mathcal Mu_n$ is a sequence of holomorphic functions on the strip which converge uniformly. Thus the limit is also holomorphic (this is a classical result originally due to Weierstrass, see \cite{Krantz}).\\
\end{proof}

\subsection{A fractional version of Gevrey--Sobolev spaces}\label{subsection:Gevrey}

Now we would like to characterize $\mathds L^2_\dag$ in terms of known function spaces.


Gevrey classes were introduced in \cite{Gevrey} as an intermediate space between $\mathcal C^\infty$ and real analytic functions. We could not find any  canonical reference in the literature except for  \cite{Rodino}. These spaces very often appear in the context of hyperbolic equations; for instance,  a classical book is \cite[Page 281]{Hormander} for the study of wave front sets in PDEs. They have been also recently used in the context of the Euler  \cite{Ionescu-Jia} and Navier--Stokes equations \cite{Bedrossian-Germain-Masmoudi1}. The precise notion is given below:

\begin{defi}
We say that $u$ belongs to the {\em Gevrey--Sobolev} class $\mathcal G_p$ if there exists a constant $C>0$ such that
\begin{equation}\label{Gevrey}
\|u\|_{H^m} \leq C^{1+m}(1+m)^{pm},
\end{equation}
for every $m=0,1,\ldots$
\end{defi}
A different ---but very much related--- definition is given by the following norm:
\begin{equation*}
\|u\|_{\mathcal G_p}^2=\sum_{m=0}^{\infty} \frac{C^m}{m!}\int(1+|\xi|^2)^{m/p}|\widehat u(\xi)|^2\,d\xi.
\end{equation*}

We will show first that,  for $n=1$, $\mathds L^2_\dag$ is a fractional version of a Gevrey--Sobolev space.
For this, we remark that, formally,
\begin{equation*}
\|u\|_\dag^2=\sum_{m=0}^\infty \frac{1}{4^m m!}\int_0^\infty \rho^{2m}|u^\dag(\rho^s)|^2\rho^{n-1}\,d\rho.
\end{equation*}
In Fourier variables, restricting to $n=1$ for simplicity,
\begin{equation*}
\int_0^\infty
\rho^{2m}|u^\dag(\rho^s)|^2\,d\rho
=\int_0^\infty \left|\frac{d^m}{d\vartheta^m}u^\sharp(\vartheta)\right|^2\,d\vartheta
=\int_0^\infty \left|\frac{d^m}{d\vartheta^m}\widehat{u}(\vartheta^{1/s})\right|^2\,d\vartheta,
\end{equation*}
and (using the change $\vartheta^{1/s}=\zeta$)
\begin{equation*}
\frac{d}{d\vartheta}\widehat u(\vartheta^{1/s})=\frac{1}{s}\zeta^{1-s}\frac{d}{d\zeta}\widehat u(\zeta),
\quad\ldots\quad
\frac{d^m}{d\vartheta^m}\widehat{u}(\vartheta^{1/s})
=\frac{1}{s^m}\left(\zeta^{1-s}\frac{d}{d\zeta}\right)^m\widehat u(\zeta).
\end{equation*}
%
This yields
\begin{equation}\label{Zemannian1}
\begin{split}
\|u\|_\dag^2&=\sum_{m=0}^\infty \frac{1}{4^m m!s^{2m}}\int_0^\infty \left|\left(\zeta^{1-s}\frac{d}{d\zeta}\right)^m\widehat u(\zeta)\right|^2\,s\zeta^{s-1}\,d\zeta\\
&=\sum_{m=0}^\infty \frac{1}{4^m m!s^{2m}}\int_0^\infty \left|\mathcal F \left\{\big((-\Delta)^{\frac{1-s}{2}} r\big)^m u\right\}(\zeta)\right|^2\,s\zeta^{s-1}\,d\zeta
\end{split}
\end{equation}
which, inspired by \eqref{Gevrey}, could be taken as our definition of a \emph{fractional Gevrey--Sobolev} norm.

As a side remark, the \emph{Zemanian} space $\mathcal Z_\nu$ is defined as the set of $\mathcal C^\infty$ functions $\omega(\zeta)$, $\zeta>0$, such that all the quantities
\begin{equation*}
\gamma_{m,k}(\omega):=\sup_{\zeta>0}\left|\zeta^m\left(\frac{1}{\zeta}\frac{d}{d\zeta}\right)^k(\zeta^{-\nu-1/2}\omega(\zeta))\right|,\quad m,k\in\mathbb N,
\end{equation*}
are finite (compare to our \eqref{Zemannian1}). Zemanian spaces are the appropriate function spaces to define the distributional Hankel transform as an automorphism (see, for instance, the note \cite{Stempak} and the references therein).

\medskip

We will also give the interpretation of the   space $\mathds L^2_\dag$ in Mellin variables. For this, we take  again a Taylor expansion of the exponential weight. The trick is that powers of $r$ become translations in Mellin variables, instead of derivatives. Set
\begin{equation*}\label{sigma-m}
\sigma_m:=sm+\frac{n(2-s)}{2}.
\end{equation*}
Then the $\mathds L^2_\dag$ norm is written as
\begin{equation*}\label{norm-Melin}\begin{split}
\|u\|^2_\dag
&=\int_0^\infty
	|\Phi_s^{-1}u(r)|^2
	\,e^{\frac{r^{2s}}{4}}
	r^{n(2-s)-1}
\,dr=\sum_{m=0}^\infty \frac{1}{4^m m!}\int_0^\infty|\Phi_s^{-1}u(r)|^2
\, r^{2sm+n(2-s)-1}\,dr\\
&=\sum_{m=0}^\infty \frac{1}{4^m m!}\int_{\R}
	|\cM \Phi_s^{-1}u(\sigma_m+\lambda i)|^2
\,d\lambda
=\sum_{m=0}^\infty \frac{1}{4^m m!}\int_{\R}
	|\Lambda_s^{-1}(\sigma_m+\lambda i)|^2|\cM u(\sigma_m+\lambda i)|^2
\,d\lambda
\end{split}
\end{equation*}
where, for each $m$, we have used the isometry property of the Mellin transform \eqref{eq:M-isom-gen}.

Note that the integrand above may have poles (and also zeroes), which may require a higher order vanishing of $\mathcal M u$. These can be avoided, for instance, if $s$ is irrational. However, to obtain a  uniform bound for all $m$ is more delicate. We will show in  Section \ref{subsection:technicalities} that
\begin{equation}\label{horrible-estimate}
\abs{
	\Lambda_s^{-1}(\sigma_m+i\lambda)
}
	\asymp \delta_m+|\lambda|^{p_m},
\end{equation}
where $p_m$ is given in \eqref{pm} and $\delta_m$ in  \eqref{delta-m}. Looking at the asymptotic behavior as  $m\to \infty$ we conclude that
\begin{equation*}\label{Mellin-Gevrey}
\|u\|^2_\dag\asymp\sum_{m=0}^\infty \frac{1}{4^m m!}\int_{\mathbb R} (\delta_m+|\lambda|^{2m(1-s)})|\mathcal M u(sm+i\lambda)|^2\,d\lambda,
\end{equation*}
which again reminds us of a fractional Gevrey--Sobolev norm.

\medskip

Likewise, one may give an interpretation for  the  $\mathds H_\dag^1$ seminorm  \eqref{defi:seminorm}.
 For this, take the Taylor expansion of the exponential in the first step and then use the Mellin isometry property \eqref{eq:M-isom-gen},
\begin{equation*}
\begin{split}
\|u\|^2_{\dot{\mathds H}^1_\dag}&=\sum_{m=0}^\infty\frac{1}{4^m m!}\int_0^\infty |\partial_\rho u^\dag|^2 \rho^{2m+n-1}\,d\rho\\
&=\sum_{m=0}^\infty\frac{1}{4^m m!} \int_0^\infty \left|\mathcal M(\partial_\rho u^\dag)(m+\tfrac{n}{2}+\lambda i)\right|^2\,d\lambda.
\end{split}
\end{equation*}
Now observe that
\begin{equation*}
\mathcal M(\partial_\rho u^\dag)(z)=-(z-1)\mathcal M(u^\dag)(z-1),
\end{equation*}
so that
\begin{equation*}
\begin{split}
\|u\|^2_{\dot{\mathds H}^1_\dag}&=
\sum_{m=0}^\infty\frac{1}{4^m m!} \int_0^\infty \left\{\lambda^2+(m+\tfrac{n}{2}-1)^2\right\}\left|\mathcal M\{u^\dag(r)\}(m+\tfrac{n}{2}-1+\lambda i)\right|^2\,d\lambda\\
&=s\sum_{m=0}^\infty\frac{1}{4^m m!} \int_0^\infty \left\{\lambda^2+(m+\tfrac{n}{2}-1)^2\right\}
\left|\mathcal M\{u^\dag(r^s)\}\left(s(m+\tfrac{n}{2}-1+\lambda i)\right)\right|^2\,d\lambda\\
&=\sum_{m=0}^\infty\frac{1}{4^m m!} \int_0^\infty \left\{(\tfrac{\tilde\lambda}{s})^2+(m+\tfrac{n}{2}-1)^2\right\}
\left|\mathcal M\{u^\dag(r^s)\}\left(s(m+\tfrac{n}{2}-1)+\tilde\lambda i\right)\right|^2\,d\tilde\lambda
\end{split}
\end{equation*}
where we have used \eqref{eq:M-rescale} for the second equality and the change $\tilde\lambda=s\lambda$ for the third. Next, we drop tilde notation,  use \eqref{eq:M-trans} to translate and take into account relation \eqref{simple-relation},
\begin{equation*}
\begin{split}
\mathcal M\{u^\dag(r^s)\}\left(s(m+\tfrac{n}{2}-1)+\lambda i\right)&=
4^{\frac{n}{2}-\frac{n}{2s}}\mathcal M\{(\tfrac{r^{2s}}{4})^{\frac{n}{2}-\frac{n}{2s}}u^\dag(r^s)\}\left(\sigma_m-s+\lambda i\right)\\
&=4^{\frac{n}{2}-\frac{n}{2s}}\mathcal M\{\Phi^{-1}_s u\}\left(\sigma_m-s+\lambda i\right)\\
&=4^{\frac{n}{2}-\frac{n}{2s}} \Lambda_s^{-1}(\sigma_m-s+\lambda i)\mathcal Mu(\sigma_m-s+\lambda i).
\end{split}
\end{equation*}
We conclude that, up to multiplicative constant,
\begin{equation*}
\begin{split}
\|u\|^2_{\dot{\mathds H}^1_\dag}
&=\sum_{m=0}^\infty\frac{1}{4^m m!} \int_0^\infty \left\{(\tfrac{\lambda}{s})^2+(m+\tfrac{n}{2}-1)^2\right\}
\left|\Lambda_s^{-1}(\sigma_m-s+\lambda i)|^2\,|\mathcal Mu(\sigma_m-s+\lambda i)\right|^2\,d\lambda,
\end{split}
\end{equation*}
and then we can insert the estimate for $\Lambda_s^{-1}$  analogous to \eqref{horrible-estimate}.

We note here that $\|\cdot\|_{\dot{\mathds H}^1_\dag}$ is not our usual Sobolev $H^s$ fractional seminorm. This is a consequence of our construction and is partly  explained in Remark \ref{remark:interpret}.

\subsection{Technicalities}\label{subsection:technicalities}

Here we obtain upper and lower bounds for the multiplier $\Lambda_s^{-1}(\sigma_m+i\lambda)$, understanding carefully the dependence of $m$ as $m\to \infty$.

First, 
from \eqref{eq:multiplier-asymp}
we can check that, when $m\to \infty$, for $|\lambda|\gg 1$,
\begin{equation*}
\abs{
	\Lambda_s^{-1}(\sigma_m+i\lambda)
}
=2^{-(\frac1s-1)n}
	s^{-m}
	|\lambda|^{(1-s)(\frac{n}{2}-m)}
	(1+O(|\lambda|^{-1}))
\asymp
s^{-m}|\lambda|^{p_m},
\end{equation*}
where
\begin{equation}\label{pm}
p_m=n-\tfrac{n}{s}-\sigma_m+\tfrac{\sigma_m}{s}=(1-s)(m-\tfrac{n}{2}).
\end{equation}

However, the expression of $\Lambda_s$ may have poles/zeroes when $\lambda=0$. The following Lemma characterizes precisely its behavior at those points:

\begin{lemma} For each $a\in \mathbb R$ fixed, we have for $m\in\bN$ large,
\begin{equation}\label{delta-m}
\Lambda_s^{-1}(sm+a)
\asymp
\delta_m
:=\frac{
	\sin (
		\frac{\pi(n-a)}{2s}
		-\frac{m\pi}{2}
	)
}{
	\sin (
		\frac{\pi(n-a)}{2}
		-\frac{sm\pi}{2}
	)
}
\left(
	\frac{m}{es^{\frac{s}{1-s}}}
\right)^{(\frac1s-1)(sm+a-n)}.
\end{equation}
\end{lemma}

\begin{proof}
Recall from \eqref{multiplier-Lambda} that
\begin{equation*}
\Lambda_s^{-1}(sm+a)
=2^{(\frac1s-1)(sm+a)}
\frac{
	\Gamma\left (
		\frac{
			n-a-sm
		}{2}
	\right )
}{
	\Gamma\left (
		\frac{
			sm+a
		}{2}
	\right )
}
\frac{
	\Gamma(
		\tfrac{n}{2}
		-\tfrac{n-a-sm}{2s}
	)
}{
	\Gamma(
		\tfrac{n-a-sm}{2s}
	)
}.
\end{equation*}
We use the relation $\Gamma(z)\Gamma(1-z)=\frac{\pi}{\sin(\pi z)}$ to have all the arguments inside the Gammas positive:
\begin{equation*}
\Lambda_s^{-1}(sm+a)
=2^{(\frac1s-1)(sm+a)}
\frac{
	\sin (\pi(\tfrac{n-sm-a}{2s}))
}{
	\sin (\pi(\tfrac{n-sm-a}{2}))
}
\frac{
	\Gamma(
		\tfrac{n}{2}-\tfrac{n-a}{2s}+\tfrac{m}{2}
	)
	\Gamma(
		1-\tfrac{n-a}{2s}
		+\tfrac{m}{2}
	)
}{
	\Gamma\left (
		\tfrac{a}{2}
		+\tfrac{sm}{2}
	\right )
	\Gamma\left (
		1-\tfrac{
			n-a
		}{2}
		+\tfrac{sm}{2}
	\right )
}
=:\delta_m^{(1)}\delta_m^{(2)}.
\end{equation*}
where we have defined
\begin{equation*}
\delta_m^{(1)}:=\frac{\sin (\pi(\tfrac{n-sm-a}{2s}))}{\sin (\pi(\tfrac{n-sm-a}{2}))},
	\qquad
\delta_m^{(2)}:=\Lambda_s^{-1}(sm+a)(\delta_m^{(1)})^{-1}.
\end{equation*}
Finally, expression \eqref{asymptotics-gamma0} for the Gamma function immediately gives
\begin{align*}
\delta_m^{(2)}
&=2^{(\frac1s-1)(sm+a)}
\dfrac{
	(\frac{m}{2})^{\frac{m}{2}+\frac{n}{2}-\frac{n-a}{2s}-\frac12}
	e^{-\frac{m}{2}-\frac{n}{2}+\frac{n-a}{2s}}
	\cdot
	(\frac{m}{2})^{\frac{m}{2}+1-\frac{n-a}{2s}-\frac12}
	e^{-\frac{m}{2}-1+\frac{n-a}{2s}}
}{
	(\frac{sm}{2})^{\frac{sm}{2}+\frac{a}{2}-\frac12}
	e^{-\frac{sm}{2}-\frac{a}{2}}
	\cdot
	(\frac{sm}{2})^{\frac{sm}{2}+1-\frac{n-a}{2}-\frac12}
	e^{-\frac{sm}{2}-1+\frac{n-a}{2}}
}
(1+O(m^{-1}))\\
&=2^{(\frac1s-1)n}
s^{-\frac{n}{2}+(n-sm-a)}
e^{(\frac1s-1)(n-sm-a)}
m^{(\frac1s-1)(sm+a-n)}
(1+O(m^{-1}))\\
&=2^{(\frac1s-1)n}s^{-\frac{n}{2}}
\left(
	\frac{m}{es^{\frac{s}{1-s}}}
\right)^{(\frac1s-1)(sm+a-n)}
(1+O(m^{-1})).
\end{align*}
This completes the estimate of $\delta_m$.
\end{proof}

We remark here that while $\delta_m^{(1)}$ has a trivial lower bound independent of $m$, estimating its upper bound is more delicate since, as $m\to\infty$, 
there may be a sequence of $k=k_m$ of even numbers such that
$n-ms-a\to -k\in\mathbb Z$, even if $s$ is irrational. 
We may try to give a sharper estimate for $\delta_m^{(1)}$, however this is a problem in the field of number theory which we do not attempt to consider here.

In any case, let us give here a flavour of the argument involved and consider the toy model  of giving a lower found for $\sin(ms)$ when  $s\not\in \mathbb Q$. It may happen that  $ms\to k\in\mathbb Z$, that is $\frac{k}{m}\to s\not\in \mathbb Q$. Thus we are looking at how to approximate an irrational number by rational numbers. If $\mu$ is the \emph{irrationality measure} of $s$, then fixed $\varepsilon>0$, for all sequences of $m$, $k$ large enough,
\begin{equation*}
\left|\frac{k}{m}-s\right|>\frac{1}{m^{\mu+\varepsilon}},
\end{equation*}
so that
\begin{equation}\label{irrationality-bound}
\sin(\pi(ms))>\frac{1}{m^{\mu+\varepsilon-1}}.
\end{equation}
For all $s$  it holds that $\mu\geq 2$. Almost every irrational have $\mu=2$, and this includes all algebraic numbers of degree $ >1$. However, there may be real numbers with
infinite $\mu$ (these are called Liouville numbers), for which an estimate such as \eqref{irrationality-bound} does not hold. For some background on this topic, see \cite{Borwein}.

\section{The eigenvalue problem via Mellin transform}\label{section:eigenvalue}

In this section we complete the proof of Theorem \ref{thm1}. The idea is that the eigenvalue problem \eqref{eigenvalue-problem} can be explicitly solved in terms of Mellin variables, for functions in the space $\mathds L^2_\dag(s)$. Although in the previous section we have strongly relied on the transformation \eqref{change} and the information available from the local case, now we will calculate the asymptotic behavior of all solutions \emph{without} any previous knowledge.




\subsection{Explicit eigenfunctions under transformations}\label{subsection:explicit-eigenfunctions}

 We take the Mellin transform of \eqref{eigenvalue-problem} to get
\begin{equation}\label{eq:20}
\cM\set{\Ds u(r)}(z)
=\Big(\nu+\frac{n}{2s}\Big)
\cM\set{u(r)}(z)
+\frac{1}{2s}
\cM\set{
	ru'(r)
}(z).
\end{equation}
By \Cref{prop:M-Ds} and \eqref{eq:M-elem-1},
\[
\Theta_s(z)\cM u(z-2s)
=(\nu+\tfrac{n-z}{2s})\cM u(z).
\]
Hence, recalling the symbol given in \eqref{eq:M-Ds-sym}, 
equation \eqref{eq:20} yields
\begin{equation}\label{eq:eigen-M}
2^{2s}
\dfrac{
	\Gamma\left (
	\frac{
		z
	}{2}
	\right )
}{
	\Gamma\left (
	\frac{
		n-z
	}{2}
	\right )
}
\dfrac{
	\Gamma\left (
	\frac{
		n+2s-z
	}{2}
	\right )
}{
	\Gamma\left (
	\frac{
		z-2s
	}{2}
	\right )
}
\cM u(z-2s)
=\left (
	\nu+\tfrac{n-z}{2s}
\right )
\cM u(z),
	\quad z\in\bC.
\end{equation}
A solution to this equation will be a (tentative) eigenfunction with eigenvalue $\nu$.

\begin{lemma}[In Mellin variable]\label{lemma:eigenfunction1}
For each $\nu\in\mathbb R$, the function
\begin{equation}\label{eq:eigen-M-sol}
	\widetilde{u}_\nu(z)
=
2^{z}
\frac{
	\Gamma\left (
	\nu+\tfrac{n-z}{2s}
	\right )
	\Gamma\left (
	\tfrac{z}{2}
	\right )
}{
	\Gamma\left (
	\tfrac{n-z}{2}
	\right )
}.
\end{equation}
is a solution of \eqref{eq:eigen-M}.
\end{lemma}

\begin{remark}\label{remark:uniqueness}
We will show in \Cref{prop:uniq} that this solution is unique in $\mathds L^2_\dag$ for $n\geq 2$, and for $n=1$ provided that $\nu\in\bZ$.
\end{remark}

\begin{proof}[Proof of \Cref{lemma:eigenfunction1}]
The essence is to observe the $2s$-periodicity of an auxiliary function. Rewriting \eqref{eq:eigen-M} as
\begin{equation*}
2^{-(z-2s)}
2^{z}
\dfrac{
	\Gamma\left (
	\frac{
		z
	}{2}
	\right )
}{
	\Gamma\left (
	\frac{
		n-z
	}{2}
	\right )
}
\dfrac{
	\Gamma\left (
	\frac{
		n-(z-2s)
	}{2}
	\right )
}{
	\Gamma\left (
	\frac{
		z-2s
	}{2}
	\right )
}
\cM u(z-2s)
=
\dfrac{
	\Gamma\left (
		\nu+\tfrac{n-(z-2s)}{2s}
	\right )
}{
	\Gamma\left (
		\nu+\tfrac{n-z}{2s}
	\right )
}
\cM u(z),
\end{equation*}
we realize that 
\begin{equation*}
v(z)
:=2^{-z}
\dfrac{
	\Gamma\left (
		\tfrac{n-z}{2}
	\right )
}{
	\Gamma\left (
		\tfrac{z}{2}
	\right )
	\Gamma\left (
		\nu+\tfrac{n-z}{2s}
	\right )
}
\cM u(z),
	\quad z\in\bC,
\end{equation*}
satisfies
\begin{equation*}
v(z-2s)=v(z),
	\quad z\in\bC.
\end{equation*}
A solution of this equation is,  obviously, the constant one, and this completes our claim.
\end{proof}

Now we would like to take inverse Mellin transform to define the eigenfunction  $u_\nu(r)$ in Euclidean variable.  For this,
note that from \eqref{estimate-Gamma-2} and \eqref{estimate-Gamma-2s} we know that $|\widetilde{u}_{\nu}(\sigma+i\lambda)| \leq C_1(\sigma)|\lambda|^{C_2(\sigma)}e^{-\frac{\pi}{4s}|\lambda|}$ as $|\lambda|\to\infty$.  Then  one may define $u_\nu(r)$ for each $\nu\in\mathbb R$ by taking the inverse Mellin transform of \eqref{eq:eigen-M-sol}.   Uniqueness of the inverse is an issue only if $\nu\in\mathbb N$ (since other values do not provide admissible eigenfunctions); this will be considered in Section  \ref{subsection:uniqueness}.\\

Next, let us prove an interpretation of formula \eqref{eq:eigen-M-sol} in terms of Fourier (Hankel) transform, and which brings us back to \eqref{eigen-int}.

\begin{prop}[In Fourier variable]
\label{prop:fractional-Laguerre}
	The Hankel transform of the function $u_\nu(r)$ in $\mathbb R^n$,  for $\nu>-\frac{n-1}{4s}$, is given by
	\begin{equation*}
		\widehat{u}_\nu(\zeta)
		=\cF\set{u_\nu(r)}(\zeta)
		=2^{\frac{n}{2}+1}s\,
		e^{-\zeta^{2s}}
		\zeta^{2s\nu}.
	\end{equation*}
\end{prop}

\begin{proof}
On the one hand, for ${\rm Re\,}(\nu+\tfrac{n-z}{2s})>0$, or ${\rm Re\,}z < n+2s\nu$,
\begin{align*}
\cM\set{
	2^{\frac{n}{2}+1}s\,e^{-\zeta^{2s}}
	\zeta^{2s\nu}
}(n-z)
&=2^{\frac{n}{2}}
\int_{0}^{\infty}
	e^{-\zeta^{2s}}
	\zeta^{2s\nu+n-z-2s}
	(2s\,\zeta^{2s-1}\,d\zeta)\\
&=2^{\frac{n}{2}}
\int_{0}^{\infty}
	e^{-\varrho}
	(\varrho^{\frac{1}{2s}})^{
		2s\nu+n-z-2s
	}
\,d\varrho
=2^{\frac{n}{2}}
\Gamma\left (
	\nu+\tfrac{n-z}{2s}
\right ).
\end{align*}
On the other hand, by \eqref{eq:MF-M} and \eqref{eq:eigen-M-sol}, for $0<{\rm Re\,}(n-z)<\frac{n+1}{2}$, or $\frac{n-1}{2}<{\rm Re\,}z<n$,
\begin{align*}
\cM\set{
	\cF u_\nu(\zeta)
}(n-z)
=2^{\frac{n}{2}-z}
\frac{
	\Gamma\left (
	\frac{
		n-z
	}{2}
	\right )
}{
	\Gamma\left (
	\frac{
		z
	}{2}
	\right )
}
\cM u_\nu(z)
=2^{\frac{n}{2}}
\Gamma\left (
	\nu+\tfrac{n-z}{2s}
\right ).
\end{align*}
By \Cref{cor:uniq} (on a strip on the right of ${\rm Re\,}z=\frac{n-1}{2}$), the result follows.
%
\end{proof}

Now we go back to the formula \eqref{eq:eigen-M-sol} for the (tentative) eigenfunction, and introduce the multiplier $\Lambda_s$  defined in \eqref{multiplier-Lambda}. For this we write
\begin{equation}\label{splitting}
\widetilde{u}_{\nu}(z)
:=
2^{z}
\frac{
	\Gamma\left (
	\nu+\tfrac{n-z}{2s}
	\right )
	\Gamma\left (
	\tfrac{z}{2}
	\right )
}{
	\Gamma\left (
	\tfrac{n-z}{2}
	\right )
}
=:\Lambda_s(z) \, \cU_{\nu}(z),
\end{equation}
where we have defined the function
\begin{equation}
	\label{eq:U-nu}
\cU_{\nu}(z):=2^{\frac{z}{s}}\dfrac{
	\Gamma(\nu+\tfrac{n-z}{2s})
}{
	\Gamma(\tfrac{n-z}{2s})
}
\Gamma(\tfrac{n}{2}-\tfrac{n}{2s}+\tfrac{z}{2s}).
\end{equation}
The splitting \eqref{splitting} has some advantages. Indeed, when $\nu\in\mathbb N$, the quotient of two Gamma functions in $\cU_\nu(z)$ simplifies to a rising factorial, allowing to get an explicit formula for its inverse Mellin transform in Lemma \ref{lem:U-k}. This is actually the reason for the introduction of the multiplier $\Lambda_s$ (and the associated $\Phi_s^{-1}$) in the definition of the norm $\|\cdot\|_\dag$.

We will show in Section \ref{sec:M-residue-new} below that eigenfuctions for the eigenvalue problem \eqref{eigenvalue-problem} that belong to the space $\mathds L^2_\dag$ are only those $\widetilde u_\nu$ with $\nu=k$, $k\in\mathbb N$ (indeed, look at the asymptotics from \eqref{U-asymptotics}).

\subsection{Fractional analogue of Hermite polynomials}\label{subsection:fractional-Hermite}

We are interested in the inverse Mellin transform (particularly on the line $\re z=\frac{n}{2}$) of the function $\cU_{\nu}$  given in \eqref{eq:U-nu}. While an explicit expression can be obtained via residues (\Cref{sec:M-residue-new}), we first present an more elementary approach for the special case where $\nu\in\bN$.

\begin{lemma}\label{lem:U-k}
For $\nu_k=k\in\mathbb N$ it holds that
	\begin{equation}\label{notation-eigenfunction}
		\cU_k(z)
		=\cM\set{\cL_k(\cdot)}(z),
	\quad \text{ for } \re(\tfrac{n-z}{2s})<\tfrac{n}{2},
	\end{equation}
	where we have defined
	\begin{equation}\label{eq:Hermite-M-inv}
		\cL_k(r):=
		2s\,k!\,
		\left (
			\tfrac{r^{2s}}{4}
		\right )^{\frac{n}{2}-\frac{n}{2s}}
		e^{-\frac{r^{2s}}{4}}
		L_k^{(\frac{n-2}{2})}
		\left (
			\tfrac{r^{2s}}{4}
		\right ).
	\end{equation}
Here $L_k^{(\alpha)}$ is the generalized Laguerre polynomial of degree $k$ with parameter $\alpha$, as defined in \Cref{section:Hermite}.
\end{lemma}

\begin{proof}
The idea is that the rising factorial translates to derivatives of a monomial under Mellin transform. 
Keeping this in mind, using the so-called Rodrigues formula for Laguerre polynomials
\begin{equation*}
	L_k^{(\alpha)}( r )
	=\frac{1}{k!}
	 r ^{-\alpha}
	e^ r
	\frac{d^k}{d r ^k}
	( r ^{\alpha+k}e^{- r }),
\end{equation*}
we compute, for $\alpha\in \R$,
\begin{align*}
\cM\set{
	k!\,e^{- r }L_k^{(\alpha)}( r )
}(y)
&=\cM\set{
	k!\, r ^{\alpha}e^{- r }L_k^{(\alpha)}( r )
}(y-\alpha)
=\cM\set{
	\dfrac{d^k}{d r ^k}
	( r ^{\alpha+k}e^{- r })
}(y-\alpha)\\
&=\int_{0}^{\infty}
	 r ^{y-\alpha-1}
	\dfrac{d^k}{d r ^k}
	( r ^{\alpha+k}e^{- r })
\,d r
=(-1)^k\int_{0}^{\infty}
	\dfrac{d^k}{d r ^k}
	( r ^{y-\alpha-1})
	 r ^{\alpha+k}e^{- r }
\,d r \\
&=
(-1)^k (y-\alpha-1)
\cdots (y-\alpha-k)
\int_{0}^{\infty}
	 r ^{y-1}e^{- r }
\,d r
\\
&=\dfrac{
	\Gamma(k+\alpha+1-y)
}{
	\Gamma(\alpha+1-y)
}
\Gamma(y),
\end{align*}
for $\re y>0$. Choosing  $\alpha=\frac{n-2}{2}$,
\begin{equation*}
\cM\set{
	k!\,
	e^{- r }
	L_k^{(\frac{n-2}{2})}( r )
}(y)
=\dfrac{
	\Gamma(k+\frac{n}{2}-y)
}{
	\Gamma(\frac{n}{2}-y)
}
\Gamma(y).
\end{equation*}
Now we make the change of variable
\begin{equation*}
\label{change-variable}\tfrac{n-z}{2s}=\tfrac{n}{2}-y,
\end{equation*}
which implies
\begin{equation*}
\cM\set{
	k!\,
	e^{- r }
	L_k^{(\frac{n-2}{2})}( r )
}(\tfrac{n}{2}-\tfrac{n}{2s}+\tfrac{z}{2s})
=\dfrac{
	\Gamma(k+\tfrac{n-z}{2s})
}{
	\Gamma(\tfrac{n-z}{2s})
}
\Gamma(\tfrac{n}{2}-\tfrac{n}{2s}+\tfrac{z}{2s}),
	\quad \re(\tfrac{n-z}{2s})<\tfrac{n}{2}.
\end{equation*}
Recalling \eqref{eq:U-nu} with $\nu=k$ we obtain
\begin{equation*}
\cU_{k}(z)
=2^{\frac{z}{s}}
\cM\set{
	k!\,
	e^{- r }
	L_k^{(\frac{n-2}{2})}( r )
}(\tfrac{n}{2}-\tfrac{n}{2s}+\tfrac{z}{2s}).
\end{equation*}
Now, using \eqref{eq:M-mult-exp}--\eqref{eq:M-rescale},
\begin{align*}
\cU_{k}(z)
&=2^{\frac{z}{s}}
	\cM\set{
		2s\,k!\,
		e^{- r ^{2s}}
		L_k^{(\frac{n-2}{2})}( r ^{2s})
	}(ns-n+z)\\
&=2^{\frac{z}{s}}
	\cM\set{
		2s\,k!\,
		 r ^{-n(1-s)}
		e^{- r ^{2s}}
		L_k^{(\frac{n-2}{2})}( r ^{2s})
	}(z)\\
&=
	\cM\set{
		2s\,k!\,
		2^{\frac{n(1-s)}{s}}
		 r ^{-n(1-s)}
		e^{-\frac{ r ^{2s}}{4}}
		L_k^{(\tfrac{n-2}{2})}(\tfrac{ r ^{2s}}{4})
	}(z),
\end{align*}
for $\re(\tfrac{n-z}{2s})<\tfrac{n}{2}$, as desired.
\end{proof}

 Recalling the splitting \eqref{splitting}, we will denote the eigenfunctions for \eqref{eigenvalue-problem}  by
\begin{equation*}
e_k(r):=\Phi_s \cL_k(r), \quad k\in\mathbb N.
\end{equation*}
%
In addition, thanks to the orthogonality properties of the Laguerre polynomials, we have:

\begin{lemma} \label{lemma:orthogonal}
	Any two eigenfunctions $e_k$, $e_j$, $k\neq j$,
	are orthogonal with respect to $\angles{\cdot,\cdot}_{\dag}$.
\end{lemma}

\begin{proof}
Substituting the formula for $\cL_j(r)$ from \eqref{eq:Hermite-M-inv} and changing variable to
$\rho=\frac{r^{2s}}{4}$ we have
\begin{equation*}
\begin{split}
	\langle e_j,e_k\rangle_\dag
	&=\int_0^\infty
		\Phi_s^{-1}e_j(r)\,
		\overline{\Phi_s^{-1}e_k(r)}
		\,e^{\frac{r^{2s}}{4}}
		r^{n(2-s)-1}\,dr
	=
	\int_0^\infty
		\cL_j(r)\,\cL_k(r)
		\,e^{\frac{r^{2s}}{4}}
		r^{n(2-s)-1}\,dr\\
&=2s\, j!k!\,2^{\frac{n}{s}}
	\int_0^\infty
		L_j^{(\frac{n-2}{2})}(\rho)
		L_k^{(\frac{n-2}{2})}(\rho)
		\,e^{-\rho}
		\rho^{\frac{n}{2}-1}
	\,d\rho
=2s(k!)^2 \, 2^{\frac{n}{s}}
\frac{
	\Gamma(k+\frac{n}{2})
}{
	k!\Gamma(\frac{n}{2})
}
\delta_{k,j}.
\end{split}
\end{equation*}
In the last equality we have used  \eqref{Laguerre-orthogonality}, taking into account that Laguerre polynomials $\{L^{(\alpha)}_k(\rho)\}_k$ are an orthogonal family with respect to the weight $\rho^{\alpha}e^{-\rho}$, for $\alpha=\frac{n-2}{2}$.\\
\end{proof}

%


\subsection{Mellin inversion and Kummer functions}\label{sec:M-residue-new}

The ultimate goal of this Subsection is to prove \eqref{asymptotics-introduction}, which motivates the choice of function spaces. Indeed,  in these asymptotics we have exponential decay if and only if $\nu\in \mathbb N$. For this,  we need to compute the inverse Mellin transform of $\cU_\nu(z)$, denoted by $U_\nu$, for $\nu\in\mathbb R$. This is the most technical part of the paper.

\begin{thm}\label{thm:u-nu}
Assume that  $\nu+\tfrac{n}{2}$ is not a non-positive integer. Then
	\begin{equation*}
		\cU_\nu(z)
		=
		\cM\set{U_\nu(r)}(z),
	\end{equation*}
for
\begin{equation}\label{U_nu}
\begin{split}
U_\nu(r)
&=2s\dfrac{
	\Gamma\left (
	\nu+\frac{n}{2}
	\right )
}{
	\Gamma\left (
	\frac{n}{2}
	\right )
}
\left (
	\frac{r^{2s}}{4}
\right )^{\frac{n}{2}-\frac{n}{2s}}
e^{-\frac{r^{2s}}{4}}
M\left (
	-\nu;
	\frac{n}{2};
	\frac{r^{2s}}{4}
\right )\\
&\qquad
-2s\sum_{\ell=0}^{\floors{-\nu-\frac{n}{2}}}
\dfrac{
	\sin(\pi\nu)
}{
	\pi
}
\dfrac{
	\Gamma(\nu+1+\ell)
	\Gamma(\nu+\tfrac{n}{2}+\ell)
}{\ell!}
\left (
	\frac{r^{2s}}{4}
\right )^{-\nu-\frac{n}{2s}-\ell}.
\end{split}
\end{equation}
Moreover, we have the following asymptotic behavior  as $r\to+\infty$,
\begin{equation}\label{U-asymptotics}
\begin{split}
&U_k(r)
\sim (-1)^k \, 2s\,e^{-\frac{r^{2s}}{4}}
	\left (
		\frac{r^{2s}}{4}
	\right )^{k+\frac{n}{2}-\frac{n}{2s}}
\quad \text{ for } \quad
\nu=k\in\bN,\\
&U_\nu(r)
\asymp
	\left (
		\frac{r^{2s}}{4}
	\right )^{-\nu-\frac{n}{2s}}
\quad \text{ for } \quad
\nu\notin \bN \quad \text{ and } \nu \leq -\frac{n}{2}-1.
\end{split}
\end{equation}
\end{thm}


\begin{remark}
If $\nu+\tfrac{n}{2}$ is a non-positive integer (in particular $\nu \leq -\tfrac{n}{2}<0$), then  $0,-1,\dots,\nu+\tfrac{n}{2}$ are all double poles with zero residue. Similar computations imply that
\[
U_\nu(r)
=2s
\sum_{j=-\nu-\frac{n}{2}+1}^{\infty}
\dfrac{
	\Gamma(\nu+\tfrac{n}{2}+j)
}{
	\Gamma(\tfrac{n}{2}+j)
}
\dfrac{(-1)^j}{j!}
\left (
	\frac{r^{2s}}{4}
\right )^{j+\frac{n}{2}-\frac{n}{2s}}.
\]
\end{remark}

\normalcolor

Recall that
\[
U_\nu(r)
=\cM_{2s\sigma+n(1-s)}^{-1}
\set{\cU_{\nu}(z)}(r)
=\cM_{2s\sigma+n(1-s)}^{-1}
\set{
	4^{\frac{z}{2s}}
	\frac{
		\Gamma\left (
		\nu+\tfrac{n-z}{2s}
		\right )
		\Gamma\left (
			\frac{n}{2}
			-\frac{n}{2s}
			+\tfrac{z}{2s}
		\right )
	}{
		\Gamma\left (
		\tfrac{n-z}{2s}
		\right )
	}
}(r).
\]
Here $\sigma\in\R$ is the real part of the vertical line along which the inverse Mellin transform is computed. It determines the asymptotic behavior of the resulting function as $r\to 0$ and as $r\to\infty$. For now we keep it free.

Note that for simplicity of notations later, we have written $2s\sigma+n(1-s)$ instead of $\sigma$ at the first place. In fact, using
\eqref{eq:M-inv-trans}, \eqref{eq:M-inv-exp} and \eqref{eq:M-inv-rescale}, we have
\begin{align*}
\cM_{2s\sigma+n(1-s)}^{-1}
\set{
	\cU_\nu(z)
}(r)
&=2s\,\cM_{\sigma+\frac{n}{2s}-\frac{n}{2}}^{-1}
\set{
	4^z
	\dfrac{
		\Gamma\left (
			\nu+\tfrac{n}{2s}-z
		\right )
		\Gamma(\tfrac{n}{2}-\tfrac{n}{2s}+z)
	}{
		\Gamma\left (
			\tfrac{n}{2s}-z
		\right )
	}
}(r^{2s})\\
&=2s\,\cM_{\sigma+\frac{n}{2s}-\frac{n}{2}}^{-1}
\set{
	\dfrac{
		\Gamma\left (
			\nu+\tfrac{n}{2s}-z
		\right )
		\Gamma(\tfrac{n}{2}-\tfrac{n}{2s}+z)
	}{
		\Gamma\left (
			\tfrac{n}{2s}-z
		\right )
	}
}\left (\frac{r^{2s}}{4}\right )\\
&=2s\,\left (
	\frac{r^{2s}}{4}
\right )^{\frac{n}{2}-\frac{n}{2s}}
\cM_{\sigma}^{-1}
\set{
	\dfrac{
		\Gamma\left (
			\nu+\tfrac{n}{2}-z
		\right )
		\Gamma(z)
	}{
		\Gamma\left (
			\tfrac{n}{2}-z
		\right )
	}
}\left (\frac{r^{2s}}{4}\right ).
\end{align*}
Therefore, for $b=\frac{n}{2}$ and any $\sigma$ such that $\sigma\notin -\bN$ and $\nu+b-\sigma \notin -\bN$, it suffices to compute the integral
\begin{align*}
I_{\sigma,\nu,b}(\rho):=
\cM_{\sigma}^{-1}
\set{
	\dfrac{
		\Gamma\left (
		\nu+b-z
		\right )
		\Gamma(z)
	}{
		\Gamma\left (
		b-z
		\right )
	}
}(\rho)
=\dfrac{1}{2\pi i}
\int_{\sigma-i\infty}^{\sigma+i\infty}
	\rho^{-z}
	\dfrac{
		\Gamma\left (
			\nu+b-z
		\right )
		\Gamma(z)
	}{
		\Gamma\left (
			b-z
		\right )
	}
\,dz,
\end{align*}
for $b=\frac{n}{2}$. Let us state the final result.

\begin{prop}\label{prop:I-alpha}
Let $b>0$, $\sigma,\nu\in\R$. If neither of $\sigma$, $\nu+b$, $\nu+b-\sigma$ is a non-positive integer, then
\begin{align*}
I_{\sigma,\nu,b}(\rho)
&=\dfrac{
	\Gamma\left (
		\nu+b
	\right )
}{
	\Gamma\left (
		b
	\right )
}
M\left (
	\nu+b;
	b;
	-\rho
\right )
-\sum_{j=0}^{\floors{-\sigma}}
	\dfrac{
	\Gamma(\nu+b+j)
}{
	\Gamma(b+j)
}
\dfrac{(-1)^j}{j!}
\rho^{j}\\
&\qquad
	-\sum_{\ell=0}^{\floors{-\nu-b+\sigma}}
\dfrac{
	\sin(\pi \nu)
}{
	\pi
}
\dfrac{
	\Gamma(\nu+1+\ell)
	\Gamma(\nu+b+\ell)
}{
	\ell!
}
\rho^{-\nu-b-\ell}.
	\end{align*}
\end{prop}

\begin{remark}\label{rmk:kummer}
We make the following remarks in the previous formula:
\begin{itemize}
\item 	We understand that the finite summation in $j$ is empty $\sigma>0$, and so is that in $\ell$ if $\nu+b-\sigma>0$.

\item The Kummer transformation states that
\[
M(a,b,z)=e^{z}M(b-a,b,-z).
\]
Thus the first term can be rewritten as
\[
\dfrac{
	\Gamma\left (
	\nu+b
	\right )
}{
	\Gamma\left (
	b
	\right )
}
e^{-\rho}
M\left (
	-\nu;
	b;
	\rho
\right ),
\]
showing that the alternating sign yields an exponential decay (this also justifies the choice of the symbol $b$.)

\item The generalized Laguerre polynomial is equal to the Kummer function with first parameter a non-positive integer:
\[
L_k^{(b-1)}(\rho)
=\dfrac{
	\Gamma(b+k)
}{
	\Gamma(b)k!
}
M(-k,b,\rho).
\]
Thus, when $\nu=k\in\bN$, putting $b=\tfrac{n}{2}$, we have 
\[
\dfrac{
	\Gamma\left (
		k+\frac{n}{2}
	\right )
}{
	\Gamma\left (
		\frac{n}{2}
	\right )
}
e^{-\rho}
M\left (
	-k;
	\tfrac{n}{2};
	\rho
\right )
=k!\,e^{-\rho}L_k^{\left (\tfrac{n-2}{2}\right )}(\rho).
\]
\item As $\rho\to+\infty$\footnote{In fact, the same is valid for $\rho\in\bC$ as $|\rho|\to\infty$ with $\tfrac{3\pi}{2}< \arg \rho \leq \tfrac{\pi}{2}$.},
\begin{align*}
\dfrac{
	\Gamma\left (
	\nu+b
	\right )
}{
	\Gamma\left (
	b
	\right )
}
e^{-\rho}
M\left (
	-\nu;
	b;
	\rho
\right )
&\sim
\Gamma\left (
	\nu+b
\right )
e^{-\rho}
\left (
	\dfrac{1}{\Gamma(-\nu)}
	e^{\rho} \rho^{-\nu-b}
	+\dfrac{1}{\Gamma(\nu+b)}
	(-\rho)^{\nu}
\right )\\
&\sim
	\dfrac{
		\Gamma(\nu+b)
	}{
		\Gamma(-\nu)
	}
	\rho^{-\nu-b}
	+e^{-\rho}(-\rho)^{\nu},
\end{align*}
Thus, $I_{\sigma,\nu,b}(\rho)$ has a exponential decay if and only if $\nu=k\in\bN$. Otherwise, $I_{\sigma,\nu,b}(\rho)$ has a power type behavior of order $\rho^{(-\nu-b) \vee \floors{-\sigma}}$.
\end{itemize}
\end{remark}

\begin{proof}[Proof of \Cref{thm:u-nu}]
Use \Cref{prop:I-alpha} with $\sigma\in(0,1)$ which is less than the distance between $\nu+b$ and $\bZ$.
\end{proof}

This $I_{\sigma,\nu,b}(\rho)$ is closely related to what is known as the Barnes--Mellin integral, where the poles of the Gamma functions in the numerator are contained in disjoint half planes. Its evaluation is based on a standard application of the residue theorem with the contour given by the a line segment on $\re z=\sigma$ and the (left) circular arc with radius $R\to+\infty$ where $\re z<\sigma$. More precisely, let
\[
\theta_\sigma=O\left (\frac{\sigma}{R}\right )
\]
be the angle of inclination from the imaginary axis such that $R\cos(\tfrac{\pi}{2}-\theta_\sigma)=\sigma$, $R\sin(\tfrac{\pi}{2}-\theta_\sigma)>0$. Note that $\theta_\sigma>0$ for $\sigma>0$. Let
\[
\cC=\cC_R \cup \cC_\sigma
\]
be a closed contour oriented counter-clockwise with
\begin{align*}
\cC_R
&=\set{
	z=Re^{i\theta}
	:\theta\in[
		\tfrac{\pi}{2}-\theta_\sigma,
		\tfrac{3\pi}{2}+\theta_\sigma
	]
},\\
\cC_\sigma
&=\set{
	z=\sigma+\lambda i
	:\lambda \in[
		R\sin(\tfrac{3\pi}{2}+\theta_\sigma),
		R\sin(\tfrac{\pi}{2}-\theta_\sigma)
	]
}.
\end{align*}
(Here all three paths depend on both $R$ and $\sigma$ despite their notations.) Denote the integrand by
\[
f(z)
=\rho^{-z}
\dfrac{
	\Gamma\left (
	\nu+b-z
	\right )
	\Gamma(z)
}{
	\Gamma\left (
	b-z
	\right )
}.
\]
By the residue theorem,
\begin{equation}\label{eq:I-alpha}
\begin{split}
I_\sigma(\rho)
&=\lim_{R\to\infty}
	\frac{1}{2\pi i}
	\int_{\cC_\sigma}
		f(z)
	\,dz\\
&=\lim_{R\to\infty}
\left (
	\sum_{z_j \text{ pole inside } \cC}
		\Res(f,z_j)
	-\frac{1}{2\pi i}
	\int_{\cC_R}
		f(z)
	\,dz
\right )
=:\lim_{R\to\infty}
(I_{\sigma,R,1}+I_{\sigma,R,2}).
\end{split}
\end{equation}

\begin{lem}\label{lem:I-alpha-1}
If $\nu+b$ is not a non-positive integer, then
\begin{align*}
I_{\sigma,R,1}
&=\sum_{j\in\bN \cap (-\sigma,R)}
	\rho^{j}
	\dfrac{
		\Gamma(\nu+b+j)
	}{
		\Gamma(b+j)
	}
	\dfrac{(-1)^j}{j!}\\
&\qquad
-\sum_{
	\ell\in\bN\cap(
		-\nu-b-R,
		-\nu-b+\sigma
	)
}
	\rho^{-\nu-b-\ell}
	\dfrac{
		\sin(\pi(\nu+\ell))
	}{
		\pi
	}
	\Gamma(\nu+1+\ell)
	\Gamma(\nu+b+\ell)
	\dfrac{(-1)^{\ell}}{\ell!}.
\end{align*}
\end{lem}

\begin{lem}\label{lem:I-alpha-2}
There exists a sequence $R_m\to+\infty$ such that
\[
\lim_{m\to\infty}
I_{\sigma,R_m,2}
=0.
\]
In fact, we can take $R_m=\frac12+2m$.
\end{lem}

To prove \Cref{lem:I-alpha-2} we need the following Lemma.

\begin{lem}\label{lem:abs-sin}
There exists a sequence $R_m\to+\infty$ such that for any $\theta\in[-\tfrac{\pi}{2}-\theta_\sigma,\tfrac{\pi}{2}+\theta_\sigma]$,
\[
|\sin(\pi R_m e^{i\theta})|
\geq ce^{\pi R_m|\sin\theta|},
\]
for an absolute constant $c>0$ that can be taken as $c=(2e)^{-1}$.
\end{lem}

\begin{proof}[Proof of \Cref{lem:abs-sin}]
Recalling the elementary formula $|\sin(x+iy)|^2=\sin^2(x)+\sinh^2(y)$ for $x,y\in\R$,
we have
\[
|\sin(\pi Re^{i\theta})|
=\sqrt{
	\sin^2(\pi R\cos\theta)
	+\sinh^2(\pi R\sin\theta)
}.
\]
When $\pi R|\sin\theta| \geq 1$,
\[
|\sin(\pi Re^{i\theta})|
\geq \sinh(\pi R|\sin\theta|)
\geq \tfrac12 e^{\pi R|\sin\theta|}.
\]
When $\pi R|\sin\theta| \leq 1$, $\theta=O(R^{-1})$ because $\theta\in[-\tfrac{\pi}{2}-\theta_\sigma,\tfrac{\pi}{2}+\theta_\sigma]$. Then
\[
\pi R\cos\theta
=\pi R(1+O(\theta^2))
=\pi R+O(R^{-1}).
\]
Thus, if we take $R_m=\tfrac12+2m \gg 1$ such that $\sin(\pi R_m)=1$, then
\[
\sin(\pi R\cos\theta)
=\sin(\pi R+O(R^{-1}))
=1+O(R^{-1})
\geq \tfrac12
\geq \tfrac{1}{2e}
	e^{\pi R|\sin\theta|}.
\qedhere
\]
\end{proof}

\begin{proof}[Proof of \Cref{lem:I-alpha-2}]
By the parameterization,
\[
I_{\sigma,R,2}
=\frac{1}{2\pi i}
\int_{\frac{\pi}{2}-\theta_\sigma}
^{\frac{3\pi}{2}+\theta_\sigma}
	\rho^{-Re^{i\theta}}
	\dfrac{
		\Gamma\left (
			\nu+b-Re^{i\theta}
		\right )
		\Gamma(Re^{i\theta})
	}{
		\Gamma\left (
			b-Re^{i\theta}
		\right )
	}
	Rie^{i\theta}
\,d\theta.
\]
Recall that the Gamma function decays in the left half-plane except at the poles (and that is why a sequence of $R$ will be chosen). Thus it is convenient to reflect
\[
Re^{i\theta}\Gamma(Re^{i\theta})
=\Gamma(1-(-Re^{i\theta}))
=\dfrac{\pi}{\sin(-\pi Re^{i\theta})}
\dfrac{1}{
	\Gamma(-Re^{i\theta})
}.
\]
Thus, we have
\begin{align*}
I_{\sigma,R,2}
&=\frac{1}{2\pi i}
\int_{\frac{\pi}{2}-\theta_\sigma}
	^{\frac{3\pi}{2}+\theta_\sigma}
	\rho^{-Re^{i\theta}}
	\dfrac{
		\pi
	}{
		\sin(-\pi Re^{i\theta})
	}
	\dfrac{
		\Gamma\left (
			\nu+b-Re^{i\theta}
		\right )
	}{
		\Gamma\left (
			b-Re^{i\theta}
		\right )
	}
	\dfrac{1}{
		\Gamma\left (
			-Re^{i\theta}
		\right )
	}
	i
\,d\theta\\
&=\frac{1}{2\pi}
\int_{-\frac{\pi}{2}-\theta_\sigma}
	^{\frac{\pi}{2}+\theta_\sigma}
	\rho^{Re^{i\theta}}
	\dfrac{
		\pi
	}{
		\sin(\pi Re^{i\theta})
	}
	\dfrac{
		\Gamma\left (
		\nu+b+Re^{i\theta}
		\right )
	}{
		\Gamma\left (
		b+Re^{i\theta}
		\right )
	}
	\dfrac{1}{
		\Gamma\left (
			Re^{i\theta}
		\right )
	}
\,d\theta.
\end{align*}
We estimate it in absolute value, using the Stirling approximation
\[
\log\Gamma(z+A)=(z+A-\tfrac12)\log z-z+\log\sqrt{2\pi}+O(|z|^{-1}),
	\quad
|\arg z|<\pi-\eps,
\]
for fixed $A\in \bC$ as $|z|\to\infty$.

First, it is immediate that for $A,B\in\bC$,
\[
\log\dfrac{\Gamma(z+A)}{\Gamma(z+B)}=(A-B)\log z+O(|z|^{-1}),
\]
hence
\[
\abs{
	\dfrac{
		\Gamma\left (
			\nu+b+Re^{i\theta}
		\right )
	}{
		\Gamma\left (
			b+Re^{i\theta}
		\right )
	}
}
=R^\nu(1+O(R^{-1})).
\]
Secondly, with $z=Re^{i\theta}$, 
\begin{align*}
\log \Gamma(Re^{i\theta}+a)
&=(R\cos\theta+a-\tfrac12+iR\sin\theta)(\log R+i\theta)
	-R\cos\theta-iR\sin\theta
	+\log\sqrt{2\pi}
	+O(R^{-1})\\
&=\left (
	R\log R\cos\theta
	-R(\cos\theta+\theta\sin\theta)
	+(a-\tfrac12)\log R
	+\log\sqrt{2\pi}
	+O(R^{-1})
\right )\\
&\qquad
+i\left (
	R\log R\sin\theta
	-R(\sin\theta-\theta\cos\theta)
	+(a-\tfrac12)\theta
	+O(R^{-1})
\right ),
\end{align*}
so
\begin{align*}
\abs{\Gamma(Re^{i\theta})}
=\exp\left (
	R\log R\cos\theta
	-R(\cos\theta+\theta\sin\theta)
	+(a-\tfrac12)\log R
	+\log\sqrt{2\pi}
	+O(R^{-1})
\right ).
\end{align*}
Thirdly, by \Cref{lem:abs-sin},
\[
\abs{
	\sin(\pi Re^{i\theta})
}
=\sqrt{
	\sin^2(\pi R|\cos\theta|)
	+\sinh^2(\pi R|\sin\theta|)
}
\gtrsim
	e^{\pi R|\sin\theta|},
\]
along a sequence $R=R_m$. Putting it all together,
\begin{align*}
&\quad\;
\abs{
	\rho^{R_m e^{i\theta}}
	\dfrac{
		\pi
	}{
		\sin(\pi R_m e^{i\theta})
	}
	\dfrac{
		\Gamma\left (
			\nu+b+R_m e^{i\theta}
		\right )
	}{
		\Gamma\left (
			b+R_m e^{i\theta}
		\right )
	}
	\dfrac{1}{
		\Gamma\left (
			R_m e^{i\theta}
		\right )
	}
}
\\
&\lesssim
	R_m^{\nu}
	\cdot e^{-\pi R_m|\sin\theta|}
	\cdot \rho^{R_m\cos\theta}
	\cdot e^{-R_m\log\frac{R_m}{e}\cos\theta+R_m\theta \sin\theta +O(\log R_m)}\\
&\lesssim
	e^{-R_m\left (
		\log\frac{R_m}{e\rho}\cos\theta
		+\pi|\sin\theta|-\theta\sin\theta
		+O(\tfrac{\log R_m}{R_m})
	\right )}\\
&\lesssim
	e^{-R_m\left (
		\log\frac{R_m}{e\rho}\cos\theta
		+\tfrac{\pi}{2}|\sin\theta|
		+o(1)
	\right )},
\end{align*}
where in the last step we have used the fact $|\theta|\leq \tfrac{\pi}{2}+O(R_m^{-1})$. Now, since the most negative value for $\log\tfrac{R_m}{\eps\rho}\cos\theta$ is $O(\tfrac{\sigma}{R_m}\log\tfrac{R_m}{\eps\rho})=o(1)$, we conclude that
\begin{align*}
|I_{\sigma,R_m,2}|
&\lesssim
\int_{-\frac{\pi}{2}-\theta_\sigma}
^{\frac{\pi}{2}+\theta_\sigma}
	e^{-cR_m}
\,d\theta
\to 0
\end{align*}
as $m\to\infty$.
\end{proof}

\begin{proof}[Proof of \Cref{lem:I-alpha-1}]
By definition,
\begin{align*}
	I_{\sigma,R,1}
	&=\sum_{z_j \text{ pole inside } \cC}
	\Res\left (
		\rho^{-z}
		\dfrac{
			\Gamma\left (
			\nu+b-z
			\right )
			\Gamma(z)
		}{
			\Gamma\left (
			b-z
			\right )
		},
		z_j
	\right ).
\end{align*}
Since $\nu+\frac{n}{2s}$ is not a non-positive integer, the integrand $f(z)$ only admits simple poles, which come from $\Gamma(z)$ at $z=-j \in (-\bN) \cap (-R,\sigma)$, and from $\Gamma(\nu+b-z)$ at $\nu+b-z=-\ell\in -\bN$ with $z\in (-R,\sigma)$, i.e. at $z=\nu+b+\ell \in (\nu+b+\bN) \cap (-R,\sigma)$. (Note that the latter case can happen only when $\nu<-n$.) The residues are given, respectively, by
	\begin{align*}
	\Res\left (
		\rho^{-z}
		\dfrac{
			\Gamma(z)
			\Gamma(\nu+b-z)
		}{
			\Gamma(b-z)
		},
		-j
		\right )
		&=
		\rho^{j}
		\dfrac{
			\Gamma(\nu+b+j)
		}{
			\Gamma(b+j)
		}
		\Res(\Gamma,-j)
		=\rho^{j}
		\dfrac{
			\Gamma(\nu+b+j)
		}{
			\Gamma(b+j)
		}
		\dfrac{(-1)^j}{j!},
	\end{align*}
and
	\begin{align*}
	\Res&\left (
		\rho^{-z}
		\dfrac{
			\Gamma(z)
			\Gamma(\nu+b-z)
		}{
			\Gamma(b-z)
		},
		\nu+b+\ell
		\right )\\
		&=
		\rho^{-\nu-b-\ell}
		\dfrac{
			\Gamma(\nu+b+\ell)
		}{
			\Gamma(-\nu-\ell)
		}
		\Res\left(
		\Gamma(\nu+b-z),
		\nu+b+\ell
		\right)\\
		&=\rho^{-\nu-b-\ell}
		\dfrac{
			\Gamma(\nu+b+\ell)
		}{
			\Gamma(-\nu-\ell)
		}
		\cdot (-1)
		\Res\left(
		\Gamma,-\ell
		\right)\\
		&=\rho^{-\nu-b-\ell}
		\dfrac{
			\sin(\pi(\nu+\ell))
		}{
			\pi
		}
		\Gamma(\nu+1+\ell)
		\Gamma(\nu+b+\ell)
		\dfrac{(-1)^{\ell+1}}{\ell!}\\
		&=-\rho^{-\nu-b-\ell}
		\dfrac{
			\sin(\pi\nu)
		}{
			\pi
		}
		\Gamma(\nu+1+\ell)
		\Gamma(\nu+b+\ell)
		\dfrac{1}{\ell!}.
	\end{align*}
	Therefore,
	\begin{align*}
		I_{\sigma,R,1}
		&=\sum_{j\in\bN \cap (-\sigma,R)}
		\rho^{j}
		\dfrac{
			\Gamma(\nu+b+j)
		}{
			\Gamma(b+j)
		}
		\dfrac{(-1)^j}{j!}\\
		&\qquad
		-\sum_{\ell\in\bN\cap (-\nu-b-R, -\nu-b+\sigma)}
		\rho^{-\nu-b-\ell}
		\dfrac{
			\sin(\pi\nu)
		}{
			\pi
		}
		\Gamma(\nu+1+\ell)
		\Gamma(\nu+b+\ell)
		\dfrac{1}{\ell!},
	\end{align*}
as desired.
\end{proof}

\begin{proof}[Proof of \Cref{prop:I-alpha}]
From \eqref{eq:I-alpha}, \Cref{lem:I-alpha-1} and \Cref{lem:I-alpha-2}, it follows that
\begin{align*}
I_{\sigma,\nu,b}(\rho)
&=
\sum_{j=0 \vee \ceils{-\sigma}}^{\infty}
\dfrac{
	\Gamma(\nu+b+j)
}{
	\Gamma(b+j)
}
\dfrac{(-1)^j}{j!}
\rho^{j}\\
&\qquad
-\sum_{\ell=0}^{\floors{-\nu-b+\sigma}}
\dfrac{
	\sin(\pi\nu)
}{
	\pi
}
\dfrac{
	\Gamma(\nu+1+\ell)
	\Gamma(\nu+b+\ell)
}{
	\ell!
}
\rho^{-\nu-b-\ell}.
\end{align*}
It remains to write the summation in $j$ into two: the first part from $0$ to $\infty$, minus the second from $0$ to $\floors{-\sigma}$. Since by definition of the rising factorial $a^{(j)}=a(a+1)\cdots(a+j-1)$, $\Gamma(a+j)=\Gamma(a)a^{(j)}$, the summation in $j\in\bN$ can be written in terms of the Kummer (or confluent hypergeometric) function
\[
M(a,b,z)
={}_1F_1(a;b;z)
=\sum_{j=0}^{\infty}
	\dfrac{
		a^{(j)}
	}{
		b^{(j)}
	}
	\dfrac{
		z^j
	}{
		j!
	},
\]
with $a=\nu+b$ 
and $z=-\rho$.
\end{proof}

\subsection{Eigenfunctions in Euclidean variable}

When $s\in(\frac12,1]$, the eigenfunctions can be expressed in terms of the Fox--Wright $\Psi$-function, a generalization of the hypergeometric function which allows non-integer increments. When $s\in(0,\frac12)$, a similar formula is obtained as a Puiseux series. The case $s=\frac12$ is critical and Fourier inversion is used instead (of Mellin inversion).

\begin{thm}[Explicit eigenfunctions]\label{thm:explicit}
Let $s\in(0,1]$. In the case $s\neq\frac12$, we assume  $\nu>-\frac{n}{2s}$ and consider the Mellin inversion along the $\re z=\sigma=0^+$. Then, the eigenfunctions
\begin{equation*}
u_\nu(r):=
\begin{cases}
\cM_{2\sigma}^{-1}\set{\widetilde{u}_\nu(z)}(r),
	& s\in(\frac12,1],\,r\geq 0,\\
\cM_{2\sigma}^{-1}\set{\widetilde{u}_\nu(z)}(r)
=\cF^{-1}\set{\widehat{u}(\zeta)}(r),
	& s=\frac12,\,0\leq r<1,\\
\cM_{2s\sigma}^{-1}\set{\widetilde{u}_\nu(z)}(r),
	& s\in(0,\frac12),\,r>1.
\end{cases}
\end{equation*}
are given explicitly for $r\in(0,+\infty)$ in terms of special functions by
\begin{equation}
u_\nu(r)
=\begin{dcases}
2\,{}_1\Psi_{1}\left(
\begin{array}{c}
(\nu+\tfrac{n}{2s},\tfrac{1}{s})
\\
(\tfrac{n}{2},1)
\end{array}
;-\frac{r^2}{4}
\right)
	& s\in(\tfrac12,1],\\
2^{-\frac{n-2}{2}}
\dfrac{\Gamma(n+\nu)}{\Gamma(\tfrac{n}{2})}
\Hyperg\big(
	\tfrac{n+\nu}{2},
	\tfrac{n+\nu+1}{2};
	\tfrac{n}{2};
	-r^2
\big)
	& s=\tfrac12,\\
-2s
	\left(
		\frac{2}{r}
	\right)^{2s\nu+n}
	{}_1\Psi_{1}\left(
		\begin{array}{c}
		(s\nu+\tfrac{n}{2},s)
		\\
		(-s\nu,-s)
		\end{array};\,
		-\Big(\frac{2}{r}\Big)^{2s}
	\right)
	& s\in(0,\tfrac12).
\end{dcases}\label{eq:expl-eig-func}
\end{equation}
\end{thm}

Notice that $u_\nu(r)$ is analytic up to $r=0$ exactly when $s\in[\frac12,1]$. Here ${}_1\Psi_{1}$ denotes the Fox--Wright $\Psi$-function and $\Hyperg$ denotes the Gauss hypergeometric function. The derivations will be carried out case by case. For the series representations, see \Cref{lem:eigen-s-large}, \Cref{lem:eigen-s-crit} and \Cref{lem:eigen-s-small} respectively.

\begin{lemma}[Subcritical case]
\label{lem:eigen-s-large}
Let $s\in(\tfrac12,1]$ and $\nu>-\tfrac{n}{2s}$. The eigenfunctions $u_\nu$, obtained by the inverse Mellin transform of $\widetilde{u}_\nu$ on $\re z=\sigma$ with $0<\sigma<s\nu+\tfrac{n}{2}$, are given explicitly by the Fox--Wright function
\begin{align*}
u_{\nu}(r)
:=
\cM_{2\sigma}^{-1}\set{\widetilde{u}_\nu(z)}(r)
&=2\,{}_1\Psi_{1}\left(
\begin{array}{c}
(\nu+\tfrac{n}{2s},\tfrac{1}{s})
\\
(\tfrac{n}{2},1)
\end{array}
;-\frac{r^2}{4}
\right)
=
2\sum_{j=0}^{\infty}
	\dfrac{
		\Gamma(\nu+\tfrac{n}{2s}+\frac{j}{s})
	}{
		\Gamma(\tfrac{n}{2}+j)
		j!
	}
\left(-\frac{r^2}{4}\right)^{j}.
\end{align*}
As $r\to+\infty$, $u_{\nu}(r)$ decays algebraically.
\end{lemma}

\begin{remark}
Observe that the series converges absolutely for all $r\in\R$ if $s>\tfrac12$. When $s=\tfrac12$, convergence is guaranteed only for $r<1$.

While the condition $\nu>-\frac{n}{2s}$ is violated, extra residues need to be taken into account, in a way similar to the proof of \Cref{lem:eigen-s-small}. In particular, the resulting eigenfunction would be singular at the origin.
\end{remark}

\begin{remark}
We note the following special cases:
\begin{itemize}
\item When $s=1$,
\[
u_{\nu}(r)
=
M\big(\nu+\tfrac{n}{2};\tfrac{n}{2};-\tfrac{r^2}{4}\big)
=
e^{-\frac{r^2}{4}}M(-\nu;\tfrac{n}{2};\tfrac{r^2}{4}),
\]
where $M={}_1F_{1}$ is the Kummer function.
\item When $s=1$ and $\nu=k\in\bN$,
\[
u_{k}(r)
=
e^{-\frac{r^2}{4}}
M\big(
	{-k};
	\tfrac{n}{2};
	\tfrac{r^2}{4}
\big)
=
\dfrac{
	\Gamma\left (
		\frac{n}{2}
	\right )
	k!
}{
	\Gamma\left (
		k+\frac{n}{2}
	\right )
}
e^{-\frac{r^2}{4}}
L_k^{ (\frac{n-2}{2} )}
\big(\tfrac{r^2}{4}\big),
\]
where $L_k^{(\frac{n-2}{2})}$ is the generalized Laguerre polynomial.
\end{itemize}
\end{remark}

\begin{proof}[Proof of \Cref{lem:eigen-s-large}]
We first scale away the factors using \eqref{eq:M-inv-exp} and \eqref{eq:M-inv-rescale},
\begin{align*}
\cM_{2\sigma}^{-1}\set{\widetilde{u}_\nu(z)}(r)
&=2\cM_{\sigma}^{-1}\set{
	\frac{
		\Gamma(\nu+\tfrac{n}{2s}-\frac{z}{s})
		\Gamma(z)
	}{
		\Gamma(\tfrac{n}{2}-z)
	}
}
\left (
	\frac{r^2}{4}
\right ).
\end{align*}
For simplicity, write $\rho=\frac{r^2}{4}$. Recall that by the inverse Mellin transform \eqref{eq:M-inv-2},
\begin{align*}
\cM_{\sigma}^{-1}\set{
	\frac{
		\Gamma(\nu+\tfrac{n}{2s}-\frac{z}{s})
		\Gamma(z)
	}{
		\Gamma(\tfrac{n}{2}-z)
	}
}(\rho)
&=\dfrac{1}{2\pi i}
\int_{\sigma-i\infty}^{\sigma+i\infty}
	\frac{
		\Gamma(\nu+\tfrac{n}{2s}-\frac{z}{s})
		\Gamma(z)
	}{
		\Gamma(\tfrac{n}{2}-z)
	}
	\rho^{-z}
\,dz\\
&=:H_{1,2}^{1,1}\left(
\rho\,
\Big|
\begin{array}{cc}
(1-\nu-\tfrac{n}{2s},\tfrac{1}{s}) &
\\
(0,1) & (1-\tfrac{n}{2},1)
\end{array}
\right)
\end{align*}
is, by definition, a special case of the Fox $H$-function \cite{Fox} (notice that $\sigma+i\R$ separates the poles of $\Gamma(\nu+\tfrac{n}{2s}-\tfrac{z}{s})$ from those of $\Gamma(z)$ since $\nu+\tfrac{n}{2s}>0$). This is in turn (note the parameter $(0,1)$ which corresponds to the simple $\Gamma(z)$) identified as the Fox--Wright $\Psi$-function \cite[p. 50]{Srivastava-Manocha}
\begin{align*}
H_{1,2}^{1,1}\left(
\rho\,
\Big|
\begin{array}{cc}
(1-\nu-\tfrac{n}{2s},\tfrac{1}{s}) &
\\
(0,1) & (1-\tfrac{n}{2},1)
\end{array}
\right)
&\,=
{}_1\Psi_{1}\left(
\begin{array}{c}
(\nu+\tfrac{n}{2s},\tfrac{1}{s})
\\
(\tfrac{n}{2},1)
\end{array}
;-\rho
\right)
:=
\sum_{j=0}^{\infty}
	\dfrac{
		\Gamma(\nu+\tfrac{n}{2s}+\frac{j}{s})
	}{
		\Gamma(\tfrac{n}{2}+j)
	}
\dfrac{(-\rho)^j}{j!}.
\end{align*}
We thus conclude that
\begin{align*}
\cM_{2\sigma}^{-1}\set{\widetilde{u}_\nu(z)}(r)
&=2\,{}_1\Psi_{1}\left(
\begin{array}{c}
(\nu+\tfrac{n}{2s},\tfrac{1}{s})
\\
(\tfrac{n}{2},1)
\end{array}
;-\frac{r^2}{4}
\right)
=
2\sum_{j=0}^{\infty}
	\dfrac{
		\Gamma(\nu+\tfrac{n}{2s}+\frac{j}{s})
	}{
		\Gamma(\tfrac{n}{2}+j)
		j!
	}
\left(-\frac{r^2}{4}\right)^{j}.
\end{align*}
The decay of $u_{\nu}(r)$ as $r\to+\infty$ follows from \cite[Theorem 4]{Wright}.
\end{proof}

\begin{lemma}[Critical case]
\label{lem:eigen-s-crit}
Let $s=\frac12$. Then
\begin{align*}
u_\nu(r)
&=2^{-\frac{n-2}{2}}
\dfrac{\Gamma(n+\nu)}{\Gamma(\tfrac{n}{2})}
\Hyperg\big(
	\tfrac{n+\nu}{2},
	\tfrac{n+\nu+1}{2};
	\tfrac{n}{2};
	-r^2
\big)\\
&=\Gamma(n+\nu)
	r^{-\frac{n-2}{2}}
	\big(\sqrt{r^2+1}\big)^{-(\frac{n}{2}+1+\nu)}
	P^{-\frac{n-2}{2}}_{\frac{n}{2}+\nu}
	\Big(
		\frac{1}{\sqrt{r^2+1}}
	\Big).
\end{align*}
Here $\Hyperg$ is the Gauss hypergeometric function and $P^{-\mu}_{\nu}$ ($\mu,\nu\in\R$) denotes the associated Legendre functions of the first kind \cite[8.704]{Gradshteyn-Ryzhik} which satisfies the differential equation
\[
(1-z^2)\frac{d^2u}{dz^2}
-2z\frac{du}{dz}
+\left (
	\nu(\nu+1)-\frac{\mu^2}{1-z^2}
\right )u
=0,
\]
and is more explicitly given by
\[
P^{\mu}_{\nu}(x)
=\frac{1}{\Gamma(1-\mu)}
\left(
	\frac{1+x}{1-x}
\right)^{\frac{\mu}{2}}
\Hyperg\Big(
	{-\nu},
	\nu+1;
	1-\mu;
	\frac{1-x}{2}
\Big),
	\qquad \text{for all }|x|<1.
\]
\end{lemma}

\begin{proof}
The inverse Hankel transform of $\widehat{u}_\nu(\zeta)$ is
\begin{align*}
u_\nu(r)
=r^{-\frac{n-2}{2}}
\int_{0}^{\infty}
	J_{\frac{n-2}{2}}(\zeta r)
	\zeta^{\frac{n}{2}+\nu}
	e^{-\zeta}
\,d\zeta,
\end{align*}
which is evaluated according to \cite[6.621 1.]{Gradshteyn-Ryzhik}.
\end{proof}

\begin{remark}\label{rmk:LegendreP}
We list some useful properties of $P^{-\mu}_{\nu}(x)$, for $|x|<1$.

\begin{itemize}
\item When $\mu,\nu\in\bN$,
	$P^{-\mu}_{\nu}$ and $P^{\mu}_{\nu}$ are related by $P^{-\mu}_{\nu}(x)=(-1)^{\mu}\frac{(\nu-\mu)!}{(\nu+\mu)!}P^{\mu}_{\nu}(x)$. For general parameters $\mu,\nu\in\R$, a similar relation holds \cite[8.737 1.]{Gradshteyn-Ryzhik}.
\item For $\mu,\nu\in\R$, $P^{\mu}_{-\nu-1}(x)=P^{\mu}_{\nu}(x)$.
\item For $\mu+\nu\in\bZ$, $P^{\mu}_{\nu}(-x)=(-1)^{\mu+\nu}P^{\mu}_{\nu}(x)$. In the general case $\mu+\nu\in\R$, see \cite[8.737 2.]{Gradshteyn-Ryzhik}.
\item
	When $\mu,\nu\in\bN$, the Rodrigues formula reads
	\[
	P^{\mu}_{\nu}(x)
	=\dfrac{(-1)^{\mu+\nu}}{2^{\nu}\nu!}
		(1-x^2)^{\frac{\mu}{2}}
	\dfrac{d^{\mu+\nu}}{dx^{\mu+\nu}}
		(1-x^2)^{\nu}.
	\]
	In particular, $P^{\mu}_{\nu}(x)$ is a closed algebraic expression involving powers of $x$ and $\sqrt{1-x^2}$.
\end{itemize}
\end{remark}

\begin{lemma}[Eigenfunctions in even dimensions]
Let $n\geq 2$ be even, $s=\frac12$, $\nu=k\in\bN$. Then:
\begin{itemize}
\item
	$u_k(r)$ is a closed algebraic expression involving powers of $r$ and $\sqrt{r^2+1}$.
\item
	The orthogonality condition holds for $k+l$ even:
	\[
	\int_{0}^{\infty}
		u_k(r)u_l(r)
		\big(\sqrt{r^2+1}\big)^{k+l+n-1}
		r^{n-1}
	\,dr
	=\frac{(k+1)!(n+k-1)!}{n+2k+1}\delta_{kl}.
	\]
\end{itemize}
\end{lemma}

\begin{proof}
The first item follows from \Cref{rmk:LegendreP}. For the second, we use \cite[7.122, 7.123]{Gradshteyn-Ryzhik} to see that
\[
\int_{-1}^{1}
	P^{\frac{n-2}{2}}_{\frac{n}{2}+k}(x)
	P^{\frac{n-2}{2}}_{\frac{n}{2}+l}(x)
\,dx
=\dfrac{2}{n+2k+1}\dfrac{(n+k-1)!}{(k+1)!}\delta_{kl}.
\]
Reflecting the upper index,
\[
\int_{-1}^{1}
	P^{-\frac{n-2}{2}}_{\frac{n}{2}+k}(x)
	P^{-\frac{n-2}{2}}_{\frac{n}{2}+l}(x)
\,dx
=\dfrac{2}{n+2k+1}\dfrac{(k+1)!}{(n+k-1)!}\delta_{kl}.
\]
Writing the integral on $[0,1]$, we have
\[
\big(1+(-1)^{k+l}\big)
\int_{0}^{1}
	P^{-\frac{n-2}{2}}_{\frac{n}{2}+k}(x)
	P^{-\frac{n-2}{2}}_{\frac{n}{2}+l}(x)	
\,dx
=\dfrac{2}{n+2k+1}\dfrac{(k+1)!}{(n+k-1)!}\delta_{kl}.
\]
This is nontrivial for $k+l$ even, and we change variable to $x=\frac{1}{\sqrt{r^2+1}}$ which yields
\[
\int_{0}^{\infty}
	P^{-\frac{n-2}{2}}_{\frac{n}{2}+k}
	\Big(
		\frac{1}{\sqrt{r^2+1}}
	\Big)
	P^{-\frac{n-2}{2}}_{\frac{n}{2}+l}
	\Big(
		\frac{1}{\sqrt{r^2+1}}
	\Big)
	\frac{r}{\big( \sqrt{r^2+1}\big)^3}
\,dr
=\dfrac{1}{n+2k+1}\dfrac{(k+1)!}{(n+k-1)!}\delta_{kl}.
\]
Writing this in terms of $u_k$ and $u_l$, the result follows.
\end{proof}

\begin{lemma}
As $s\to (\frac12)^+$,
\[
\cM_{2\sigma}^{-1}\set{u_\nu(z)}(r)
\to C\cF^{-1}\set{\widehat{u}(\zeta)}(r),
	\qquad \forall r\in[0,1),
\]
for some explicit constant $C$.
\end{lemma}

\begin{proof}
By legendre duplication formula, we have
\begin{align*}
\cM_{2\sigma}^{-1}\set{u_\nu(z)}(r)
&\to
2\sum_{j=0}^{\infty}
	\dfrac{
		\Gamma(\nu+n+2j)
	}{
		\Gamma(\tfrac{n}{2}+j)
		j!
	}
\left(-\frac{r^2}{4}\right)^{j}
\\
&=
\sum_{j=0}^{\infty}
	2^{\nu+n+2j}\pi^{-\frac{1}{2}}
	\frac{
		\Gamma(\frac{\nu+n}{2}+j)
		\Gamma(\frac{\nu+n+1}{2}+j)
	}{
		\Gamma(\frac{n}{2}+j)
		j!
	}
	\left(
		-\frac{r^2}{4}
	\right)^{j}\\
&=\frac{2^{\nu+n}}{\sqrt{\pi}}
\frac{
	\Gamma(\frac{\nu+n}{2})
	\Gamma(\frac{\nu+n+1}{2})
}{
	\Gamma(\frac{n}{2})
}
\Hyperg\big(
	\tfrac{n+\nu}{2},
	\tfrac{n+\nu+1}{2};
	\tfrac{n}{2};
	-r^2
\big)\\
&=
\dfrac{2\Gamma(n+\nu)}{\Gamma(\tfrac{n}{2})}
\Hyperg\big(
	\tfrac{n+\nu}{2},
	\tfrac{n+\nu+1}{2};
	\tfrac{n}{2};
	-r^2
\big)\\
&=
2^{-\frac{n-2}{2}}
\dfrac{\Gamma(n+\nu)}{\Gamma(\tfrac{n}{2})}
\Hyperg\big(
	\tfrac{n+\nu}{2},
	\tfrac{n+\nu+1}{2};
	\tfrac{n}{2};
	-r^2
\big),
\end{align*}
which is $\cF^{-1}\set{\widehat{u}(\zeta)}(r)$ up to a constant.
\end{proof}

\begin{lemma}[Supercritical case]
\label{lem:eigen-s-small}
Let $s\in(0,\frac12)$, $\nu>-\frac{n}{2s}$. The eigenfunctions $u_\nu(r)$, obtained by the inverse Mellin transform of $\widetilde{u}_\nu(z)$ on $\re z=\sigma$ with $0<\sigma<s\nu+\tfrac{n}{2}$, are given explicitly by the Fox--Wright function
\begin{align*}
u_\nu(r)
:=\cM_{2s\sigma}^{-1}\set{\widetilde{u}_\nu(z)}(r)
&=-2s
	\left(
		\frac{2}{r}
	\right)^{2s\nu+n}
	{}_1\Psi_{1}\left(
		\begin{array}{c}
		(s\nu+\tfrac{n}{2},s)
		\\
		(-s\nu,-s)
		\end{array};\,
		-\Big(\frac{2}{r}\Big)^{2s}
	\right)\\
&=2s
	\sum_{j=0}^{\infty}
	(-1)^{j+1}
	\dfrac{
		\Gamma(s\nu+\frac{n}{2}+sj)
	}{
		\Gamma(-s\nu-sj)
		j!
	}
	\Big(\frac{2}{r}\Big)^{n+2s(\nu+j)}.
\end{align*}
\end{lemma}

\begin{proof}
As in the subcritical case $s\in(\frac12,1)$, we first scale away the factors using \eqref{eq:M-inv-exp} and \eqref{eq:M-inv-rescale},
\begin{align*}
\cM_{2s\sigma}^{-1}\set{\widetilde{u}_\nu(z)}(r)
&=2s\cM_{\sigma}^{-1}\set{
	\frac{
		\Gamma(\nu+\tfrac{n}{2s}-z)
		\Gamma(sz)
	}{
		\Gamma(\tfrac{n}{2}-sz)
	}
}
\left (
	\frac{r^{2s}}{2^{2s}}
\right ).
\end{align*}
Write $\rho=2^{-2s}r^{2s}$. Recall that by the inverse Mellin transform \eqref{eq:M-inv-2},
\begin{align*}
\cM_{\sigma}^{-1}\set{
	\frac{
		\Gamma(\nu+\tfrac{n}{2s}-z)
		\Gamma(sz)
	}{
		\Gamma(\tfrac{n}{2}-sz)
	}
}(\rho)
&=\dfrac{1}{2\pi i}
\lim_{R\to\infty}
\int_{\sigma-iR}^{\sigma+iR}
	\frac{
		\Gamma(\nu+\tfrac{n}{2s}-z)
		\Gamma(sz)
	}{
		\Gamma(\tfrac{n}{2}-sz)
	}
	\rho^{-z}
\,dz.
\end{align*}
Using Stirling's and Euler's reflection formulae, we can show that the integrand decays along the right semicircle like $R^{-(\frac{1}{2s}-1)R\cos\theta}e^{O(R)}$, for $\theta=\arg z\in(-\frac{\pi}{2},\frac{\pi}{2})$, as $R=|z|\to\infty$.  By residue theorem, the contour (on the right half plane) encloses the (simple) poles of $\Gamma(\nu+\frac{n}{2s}-z)$, namely $\nu+\frac{n}{2s}+j$, $j\in\bN$, at which the residues are
\begin{align*}
\Res&\left (
	\dfrac{
		\Gamma(\nu+\frac{n}{2s}-z)
		\Gamma(sz)
	}{
		\Gamma(\frac{n}{2}-sz)
	}
	\rho^{-z},
	\nu+\tfrac{n}{2s}+j
	\right )\\
	&=
	\dfrac{
		\Gamma(s\nu+\frac{n}{2}+sj)
	}{
		\Gamma(-s\nu-sj)
	}
	\rho^{-\nu-\frac{n}{2s}-j}
	\Res\left(
		\Gamma(\nu+\tfrac{n}{2s}-z),
		\nu+\tfrac{n}{2s}+j
	\right)
	\\
	&=
	\dfrac{
		\Gamma(s\nu+\frac{n}{2}+sj)
	}{
		\Gamma(-s\nu-sj)
	}
	\rho^{-\nu-\frac{n}{2s}-j}
	\cdot (-1) \Res(\Gamma,-j)
	\\
	&=
	\dfrac{
		\Gamma(s\nu+\frac{n}{2}+sj)
	}{
		\Gamma(-s\nu-sj)
	}
	\rho^{-\nu-\frac{n}{2s}-j}
	\dfrac{(-1)^{j+1}}{j!}.
\end{align*}
Therefore,
\begin{align*}
\cM_{\sigma}^{-1}\set{
	\frac{
		\Gamma(\nu+\tfrac{n}{2s}-z)
		\Gamma(sz)
	}{
		\Gamma(\tfrac{n}{2}-sz)
	}
}(\rho)
&=\sum_{j=0}^{\infty}
\Res\left (
	\dfrac{
		\Gamma(\nu+\frac{n}{2s}-z)
		\Gamma(sz)
	}{
		\Gamma(\frac{n}{2}-sz)
	}
	\rho^{-z},
	\nu+\tfrac{n}{2s}+j
	\right )\\
&=-\dfrac{1}{\rho^{\nu+\frac{n}{2s}}}
\sum_{j=0}^{\infty}
	\dfrac{
		\Gamma(s\nu+\frac{n}{2}+sj)
	}{
		\Gamma(-s\nu-sj)
		j!
	}
	\Big({-}\frac{1}{\rho}\Big)^j\\
&=-\dfrac{1}{\rho^{\nu+\frac{n}{2s}}}
	{}_1\Psi_{1}\left(
		\begin{array}{c}
		(s\nu+\tfrac{n}{2},s)
		\\
		(-s\nu,-s)
		\end{array};\,
		-\frac{1}{\rho}
	\right).
\end{align*}
We conclude that
\begin{align*}
u_\nu(r)
:=\cM_{2s\sigma}^{-1}\set{\widetilde{u}_\nu(z)}(r)
&=-2s
	\left(
		\frac{2}{r}
	\right)^{2s\nu+n}
	{}_1\Psi_{1}\left(
		\begin{array}{c}
		(s\nu+\tfrac{n}{2},s)
		\\
		(-s\nu,-s)
		\end{array};\,
		-\Big(\frac{2}{r}\Big)^{2s}
	\right)\\
&=-2s
	\left(
		\frac{2}{r}
	\right)^{2s\nu+n}
\sum_{j=0}^{\infty}
	\dfrac{
		\Gamma(s\nu+\frac{n}{2}+sj)
	}{
		\Gamma(-s\nu-sj)
		j!
	}
	\Big({-}\frac{2}{r}\Big)^{2sj},
\end{align*}
as desired.
\end{proof}

\subsection{Uniqueness for a non-local ODE}\label{subsection:uniqueness}

Here we prove the uniqueness claim for solutions of the eigenvalue equation \eqref{eigenvalue-problem} in the space $\mathds L^2_\dag$ (Remark \ref{remark:uniqueness}). Contrary to the local case $s=1$, where one may calculate the indicial roots of an ODE (see, for instance, the characterization in \cite[Proposition 2.2]{Mizoguchi}), in the non-local setting $s\in(0,1)$ uniqueness is a non-trivial issue. Here we use a complex variable approach.

\begin{prop}[Uniqueness]
\label{prop:uniq}
Let $s\in(0,1)$, $n\geq 2s$, $u\in\mathds L^2_\dag$ be a radially symmetric solution to \eqref{eigenvalue-problem}. Assume that $\nu\in\bZ$ in case $n=1$. Then, up to a scalar multiple,
\[
\cM u(z)=2^{z}\frac{\Gamma(\frac{z}{2})\Gamma(\nu+\frac{n-z}{2s})}{\Gamma(\frac{n-z}{2})}.
\]
\end{prop}

\begin{proof}
By \Cref{lem:Ls-L1}, $u\in \mathds L^2_\dag$ implies that $u\in \mathds H^k_\dag$ for all $k\in\bN$. We have seen in the proof of Lemma \ref{lemma:eigenfunction1} that, if set
\begin{equation*}
\mathcal Mu(z):=2^{z}\frac{\Gamma(\frac{z}{2})\Gamma(\nu+\frac{n-z}{2s})}{\Gamma(\frac{n-z}{2})}v(z),
\end{equation*}
then $v$ is a $2s$-periodic function on 
some strip.
Using the multiplier notation, this is equivalent to
\begin{equation*}
\mathcal M(\Phi_s^{-1}u)(z)=2^{\frac{z}{s}}\frac{\Gamma(\frac{n}{2}-\frac{n}{2s}+\frac{z}{2s})\Gamma(\nu+\frac{n-z}{2s})}{\Gamma(\frac{n-z}{2s})}v(z).
\end{equation*}
In order to handle the poles of  $\Gamma(\frac{n-z}{2s})$, we may rewrite this as
\begin{equation}\label{eq:bdd}
\mathcal M(\Phi_s^{-1}u)(z)
=2^{\frac{z}{s}}
\frac{\Gamma(\frac{n}{2}-\frac{n}{2s}+\frac{z}{2s})\Gamma(\nu+\frac{n-z}{2s})}{\Gamma(\frac{n-z}{2s})\sin(\frac{n-z}{2s}2\pi)}\phi(\tfrac{n-z}{2s}).
\end{equation}
for some $1$-periodic function $\phi$.

Note that, on the one hand, for $\tilde{z}\in \bC$ with $\re \tilde{z}=\sigma$,
\begin{equation}\label{eq:Hk-bd-1}
\begin{split}
\abs{\cM\set{\partial_\rho^k u^\dag(\rho)}(\tilde{z})}
&\leq \int_{0}^{\infty}
	\rho^{\sigma-1}
	\abs{\partial_\rho^k u^\dag(\rho)}
\,d\rho\\
&\leq
\left(
	\int_{0}^{\infty}
		\abs{\partial_\rho^k u^\dag(\rho)}^2
		e^{\frac{\rho^2}{4}}
		\rho^{n-1}
	\,d\rho
\right)^{\frac12}
\left(
	\int_{0}^{\infty}
		e^{-\frac{\rho^2}{4}}
		\rho^{1-n+2\sigma-2}
	\,d\rho
\right)^{\frac12}
<\infty,
\end{split}\end{equation}
as long as $u \in \mathds H_\dag^k$ and $\sigma>\frac{n}{2}$. On the other hand, transforming the derivatives
\[\begin{split}
\cM\set{\partial_\rho^k u^\dag(\rho)}(\tilde{z})
&=\cM\set{\rho\partial_\rho^k u^\dag(\rho)}(\tilde{z}-1)
=-(\tilde{z}-1)\cM\set{\partial_\rho^{k-1}u^\dag(\rho)}(\tilde{z}-1)\\
&=-(\tilde{z}-1)\cM\set{\rho\partial_\rho^{k-1}u^\dag(\rho)}(\tilde{z}-2)
=(\tilde{z}-1)(\tilde{z}-2)\cM\set{\partial_\rho^{k-2}u^\dag(\rho)}(\tilde{z}-2)\\
&=\cdots=(-1)^k(\tilde{z}-1)\cdots(\tilde{z}-k)\cM\set{u^\dag(\rho)}(\tilde{z}-k),
\end{split}\]
then using \eqref{simple-relation} to relate with $\Phi_s^{-1}u$,
\[\begin{split}
\cM\set{\partial_\rho^k u^\dag(\rho)}(\tilde{z})
&=(-1)^k\frac{\Gamma(\tilde{z})}{\Gamma(\tilde{z}-k)}
\cM\set{\left(\tfrac{\rho}{2}\right)^{\frac{n}{s}-n}\Phi_s^{-1}u(\rho^{\frac1s})}(\tilde{z}-k)\\
&=2^{n-\frac{n}{s}}s
(-1)^k\frac{\Gamma(\tilde{z})}{\Gamma(\tilde{z}-k)}
\cM\set{r^{n-ns}\Phi_s^{-1}u(r)}(s\tilde{z}-ks)\\
&=2^{n-\frac{n}{s}}s
(-1)^k\frac{\Gamma(\tilde{z})}{\Gamma(\tilde{z}-k)}
\cM\set{\Phi_s^{-1}(r)}(s\tilde{z}-ks+n-ns),
\end{split}\]
we can use \eqref{eq:bdd} by setting $z=s\tilde{z}-ks+n-ns$. In the variable $w=\frac{n-z}{2s}=\frac{k+n-\tilde{z}}{2}$, we have
\begin{equation}
\label{eq:Hk-bd-2}
\begin{split}
\cM\set{\partial_\rho^k u^\dag(\rho)}(k+n-2w)
&=2^{n-2w}s(-1)^k
\frac{
	\Gamma(k+n-2w)
}{
	\Gamma(n-2w)
}
\frac{
	\Gamma(\frac{n-2w}{2})
	\Gamma(\nu+w)
}{
	\Gamma(w)
	\sin(2\pi w)
}
\phi(w)\\
&=2\sqrt{\pi}s(-1)^k
\frac{
	\Gamma(k+n-2w)
	\Gamma(\nu+w)
}{
	\Gamma(\frac{n+1}{2}-w)
	\Gamma(w)
	\sin(2\pi w)
}
\phi(w),
\end{split}\end{equation}
where $\phi(w)$ is $1$-periodic, after using Legendre duplication formula. Combining \eqref{eq:Hk-bd-1}--\eqref{eq:Hk-bd-2} yields
\begin{equation}
\label{eq:Hk-bd-3}
\begin{split}
|\phi(w)|
&\leq
\frac{C(n,s,k)}{2\sigma-n}
\norm[\mathds H^k_\dag]{u}
\frac{
	\Gamma(\frac{n+1}{2}-w)
	\Gamma(w)
	\sin(2\pi w)
}{
	\Gamma(k+n-2w)
	\Gamma(\nu+w)
}<\infty.
\end{split}\end{equation}
whenever $\sigma=\re(k+n-2w)=\re \tilde{z}>\frac{n}{2}$, i.e. $\re w<\frac{n}{4}+\frac{k}{2}$. In fact,
\begin{equation*}
|\phi(w)|\leq C e^{\frac{5\pi}{2}|\lambda|}
|\lambda|^C
\quad\text{as}\quad |\lambda |:=|\im w|\to\infty,
\end{equation*}
in a band $\mathcal B:=\{\sigma_*-1<\re w <\sigma_*\}$ for some $\sigma_*<\frac{n}{4}+\frac{k}{2}$.

Now we make the change of variable
\begin{equation*}
\psi(\eta)=\phi(w),
	\qquad \text{ for } \qquad
\eta=e^{2\pi iw},
\end{equation*}
which maps the band $\mathcal B$ one-to-one into the complex plane. Moreover, $\psi(\eta)$ is an holomorphic function on $\mathbb C\setminus\{0\}$ which satisfies
\begin{equation*}
|\psi(\eta)|\leq
\begin{cases}
C|\eta|^{(\frac{5}{4})^+}
&\text{as } |\eta|\to \infty,\\
C|\eta|^{-(\frac{5}{4})^+}
&\text{as }|\eta|\to 0.
\end{cases}
\end{equation*}
By a generalization of Liouville's theorem, we conclude that
\begin{equation*}
\psi(\eta)=a\eta+b\eta^{-1}+c
\end{equation*}
for some constants $a,b,c\in\mathbb C$. We go back to the variable $w$ and write  trigonometrically,
\[
\phi(w)
=
	a\sin(2\pi w)
	+b\cos(2\pi w)
	+c.
\]
for some (renamed) $a,b,c\in \bC$. Therefore, \eqref{eq:Hk-bd-3} gives
\begin{equation*}
\abs{
a+
\dfrac{
	b\cos(2\pi w)
	+c
}{
	\sin(2\pi w)
}
}
\dfrac{
	|\Gamma(k+n-2w)|
	|\Gamma(\nu+w)|
}{
	|\Gamma(\frac{n+1}{2}-w)|
	|\Gamma(w)|
}
<\infty,
	\quad \text{ for } \re w<\tfrac{n}{4}+\tfrac{k}{2}.
\end{equation*}
Taking $w=-\frac12$, we get $b=c$, so we arrive (using trigonometry and Euler reflection) to
\[
\abs{
	a\sin(\pi w)+b\cos(\pi w)
}
\frac{
	|\Gamma(k+n-2w)|
	|\Gamma(\nu+w)|
	|\Gamma(1-w)|
}{
	|\Gamma(\frac{n+1}{2}-w)|
}
<\infty,
	\quad \text{ for } \re w<\tfrac{n}{4}+\tfrac{k}{2}.
\]
If $\nu\in\bZ$, putting $w\in -\nu-\bN$ negative gives $b=0$, as desired. If $n\geq 2$, we fix $k=2$ so that $\frac{n}{4}+\frac{k}{2}>1$, and putting $w=1$ again forces $b=0$.
\end{proof}

\begin{remark}
If $\nu\notin \bZ$ and $n=1$, it can happen that $a=0$, $b\neq 0$, $\nu=\frac12$, say.
\end{remark}

\section{Eigenvalues for the adjoint problem}\label{section:adjoint}

Now we consider the eigenvalue problem for the  adjoint operator, this is
\begin{equation}\label{problem-adjoint}
L_s^*w=(-\Delta)^s w+\frac{1}{2s}r w_r=\nu w,\quad w=w(r).
\end{equation}
We will characterize the $L^2$-dual space of $\mathds L^2_\dag$ in Section \ref{subsection:duality}. As a consequence,  having already a good description of the spectral properties for $L_s$ in our function space, the properties of its adjoint $L_s^*$ follow easily.

If, instead, one tries to use the arguments in Section \eqref{section:eigenvalue} directly, one encounters the need  of introducing homogeneous distributions (see, for instance, \cite{Kwasnicki} and the references therein). Some of these difficulties are already  present in the local case $s=1$, since eingenfunctions are polynomials
\begin{equation*}
\omega_{k}^{(1)}(r)=L_{k}^{(\frac{n-2}{2})}\big(\tfrac{r^{2}}{4}\big),
\end{equation*}
 with eigenvalues $\nu_k=k$, $k\in\mathbb N$, which do not have a well behaved Fourier/Mellin transform.

In the fractional case $s\in(0,1)$, we encounter the same obstacles, as we next explain. First, proceed as in the previous sections, taking Mellin transform of \eqref{problem-adjoint},
\begin{equation*}\label{equation-adjoint-mellin}
\Theta_{s}(\lambda)\,\mathcal  Mw(z-2s)=\left(\nu+\frac{z}{2s}\right)\mathcal M w(z).
\end{equation*}
Now we use the auxiliary function
\begin{equation*}
v(z)
:=2^{-z}
\dfrac{
	\Gamma\left (
		\tfrac{n-z}{2}
	\right )\Gamma\left (
		\nu+\tfrac{z}{2s}+1
	\right )
}{
	\Gamma\left (
		\tfrac{z}{2}
	\right )	
}
\cM w(z),
	\quad z\in\bC,
\end{equation*}
which implies
\begin{equation*}\label{simple-equation}
v(z-2s)-v(z)
=0.
\end{equation*}
Unfortunately, now we cannot conclude that  $v\equiv 1$, since homogeneous distributions are involved. In any case, let us write, formally,
\begin{equation*}
\widetilde w_\nu(z):=\mathcal Mw(z)=2^{z}
\dfrac{
	\Gamma\left (
		\tfrac{z}{2}
	\right )	
}
{
	\Gamma\left (
		\tfrac{n-z}{2}
	\right )\Gamma\left (
		\nu+\tfrac{z}{2s}+1
	\right )
}v_*(z)
\end{equation*}
Again, we split
\begin{equation*}\label{eq:60}
 \widetilde w_\nu(z)=\Lambda^*_s(z)\mathcal W_\nu(z),
\end{equation*}
where we have defined the multiplier $\Lambda^*_s(\lambda)$ by
\begin{equation}\label{eq:mult-adj}
\Lambda^*_s(z)
=2^{z}
\frac{
	\Gamma(\frac{z}{2})
}{
	\Gamma(\frac{n}{2}-\frac{z}{2})
}
2^{-\frac{z}{s}}
\frac{
	\Gamma(\frac{n}{2}-\frac{z}{2s})
}{
	\Gamma(\frac{z}{2s})
}.
\end{equation}
and the function
\begin{equation}\label{function-WWW}
\mathcal W_\nu(z):=2^{\frac{z}{s}}\frac{\Gamma\big(\frac{z}{2s}\big)}{\Gamma\big(\frac{n}{2}-\frac{z}{2s}\big)}
        \frac{1}{\Gamma(\frac{z}{2s}+1+\nu)}v_*(z).
\end{equation}
One may try to show the counterpart of Lemma \ref{lem:U-k} for the adjoint problem. For $\nu=k$, $k=0,1,\ldots$, up to multiplicative constant, the function in \eqref{function-WWW} is the  Mellin transform of a polynomial
\begin{equation}\label{formal}
\mathcal W_k=\mathcal M \mathfrak L_k,
\end{equation}
where we have defined
\begin{equation}\label{eq:50}
\mathfrak{L} _k(r)
=\frac{2s}{\Gamma(k+\frac{n}{2})}
L_k^{(\frac{n-2}{2})}(\tfrac{r^{2s}}{4}).
\end{equation}
Unfortunately, proving \eqref{formal} needs a serious immersion in distribution theory. Instead, in Section  \ref{subsection:duality} we will use duality to calculate the eigenfunctions of \eqref{problem-adjoint}.

\subsection{The multiplier}\label{section:multiplier-adjoint}

Similarly to the arguments in Section \ref{subsection:multiplier}, we define $\Phi^*_s$ to be the functional with Mellin multiplier given by the multiplier $\Lambda^*_s$ from \eqref{eq:mult-adj}. We will show that it has a simple formula using Fourier transform. Similarly to the above, we change variable $\vartheta=\zeta^s$ and define
\begin{equation*}\label{equation100000}
w^\ddag (\rho)
=\cF^{-1}\set{\vartheta^{\frac{n}{s}-n}
	{w^\sharp}(\vartheta)}(\rho)
=\rho^{\frac{2-n}{2}}
\int_{0}^{\infty }
	J_{\frac{n-2}{2}}(\zeta \rho)
	\zeta^{\frac{n}{2}}
	\zeta^{\frac{n}{s}-n}
	\widehat{w}(\zeta^{\frac1s})
\,d\zeta,
\end{equation*}
where $w^\sharp(\vartheta)=\widehat{w}(\vartheta^{\frac1s})$ as in \eqref{f-sharp}.
From here we can understand $w\mapsto w^\ddag$ as an integral operator
\begin{equation}\label{f-ddag}\begin{split}
w^\ddag(\rho)
&=
\int_0^\infty
	A^\ddag_s(\rho,\rho_1)
	w(\rho_1)
\,d\rho_1,
\end{split}
\end{equation}
where 
\begin{equation}\label{H-ddag}
A^\ddag_s(\rho,\rho_1)
=\rho^{\frac{2-n}{2}}
\rho_1^{\frac{n}{2}}
\int_0^\infty
	J_{\frac{n-2}{2}}(\zeta \rho)
	J_{\frac{n-2}{2}}(\zeta^{1/s}\rho_1)
	\zeta^{-\frac{n}{2}-\frac{n-2}{2s}+\frac{n}{s}}
\,d\zeta.
\end{equation}
Although \eqref{H-ddag} does not converge, we provide a formal interpretation,
using the well known formula
\begin{equation*}
\left(\frac{1}{\rho}\frac{d}{d\rho} \right)^m [\rho^\alpha J_\alpha(\rho) ]=\rho^{\alpha-m} J_{\alpha-m}(\rho),
\end{equation*}
for $m\in \bN$, which leads in our case to
\begin{equation*}
J_\frac{n-2}{2}(\zeta^{1/s}\rho_1)=\frac{1}{\rho_1^{\frac{n-2}{2}}}\left(\frac{1}{\rho_1} \frac{d}{d\rho_1}\right)^{m}
\left[ \rho_1^{\frac{n-2}{2}+m}J_{\frac{n-2}{2}+m}(\zeta^{1/s}\rho_1)\right]\frac{1}{\zeta^{m/s}},
\end{equation*}
so that the redefined kernel
\begin{equation*}
A^\ddag_s(\rho,\rho_1)
:=\rho^{\frac{2-n}{2}}
\rho_1
\left(\frac{1}{\rho_1} \frac{d}{d\rho_1}\right)^{m}
\left[ \rho_1^{\frac{n-2}{2}+m}\int_0^\infty
J_{\frac{n-2}{2}}(\zeta\rho) J_{\frac{n-2}{2}+m}(\zeta^{1/s}\rho_1)\zeta^{-\frac{n}{2}-\frac{n-2}{2s}+\frac{n}{s}-\frac{m}{s}}
\,d\zeta\right]
\end{equation*}
makes sense for $m\in\mathbb N$ large enough.

\medskip

The following Lemma states that $\Phi_s^*$ also encodes the change of variable \eqref{change}:

\begin{lemma}\label{lem:mult-FM-adjoint}
It holds
\begin{equation*}
(\Phi_s^*)^{-1}w(r)
=w^\ddag(r^s).
\end{equation*}
\end{lemma}

\begin{proof}
Using \eqref{eq:M-trans}, \eqref{eq:M-rescale} as well as \eqref{eq:MF-M} in the form
\begin{equation*}\label{eq:MF-M-1-adj}
\cM\set{\cF^{-1}u(r)}(z)
=2^{-\frac{n}{2}}
	2^z
	\frac{
		\Gamma(\frac{z}{2})
	}{
		\Gamma(\frac{n-z}{2})
	}
	\cM u(n-z),
	\quad \text{ for } \quad
0<\re z<\tfrac{n+1}{2},
\end{equation*}
and
\begin{equation*}\label{eq:MF-M-2-adj}
\cM\set{\cF^{-1}u(r)}(\tfrac{z}{s})
=2^{-\frac{n}{2}}
	2^{\frac{z}{s}}
	\frac{
		\Gamma(\frac{z}{2s})
	}{
		\Gamma(\frac{n}{2}-\frac{z}{2s})
	}
	\cM u(n-\tfrac{z}{s}),
	\quad \text{ for } \quad
0<\re \tfrac{z}{s}<\tfrac{n+1}{2},
\end{equation*}
we compute
\begin{align*}
\cM\set{
	(\Phi_s^*)^{-1}w(r)
}(z)
&=
	2^{\frac{z}{s}}
	\frac{
		\Gamma(\frac{z}{2s})	
	}{
		\Gamma(\frac{n}{2}-\frac{z}{2s})
	}
	\cdot
	2^{-z}
	\frac{
		\Gamma(\frac{n}{2}-\frac{z}{2})
	}{
		\Gamma(\frac{z}{2})	
	}
	\cM w(z)\\
&=2^{-\frac{n}{2}}
	2^{\frac{z}{s}}
	\frac{
		\Gamma(\frac{z}{2s})	
	}{
		\Gamma(\frac{n}{2}-\frac{z}{2s})
	}
	\cM\set{
		\widehat{w}(\zeta)
	}(n-z)\\
&=2^{-\frac{n}{2}}
	2^{\frac{z}{s}}
	\frac{
		\Gamma(\frac{z}{2s})	
	}{
		\Gamma(\frac{n}{2}-\frac{z}{2s})
	}
	\frac{1}{s}
	\cM\set{
		\widehat{w}(\zeta^{\frac1s})
	}\left (
		\tfrac{n-z}{s}
	\right )\\
&=\frac{1}{s}
	\cdot
	2^{-\frac{n}{2}}
	2^{\frac{z}{s}}
	\frac{
		\Gamma(\frac{z}{2s})	
	}{
		\Gamma(\frac{n}{2}-\frac{z}{2s})
	}
	\cM\set{
		\zeta^{\frac{n}{s}-n}
		\widehat{w}(\zeta^{\frac1s})
	}\left (
		n
		-\tfrac{z}{s}
	\right )\\
&=\frac{1}{s}
	\cM\set{
		\cF^{-1}\set{
			\zeta^{\frac{n}{s}-n}
			\widehat{w}(\zeta^{\frac1s})
		}(r)
	}\left (
		\tfrac{z}{s}
	\right )\\
&=
	\cM\set{
		\cF^{-1}\set{
			\zeta^{\frac{n}{s}-n}
			\widehat{w}(\zeta^{\frac1s})
		}(r^s)
	}(z).
\end{align*}
Since these equalities are valid on the strip $0<\re z<\tfrac{n+1}{2}s$, \Cref{thm:M-inv} yields the result.
\end{proof}

As in \Cref{lem:Ls-L1}, the transformation $w\mapsto w^\ddag$ relates $L_s^*$ to $L_1^*$. Function spaces will be detailed below.
\begin{lemma}\label{lem:Ls-L1-adj}
If $w$ solves $L_s^*w=g$, then $L_1^* w^{\ddag}=g^{\ddag}$.
\end{lemma}

\begin{proof}
Similarly to \Cref{lem:Ls-L1}, in the radial Fourier variable $\zeta=|\xi|$ we have
\[
\zeta^{2s}\widehat{w}(\zeta)
-\frac{1}{2s}
\left(
	n\widehat{w}(\zeta)+\zeta\partial_\zeta \widehat{w}(\zeta)
\right)
=\widehat{g}(\zeta),
	\qquad \zeta>0.
\]
Putting $\vartheta=\zeta^s$, we have $\zeta\partial_{\zeta}=s\vartheta\partial_{\vartheta}$. Now $\widehat{w^\ddag}(\vartheta)=\vartheta^{\frac{n}{s}-n}\widehat{w}(\vartheta^{\frac1s})$ satisfies
\[
(n+\vartheta\partial_\vartheta)\widehat{w^\ddag}(\vartheta)
=\vartheta^{\frac{n}{s}-n}
\left(
	\frac{n}{s}+\vartheta\partial_\vartheta
\right)
\widehat{w}(\vartheta^{\frac1s}),
\]
and we conclude by multiplying both sides by $\vartheta^{\frac{n}{s}-n}$ and taking the inverse Fourier transform.
\end{proof}

\subsection{The scalar product}\label{subsection:dual-space}

We define the Hilbert space $\mathds L^2_\ddag$ by the scalar product
\begin{equation*}
\begin{split}
	\angles{w_1,w_2}_\ddag
	&=\int_0^\infty
		(\Phi_s^*)^{-1}w_1(r)
\overline{		(\Phi_s^*)^{-1}w_2(r)}
		\,e^{-\frac{r^{2s}}{4}}
		r^{sn-1}
	\,dr.
\end{split}
\end{equation*}
Thanks to Lemma \ref{lem:mult-FM-adjoint}, we can rewrite it as follows:
\begin{equation*}
\langle w_1, w_2\rangle_{\ddag}
=\int_0^\infty
	w_1^\ddag(r^s) \overline{w_1^\ddag(r^s)}
	e^{-\frac{r^{2s}}{4}}r^{sn-1}\,dr.
\end{equation*}
We can also give a nice integral representation (although it needs to be interpreted distributionally (recall Section \ref{section:multiplier-adjoint}):

\begin{prop}\label{prop:product-adjoint} It holds
\begin{equation}\label{formula-product-G}
\begin{split}
\langle w_1, w_2\rangle_\ddag=\int_0^\infty \int_0 ^\infty G^\ddag_s(r_1,r_2)w_1(r_1)w_2(r_2)\,dr_1\,dr_2
\end{split}
\end{equation}
for a kernel
\begin{equation}\label{formula-G}
\begin{split}
G_s^\ddag(r_1,r_2)= \int_{0}^{\infty }\int_{0}^{\infty }\mathcal K_s(\zeta_1,\zeta_2)(\zeta_1\zeta_2)^{\frac{n}{s}}J_{\frac{n-2}{2}}(\zeta_1^{1/s} r_1) r_1^\frac{n}{2}\zeta_1 ^{\frac{2-n}{2s}} J_{\frac{n-2}{2}}(\zeta_2^{1/s} r_2) r_2^\frac{n}{2}\zeta_2 ^{\frac{2-n}{2s}} \,d\zeta_1\,d\zeta_2.
\end{split}
\end{equation}
Here
\begin{equation}\label{kernel-product}
\mathcal K_s(\zeta_1,\zeta_2)=\frac{1}{2}(\zeta_1\zeta_2)^{-\frac{n}{2}}e^{-(\zeta_1^2+\zeta_2^2)}I_{\frac{n-2}{2}}\left(2\zeta_1\zeta_2\right).
\end{equation}
\end{prop}

\begin{proof}
We first use  formula \eqref{f-ddag} to write, up to a constant factor,
\begin{equation*}\label{G-product}
\begin{split}
\langle w_1, w_2\rangle_\ddag
&=\int_0^\infty \int_0 ^\infty G^\ddagger_s(\rho_1,\rho_2)w_1(\rho_1)w_2(\rho_2)\,d\rho_1\,d\rho_2
\end{split}
\end{equation*}
for
\begin{equation*}
G^\ddag_s(\rho_1,\rho_2):=\int_0^\infty A^\ddag_s(\rho,\rho_1)A^\ddag_s(\rho,\rho_2)e^{-\frac{\rho^{2}}{4}}\rho^{n-1}\,d\rho,
\end{equation*}
We define
 \begin{equation*}
\mathcal K_s(\zeta_1,\zeta_2):=(\zeta_1\zeta_2)^{-\frac{n}{2}}\int_0^\infty J_{\frac{n-2}{2}}(\zeta_1 \rho)J_{\frac{n-2}{2}}(\zeta_2 r)e^{-\frac{\rho^{2}}{4}}\rho\,d\rho.
\end{equation*}
Then claim  \eqref{formula-G} follows directly from expression \eqref{H-ddag}.

 Now, to obtain a more explicit formula for $\mathcal K_s$ we use
the following formula from \cite[Section 13.31]{Watson}
\begin{equation}\label{formula-watson}
\int_0^\infty e^{-p^2 \rho^2} J_\nu(\zeta_1 \rho)J_\nu(\zeta_2 \rho)\rho\,d\rho=\frac{1}{2p^2} e^{-\frac{\zeta_1^2+\zeta_2^2}{4p^2}}I_\nu\left(\frac{\zeta_1\zeta_2}{2p^2}\right),\quad \re(\nu)>-1,\, |\argg\, p|<\frac{\pi}{4},
\end{equation}
which yields \eqref{kernel-product}.
\end{proof}

\begin{prop}
In the limit $s\to 1$,
\begin{equation*}
G_1^\ddag(r,r_2)=\tfrac{1}{2}e^{-\frac{r^{2}}{4}}r^{\frac{n}{2}-1}r_2^{\frac{n}{2}}\delta_{r=r_2}
\end{equation*}
as a function of $r$.
This shows, in particular, that one recovers the usual  scalar product in $L^2_r(\mathbb R^n,e^{-|x|^2/4})$ from the expression \eqref{formula-product-G}.
\end{prop}

\begin{proof} Set $s=1$.
Let us integrate with respect to the variable $d\zeta_1$ in formula \eqref{formula-G}. For this, note that
\begin{equation*}
\int_{0}^{\infty }e^{-\zeta_1^2}I_{\frac{n-2}{2}}\left(2\zeta_1\zeta_2\right)J_{\frac{n-2}{2}}(\zeta_1 r_1) \zeta_1  \,d\zeta_1
=\tfrac{1}{2} e^{\zeta_2^2}e^{-\frac{r_1^2}{4}}J_{\frac{n-2}{2}}(\zeta_2 r_1).
\end{equation*}
The above formula follows from \eqref{formula-watson}, just taking into account that $I_\alpha(z)=i^\alpha J_\alpha(iz)$. Substituting into
\eqref{formula-G} we obtain
\begin{equation*}
G^\ddag_s(r_1,r_2)=\tfrac{1}{2}r_1^{\frac{n}{2}}r_2^{\frac{n}{2}}e^{-\frac{r_1^2}{4}}\int_0^\infty J_{\frac{n-2}{2}}(\zeta_2 r_2)J_{\frac{n-2}{2}}(\zeta_2 r_1)\zeta_2\,d\zeta_2.
\end{equation*}
The result follows by the orthogonality property of the Bessel functions, which is
\begin{equation*}
\int_0^\infty  J_{\alpha}(\zeta r_2)J_{\alpha}(\zeta r_1)\zeta\,d\zeta=\tfrac{1}{r_1}\delta_{r_1=r_2}.
\end{equation*}
\end{proof}

We can give the following (not-sharp) estimate which tells us that $\mathds L^2_\ddag$ can be related to more standard Sobolev spaces:

\begin{lemma}
Let $h\in |\cdot|^{2} L^1(\R_+)$.
Then
\begin{equation*}
\|w\|^2_{\ddag}\leq C\int_0^\infty z^{s-1+n-ns}\frac{1}{h(z^s)}\widehat w^2(z)\,dz.
\end{equation*}
\end{lemma}

A typical choice for $h$ is $h(\zeta)=\zeta^{1+\eps}(1+\zeta)^{-2\eps}$, for some $\varepsilon>0$.

\begin{proof}
From \eqref{f-ddag} and the discussion in Proposition \ref{prop:product-adjoint} we also have
\begin{equation*}\label{equation301}
\langle w_1,w_2\rangle_\ddag=\int_0^\infty\int_0^\infty  \mathcal K_s(\zeta_1,\zeta_2) \zeta_1^{\frac{n}{s}-n}
	\widehat{w}_1(\zeta_1^{\frac1s})\zeta_2^{\frac{n}{s}-n}
	\widehat{w}_2(\zeta_2^{\frac1s}) \,d\zeta_1 \, d\zeta_2.
\end{equation*}
Let us introduce the auxiliary function $h$ in the statement of the Lemma and use Cauchy-Schwarz to estimate
\begin{equation}\label{equation20}
\begin{split}
\|w\|^2_\ddag&\leq \left(\int_0^\infty\int_0^\infty  \mathcal K_s^2(\zeta_1,\zeta_2)
h(\zeta_1)h(\zeta_2)\,d\zeta_1 \,d\zeta_2\right)^{1/2}\\
&\quad\cdot\left(\int_0^\infty\int_0^\infty  \frac{1}{h(\zeta_1)h(\zeta_2)}\left[\zeta_1^{\frac{n}{s}-n}\widehat{w}(\zeta_1^{\frac1s})\zeta_2^{\frac{n}{s}-n}\widehat{w}(\zeta_2^{\frac1s})\right]^2
\,d\zeta_1 \,d\zeta_2\right)^{1/2}.
\end{split}
\end{equation}
Taking into account that
\begin{equation*}
I_{\frac{n-2}{2}}(z)\sim \frac{1}{\sqrt{z}} e^{z}\quad \text{as} \quad z\to\infty,
\end{equation*}
while
\begin{equation*}
I_{\frac{n-2}{2}}(z)\sim z^{\frac{n-2}{2}}\quad \text{as} \quad z\to 0,
\end{equation*}
using formula \eqref{kernel-product} we can show that $\mathcal{K}_s^2(\zeta_1,\zeta_2)\leq C(\zeta_1\zeta_2)^{-2}$, thus
\begin{equation*}
\left(\int_0^\infty\int_0^\infty  \mathcal K_s^2(\zeta_1,\zeta_2)
h(\zeta_1)h(\zeta_2)\,d\zeta_1 \,d\zeta_2\right) \leq C.
\end{equation*}
We conclude from \eqref{equation20} that
\begin{equation*}
\|w\|_{\ddag}^2\leq C\left(\int_0^\infty \frac{1}{h^2(\zeta)}\left[\widehat{w}(\zeta^{\frac1s})\zeta^{\frac{n}{s}-n}\right]^2\,d\zeta\right).
\end{equation*}
Changing variable $z=\zeta^{\frac1s}$ we arrive to the desired conclusion.
\end{proof}

\subsection{Duality}\label{subsection:duality}
We define the duality pairing of $\mathds L^2_\dag$ and $\mathds L^2_\ddag$ by
\begin{equation}\label{eq:duality-Phi}
(u,w)_*
=
\int_0^\infty
	\Phi_s^{-1} u(r) (\Phi_s^*)^{-1} w(r)
\,dr.
\end{equation}

We will show that this pairing is just the usual $L_r^2(\mathbb R^n)$ product (up to multiplicative constant):

\begin{prop} Suppose $u$ and $w$ are functions such  that $\cM u(z)$ and $\cM w(n-z)$ are simultaneously well-defined in a strip around $\re z=\sigma$. Then
\begin{equation*}
(u,w)_*=\frac{2^{\frac{n}{s}-n}}{\cH^{n-1}(\bS^{n-1})}
\int_{\mathbb R^n}uw\,dx.
\end{equation*}
\end{prop}

\begin{proof}
By the isometry property \eqref{eq:M-isom-gen} of the Mellin transform,
\begin{equation*}\begin{split}
(u,w)_{L^2(\mathbb R^n)}
&=
\int_0^\infty
	u w r^{n-1}
\,dr\\
&=\int_{-\infty}^\infty
	\cM u(\sigma+\lambda i)
	\cM w(n-\sigma-\lambda i)
\,d\lambda
=-i
\int_{\sigma-i\infty}^{\sigma+i\infty}
	\cM u(z)
	\cM w(n-z)
\,dz.
\end{split}\end{equation*}
We observe from \eqref{multiplier-Lambda} and \eqref{eq:mult-adj} that, for any $z=\sigma+i\lambda \in \bC$, it holds that
\begin{equation*}
\Lambda_s^*(n-z)\Lambda_s(z)
=2^{n-\frac{n}{s}},
\end{equation*}
so that we can write
\begin{equation*}\begin{split}
2^{-n+\frac{n}{s}}(u,w)_{L^2(\mathbb R^n)}
&=-i
\int_{\sigma-i\infty}^{\sigma+i\infty}
	\Lambda_s(z)^{-1}
	\cM u(z)
	\cdot
	\Lambda_s^*(n-z)^{-1}
	\cM w(n-z)
\,dz\\
&=-i
\int_{\sigma-i\infty}^{\sigma+i\infty}
	\cM\set{\Phi_s^{-1} u}(z)
	\cdot
	\cM\set{(\Phi_s^*)^{-1}w}(n-z)
\,dz\\
&=
\int_0^\infty
	\Phi_s^{-1} u(r)
	(\Phi_s^*)^{-1} w(r)
\,dr,
\end{split}\end{equation*}
as desired.
\end{proof}

As a consequence, one obtains that the eigenfunctions for \eqref{problem-adjoint} are those dual to the eigenfunctions from Theorem \ref{thm1}. In particular, this recover our formal guess \eqref{formal}.

\begin{cor}
The eigenfunctions  for \eqref{problem-adjoint} in the space $\mathds L^2_\ddag$ are given by
\begin{equation*}
\omega_k^{(s)}(r)=\Phi^*_s \mathfrak L_k(r),\quad k\in\mathbb N,
\end{equation*}
with eigenvalue $\nu_k=k$, $k\in\mathbb N$. In particular, we have found two dual basis of eigenfunctions: for $j\neq k$,
\begin{equation*}
(e_j, \omega_k)_*=0.
\end{equation*}

\end{cor}

\begin{proof}
Note that,  using \eqref{eq:Hermite-M-inv}, \eqref{eq:50}, \eqref{eq:duality-Phi} and \eqref{Laguerre-orthogonality}, 
\begin{equation*}\begin{split}
(e_j,\omega_k)_*
&=\int_0^\infty
	\mathcal L_j(r)
	\cdot
	\mathfrak L_k(r)
\,dr\\
&=\int_0^\infty
	2s\,j!
	\left (
		\tfrac{r^{2s}}{4}
	\right )^{\frac{n}{2}-\frac{n}{2s}}
	e^{-\frac{r^{2s}}{4}}
	L_j^{(\frac{n-2}{2})}
	\left (
		\tfrac{r^{2s}}{4}
	\right )
	\cdot
	\frac{2s}{\Gamma(k+\frac{n}{2})}\,
	L_k^{(\frac{n-2}{2})}(\tfrac{r^{2s}}{4})
	\,r^{n-1}
\,dr\\
&=\frac{2s\,j!}{\Gamma(k+\frac{n}{2})}
\int_0^\infty
	\rho^{\frac{n}{2}-\frac{n}{2s}}
	e^{-\rho}
	L_j^{(\frac{n-2}{2})}(\rho)
	\,L_k^{(\frac{n-2}{2})}(\rho)
	(4\rho)^{\frac{n}{2s}}
\frac{d\rho}{\rho}\\
&=
	\frac{
		2s\,2^{\frac{n}{s}}
	}{
		\Gamma(\frac{n}{2})
	}
\delta_{j,k}.
\end{split}\end{equation*}
\end{proof}

We finish with a quick calculation:

\begin{lemma} \label{lemma:adjoint-orthogonal}
	Any two eigenfunctions $\omega_k$, $\omega_j$, $k\neq j$,
	are orthogonal with respect to $\angles{\cdot,\cdot}_{\ddag}$.
\end{lemma}

\begin{proof} Given $j,k$ non-negative integers:
\begin{equation*}
\begin{split}
\langle \omega_j,\omega_k\rangle_{\ddag}
&=	\int_0^\infty \mathfrak L_j(r)\mathfrak L_k(r)\,e^{-\frac{r^{2s}}{4}}
		r^{sn-1}
	\,dr
	=\int_0^\infty L_j^{(\frac{n-2}{2})}(\tfrac{r^{2s}}{4})L_k^{(\frac{n-2}{2})}(\tfrac{r^{2s}}{4})\,e^{-\frac{r^{2s}}{4}}
		r^{sn-1}
	\,dr\\
&=2^n\int_0^\infty L_j^{(\frac{n-2}{2})}(\varrho)L_k^{(\frac{n-2}{2})}(\varrho)\varrho^{\frac{n-2}{2}}e^{-\varrho}\,d\varrho
=2^n\frac{
	\Gamma(k+\tfrac{n}{2})
}{
	k!\Gamma(\tfrac{n}{2})
}
\delta_{j,k},
\end{split}
\end{equation*}
by the orthogonality \eqref{Laguerre-orthogonality}
of the Laguerre polynomials $\{L^{(\alpha)}_k(\varrho)\}_k$ with respect to the weight $e^{-\varrho}\varrho^\alpha$, for $\alpha=\tfrac{n-2}{2}$.
\end{proof}

\appendix

\section{Hermite and Laguerre polynomials}\label{section:Hermite}

``Probabilist's'' Hermite polynomials are defined by
\begin{equation*}
He_k(x)=(-1)^k e^{\frac{1}{2}x^2}\frac{d^k}{dx^k}\left(e^{-\frac{1}{2}x^2}\right).
\end{equation*}
They form an orthogonal basis of the space $L^2(\mathbb R,e^{-\frac{x^2}{2}})$. It is well known that they are the eigenfunctions of the equation
\begin{equation*}
u''(x)-xu'(x)=-\lambda_k u(x),\quad x\in\mathbb R,
\end{equation*}%
with eigenvalues $\lambda_k:=k$, $k=0,1,2,\ldots$ In contrast, ``Physicist's'' Hermite polynomials are given by
\begin{equation*}
H_k(x)=(-1)^k e^{x^2}\frac{d^k}{dx^k}e^{-x^2}.
\end{equation*}

Here we will use a different normalization, given by
\begin{equation*}
H{e_k}^{[\alpha]}(x)=\alpha^{\frac{k}{2}}He_{k}\left(\frac{x}{\sqrt \alpha}\right)=\left(\frac{\alpha}{2}\right)^{\frac{k}{2}}H_{k}\left(\frac{x}{\sqrt{2\alpha}}\right).
\end{equation*}
For $\alpha=2$, these will be denoted by $\mathcal H_k(x)$, and they can be obtained by Rodrigues' formula
\begin{equation*}\label{Rodrigues}
\mathcal H_k(x)=(-1)^k2^k e^{\frac{1}{4}x^2}\frac{d^k}{dx^k}\left(e^{-\frac{x^2}{4}}\right).
\end{equation*}
We remark that the $\mathcal H_k(x)$ are an orthogonal basis of $L^2(\mathbb R,e^{-\frac{x^2}{4}})$.
We are interested in this particular normalization because the $\mathcal H(x)$ are
the eigenfunctions of
\begin{equation*}
u''(x)-\frac{1}{2}xu'(x)=-\frac{1}{2}\lambda u(x),\quad x\in\mathbb R,
\end{equation*}%
with eigenvalues $\lambda_k:=k$, $k=0,1,2,\ldots$ Note also that if we are interested in even eigenfunctions, functions, we need to restrict to the even integers.

Hermite polynomials are related to Laguerre polynomials, by the formulas
\begin{equation}\label{Hermite-Laguerre}
\begin{split}
H_{2k}(x) &=(-1)^{k}2^{2k}k!L_{k}^{(-1/2)}(x^{2}), \\
H_{2k+1}(x) &=(-1)^{k}2^{2k+1}k!xL_{k}^{(1/2)}(x^{2}),
\end{split}
\end{equation}%
where
$L_k^{(\alpha)}$ are the Laguerre polynomials, defined by
\begin{equation}\label{Laguerre}
L_{k}^{(\alpha )}(x) =\frac{\Gamma (k+\alpha +1)}{\Gamma (k+1)\Gamma
(\alpha +1)}M(-k,\alpha +1;x),
\end{equation}
and $M$ is Kummer's (confluent hypergeometric) function. An important property of the Laguerre polynomials is that they satisfy the orthogonality relation
\begin{equation}\label{Laguerre-orthogonality}
\int_{0}^{\infty }L_{k}^{(\alpha )}(x)L_{m}^{(\alpha )}(x)e^{-x}x^{\alpha
}\,dx=\frac{\Gamma (k+\alpha +1)}{\Gamma (k+1)\Gamma (\alpha +1)}\delta _{m,k}.
\end{equation}
In addition, they are a complete basis of the corresponding $L^2(0,\infty)$ space.

We will also need the following integration formula from \cite[6.643 4.]{Gradshteyn-Ryzhik}
\begin{equation*}\label{eq:J-exp}
\int_0^\infty
	\zeta^{k+\frac{n-2}{4}}
	e^{-\alpha \zeta}
	J_{\frac{n-2}{2}}(r\sqrt{\zeta})
\,d\zeta
=k! \left(\frac{r}{2}\right)^{\frac{n-2}{2}}
	e^{-\frac{r^2}{4\alpha}}
	\alpha^{-k-\frac{n-2}{2}-1}
	L_k^{(\frac{n-2}{2})}\left(
		\frac{r^2}{4\alpha}
	\right),
\end{equation*}
valid for $k+\frac{n-2}{2}>-1$. In particular,
\begin{equation}\label{eq:J-exp-2}
r^{-\frac{n-2}{2}}
\int_0^\infty
	\zeta^{2k+\frac{n}{2}}
	e^{-\zeta^2}
	J_{\frac{n-2}{2}}(r\zeta)
\,d\zeta
=2^{-\frac{n}{2}}k!
	e^{-\frac{r^2}{4}}
	L_k^{(\frac{n-2}{2})}\left(
		\frac{r^2}{4}
	\right).
\end{equation}


\section{Asymptotics of the Gamma function}\label{subsection:asymptotics-Gamma}

Using the asymptotics %
\begin{equation*}\begin{split}
\log \Gamma(z+\vartheta)
&=\bigl(z+\vartheta-\tfrac12\bigr)\log(z+\vartheta)
    -(z+\vartheta)
    +\tfrac12\log(2\pi)
    +O(|z|^{-1})\\
&= \big(z+\vartheta-\tfrac{1}{2}\big)\log z-z+\tfrac{1}{2}\log (2\pi
)+O(|z|^{-1}),
\end{split}
\end{equation*}
we find, as $\left\vert z\right\vert \rightarrow \infty $,%
\begin{equation}\label{asymptotics-gamma0}
\Gamma (z+\vartheta)=\sqrt{2\pi }z^{z+\vartheta-\frac{1}{2}}e^{-z}e^{O(|z|^{-1})},
\end{equation}
which is known as Stirling's approximation for the Gamma function.

In particular, if $\vartheta$ is real and $z=\frac{\lambda }{2}i$, then%
\begin{equation*}\label{eq:Gamma-asymp-half0}
\begin{split}
\Gamma \left(\tfrac{\lambda }{2}i+\vartheta\right)
&=\sqrt{2\pi }\left( \tfrac{\left\vert \lambda
\right\vert }{2}e^{i\frac{\pi }{2}\sign\lambda}\right) ^{i\frac{\lambda }{2}+\vartheta-\frac{1}{2
}}e^{-i\frac{\lambda }{2}}e^{O(\left\vert \lambda \right\vert ^{-1})}\\
&=
    \sqrt{2\pi}\,
    e^{-\frac{\pi}{4}\left\vert \lambda \right\vert}
    \left(
        \tfrac{\left\vert \lambda\right\vert}{2}
    \right)^{\vartheta-\frac{1}{2}}
    e^{i\left(
        \frac{\lambda}{2}\log\frac{|\lambda|}{2}
        -\frac{\lambda}{2}
        +(\vartheta-\tfrac12)\frac{\pi}{2}\sign\lambda
    \right)}
    \bigl(
        1+O(\left\vert \lambda \right\vert^{-1})
    \bigr),
\end{split}
\end{equation*}
and
\begin{equation}\label{estimate-Gamma-2}
\begin{split}
\abs{\Gamma \left(\tfrac{\lambda }{2}i+\vartheta\right)}
&=
    \sqrt{2\pi}\,
    e^{-\frac{\pi}{4}\left\vert \lambda \right\vert}
    \left(
        \tfrac{\left\vert \lambda\right\vert}{2}
    \right)^{\vartheta-\frac{1}{2}}
    \bigl(
        1+O(\left\vert \lambda \right\vert^{-1})
    \bigr).
\end{split}
\end{equation}
Replacing $\lambda$ by $\lambda/s$,
\begin{equation}\label{estimate-Gamma-2s}
\begin{split}
\abs{\Gamma \left(\tfrac{\lambda }{2s}i+\vartheta\right)}
&=
    \sqrt{2\pi}\,
    e^{-\frac{\pi}{4s}\left\vert \lambda \right\vert}
    \left(
        \tfrac{\left\vert \lambda\right\vert}{2s}
    \right)^{\vartheta-\frac{1}{2}}
    \bigl(
        1+O(\left\vert \lambda \right\vert^{-1})
    \bigr).
\end{split}
\end{equation}

\bigskip

\textbf{Acknowledgements}
H. C. has received funding from the Spanish Government under Grant CEX2019-000904-S funded by MCIN/AEI/10.13039/501100011033
and PID2020-113596GB-I00, and from the Swiss National Science Foundation under Grant
PZ00P2\_202012/1. In addition, M. F. and M. G. are supported by the Spanish government, grant numbers PID2020-113596GB-I00 and PID2023-150166NB-I00.

\end{document}